\documentclass{article}
\usepackage{inputenc}
\usepackage{booktabs} % for much better looking tables
\usepackage{array} % for better arrays (eg matrices) in maths
\usepackage{paralist} % very flexible & customisable lists (eg. enumerate/itemize, etc.)
\usepackage[english]{babel}
\usepackage{verbatim} % adds environment for commenting out blocks of text & for better verbatim
\usepackage{subfig} % make it possible to include more than one captioned figure/table in a single float
\usepackage{amsmath}
\usepackage{amsfonts}
\usepackage{hyperref}
\usepackage{bbm}
\usepackage{babel}
\usepackage{amsthm}
\usepackage{amssymb}
\usepackage{mathtools}
\usepackage{eufrak}
\usepackage{bbold}
\usepackage{tikz-cd}

\usepackage{geometry} % to change the page dimensions
\makeatletter
\newtheorem*{rep@theorem}{\rep@title}
\newcommand{\newreptheorem}[2]{%
\newenvironment{rep#1}[1]{%
 \def\rep@title{#2 \ref{##1}}%
 \begin{rep@theorem}}%
 {\end{rep@theorem}}}
\makeatother

\newtheorem{thm}{Theorem}[section]
\newtheorem{conget}[thm]{Conjecture}
\newtheorem{cor}[thm]{Corollary}
\newtheorem{lem}[thm]{Lemma}
\newtheorem{prop}[thm]{Proposition}
\newtheorem{defn}[thm]{Definition}
\newtheorem{rem}[thm]{Remark}

\newtheorem*{cor*}{Corollary}

\newtheorem{theorem}{Theorem}
\newreptheorem{theorem}{Theorem}

\newreptheorem{proposition}{Proposition}

\newreptheorem{corollary}{Corollary}

\title{Multipoint Okounkov bodies}
\author{Antonio Trusiani\footnote{\href{mailto:antonio.trusiani91@gmail.com}{antonio.trusiani91@gmail.com}.}}
\date{}

\begin{document}

\maketitle

\begin{abstract}
Starting from the data of a big line bundle $L$ on a projective manifold $X$ with a choice of $N\geq 1$ different points on $X$ we provide a new construction of $N$ Okounkov bodies which encodes important geometric features of $(L\to X; p_{1},\dots,p_{N})$ such as the volume of $L$, the (moving) multipoint Seshadri constant of $L$ at $p_{1},\dots,p_{N}$, and the possibility to construct Kähler packings centered at $p_{1},\dots,p_{N}$. Toric manifolds and surfaces are examined in detail.
\end{abstract}
\vspace{5pt}
{\small \textbf{Keywords:} Okounkov body, Seshadri constant, packings problem, projective manifold, ample line bundle.}
\section{Introduction}
Okounkov in \cite{Oko96} and \cite{Oko03} found a way to associate a convex body $\Delta(L)\subset \mathbbm{R}^{n}$ to a polarized manifold $(X,L)$ where $n=\dim_{\mathbbm{C}}X$. Namely,
$$
\Delta(L):=\overline{\bigcup_{k\geq 1}\Big\{\frac{\nu^{p}(s)}{k}\, : \, s\in H^{0}(X,kL)\setminus\{0\}\Big\}}
$$
where $\nu^{p}(s)$ is the \emph{leading term exponent} at $p$ with respect to a total additive order on $\mathbbm{Z}^{n}$ and holomorphic coordinates centered at $p\in X$ (see subsection \ref{paragraph:ParticularValuations}). This convex body is now called \emph{Okounkov body}.\\
Okounkov's construction was inspired by toric geometry, indeed in the toric case, if $L_{P}$ is a torus-invariant ample line bundle, $\Delta(L_{P})$ is essentially equal to the polytope $P$.\\
The same construction works even if $L$ is a big line bundle, i.e. a line bundle such that $\mathrm{Vol}_{X}(L):=\limsup_{k\to\infty}\frac{n!}{k^{n}}\dim_{\mathbbm{C}} H^{0}(X,kL)>0$, as proved in \cite{LM09}, \cite{KKh12} (see also \cite{Bou14}) and the Okounkov body captures the volume of $L$ since
$$
\mathrm{Vol}_{X}(L)=n!\mathrm{Vol}_{\mathbbm{R}^{n}}\big(\Delta(L)\big).
$$
Moreover if $>$ is the lexicographical order then the $(n-1)-$volume of any not trivial \emph{slice} of the Okounkov body given by $\Delta(L)\cap \{x_{1}=t\}$ is related to the \emph{restricted volume} of $L-tY$ along $Y$ where $Y$ is a smooth irreducible divisor such that $Y_{\vert U_{p}}=\{z_{1}=0\}$.

Another invariant which can be encoded by the Okounkov body is the \emph{(moving) Seshadri constant} $\epsilon_{S}( ||L||;p)$ (see \cite{Dem90} in the ample case, or \cite{Nak03} for the extension to the big case). In fact, as K\"uronya-Lozovanu showed in \cite{KL15a}, \cite{KL17}, if the Okounkov body is defined using the \emph{deglex order}\footnote{$\alpha <_{deglex}\beta$ iff $\vert \alpha\vert:=\sum_{j=1}^{n}\alpha_{j} <\vert \beta \vert $ or $\vert \alpha\vert=\vert \beta \vert$ and $\alpha <_{lex} \beta$, where $<_{lex}$ is the lexicographical order}, then
$$
\epsilon_{S}(||L||; p)=\max\big\{0,\sup\{t\geq 0 \, : \, t\Sigma_{n}\subset \Delta(L)\}\big\}
$$
where $\Sigma_{n}$ is the unit $n-$simplex.\\
As observed by Witt Nyström in \cite{WN15}, one can restrict ourselves to considering the \emph{essential} Okounkov body $\Delta(L)^{\mathrm{ess}}$ to get the same characterization of the moving Seshadri constant. This last object is defined as $\Delta(L)^{\mathrm{ess}}:=\bigcup_{k\geq 1}\Delta^{k}(L)^{\mathrm{ess}}$, where $\Delta^{k}(L)=\mathrm{Conv}(\{\frac{\nu(s)}{k}\, : \, s\in H^{0}(X,kL)\setminus\{0\}\})$ and the \emph{essential} part of $\Delta^{k}(L)$ consists of its interior as subset of $\mathbbm{R}_{\geq 0}^{n}$ with its natural induced topology.\\

Seshadri constants are also defined for a collection of different points. For a nef line bundle $L$, the \emph{multipoint Seshadri constant of} $L$ \emph{at} $p_{1},\dots,p_{N}$ is given as
$$
\epsilon_{S}(L;p_{1},\dots,p_{N}):=\inf_{C}\frac{L\cdot C}{\sum_{j=1}^{N}\mathrm{mult}_{p_{j}}C}.
$$

In this paper we introduce a multipoint version of the Okounkov body. More precisely, for a fixed big line bundle $L$ on a projective manifold $X$ of dimension $n$ and $p_{1},\dots,p_{N}\in X$ different points, we construct $N$ Okounkov bodies $\Delta_{j}(L)\subset \mathbbm{R}^{n}$ for $j=1,\dots,N$.
\begin{defn}
Let $L$ be a big line bundle and let $>$ be a fixed total additive order on $\mathbbm{Z}^{n}$.
$$
\Delta_{j}(L):=\overline{\bigcup_{k\geq 1}\Big\{\frac{\nu^{p_{j}}(s)}{k}\, : \, s\in V_{k,j}\Big\}}\subset \mathbbm{R}^{n}
$$
is called \textbf{multipoint Okounkov body} of $L$ at $p_{j}$, where $V_{k,j}:=\{{s\in H^{0}(X,kL)\setminus\{0\}}\, : \, \nu^{p_{j}}(s)<\nu^{p_{i}}(s) \, \mathrm{for}\, \mathrm{any} \, i\neq j \}$ for any $k\geq 1$.
\end{defn}
We observe that the multipoint Okounkov body of $L$ at $p_{j}$ is obtained by considering all sections whose leading term in $p_{j}$ is strictly smaller than those at the other points.\\
They are convex compact sets in $\mathbbm{R}^{n}$ but, unlike the one-point case, for $N\geq 2$ it can happen that some $\Delta_{j}(L)$ are empty (Remark \ref{rem:Ex}). The definition does not depend on the order of the points.\\

Our first theorem concerns the relationship between the multipoint Okounkov bodies and the volume of the line bundle:
\begin{theorem}
\label{ThmA}\footnote{The theorem holds in the more general setting of a family of \emph{faithful valuations} $\nu^{p_{j}}:\mathcal{O}_{X,p_{j}}\setminus\{0\}\to (\mathbbm{Z}^{n},>)$ respect to a fixed total additive order $>$ on $\mathbbm{Z}^{n}$.}
Let $L$ be a big line bundle. Then
$$
n!\sum_{j=1}^{N}\mathrm{Vol}_{\mathbbm{R}^{n}}\big(\Delta_{j}(L)\big)=\mathrm{Vol}_{X}(L).
$$
\end{theorem}
Furthermore, similar to section \S $4$ in \cite{LM09}, we show that $\Delta_{j}(\cdot)$ is a numerical invariant and that there exists an open subset of the big cone containing $B_{+}(p_{j})^{C}=\{{\alpha\in N^{1}(X)_{\mathbbm{R}}}\, : \, p_{j}\notin\mathbbm{B}_{+}(\alpha)\}$ over which $\Delta_{j}(\cdot)$ can be extended continuously (see section \S $\ref{subsection:Variation}$). Recall that the points, and more in general the valuations $\nu^{p_{j}}$, are fixed.\\
Moreover when $>$ is the lexicographical order and $Y_{1},\dots,Y_{N}$ are smooth irreducible divisors such that $Y_{j\vert U_{p_{j}}}=\{z_{j,1}=0\}$, the fibers of $\Delta_{j}(L)$ are related to the restricted volumes of $L-t\sum_{i=1}^{N}Y_{i}$ along $Y_{j}$ (see section\nopagebreak \S \ref{subsection:Consequence}).\\

The multipoint Okounkov bodies can be finer invariants than the \emph{moving} multipoint Seshadri constant (a natural generalization of the multipoint Seshadri constant to big line bundles, see section $\S$ \ref{section:SeshadriConstant}) as our next Theorem shows.
\begin{theorem}
\label{ThmB}
Let $L$ be a big line bundle and let $>$ be the deglex order. Then
$$
\epsilon_{S}(||L||; p_{1},\dots,p_{N})=\max\big\{0,\xi(L;p_{1},\dots,p_{N})\big\}
$$
where $\xi(L;p_{1},\dots,p_{N}):=\sup\{t\geq 0 \, : \, t\Sigma_{n}\subset \Delta_{j}(L)^{\mathrm{ess}}\,\mathrm{for}\,\mathrm{any}\, {j=1,\dots,N}\}$.
\end{theorem}
Next we recall another interpretation of the one point Seshadri constant: $\epsilon_{S}(L;p)$ is equal to the supremum of $r$ such that there exists an holomorphic embedding $f:(B_{r}(0),\omega_{st})\to (X,L)$ with the property that $f_{*}\omega_{st}$ extends to a K\"ahler form $\omega$ with cohomology class $c_{1}(L)$ (see Theorem 5.1.22 and Proposition 5.3.17. in \cite{Laz04}). This result is a consequence of a deep analysis in symplectic geometry by McDuff-Polterovich (\cite{MP94}), where they dealt with the \emph{symplectic packings problem} (in the same spirit, Biran in \cite{Bir97} proved the symplectic analog of the Nagata's conjecture).\\
Subsequently Kaveh in \cite{Kav16} showed how the one-point Okounkov body can be used to construct a sympletic packing. Along the same lines Witt Nyström in \cite{WN15} introduced the torus-invariant domain $\Omega(L):=\mu^{-1}\big(\Delta(L)^{\mathrm{ess}}\big)$ (called \emph{Okounkov domain}) for $\mu:\mathbbm{C}^{n}\to \mathbbm{R}^{n}$, $\mu(z_{1},\dots,z_{n}):=(\lvert z_{1}\rvert^{2},\dots,\lvert z_{n}\rvert^{2})$, and showed how it approximates the polarized manifold.\\

To get a similar characterization of the multipoint Seshadri constant, we give the following definition of \emph{Kähler packing}.
\begin{defn}
We say that a finite family of $n-$dimensional K\"ahler manifolds $\{(M_{j},\eta_{j})\}_{j=1}^{N}$ packs into $(X,L)$ for $L$ ample line bundle on a $n-$dimensional projective manifold $X$ if for any family of relatively compact open set $U_{j}\Subset M_{j}$ there are a holomorphic embedding $f:\bigsqcup_{j=1}^{N}U_{j}\to X$ and a K\"ahler form $\omega$ lying in $c_{1}(L)$ such that $f_{*}\eta_{j}=\omega_{\vert f(U_{j})}$. If, in addition,
$$
\sum_{j=1}^{N}\int_{M_{j}}\eta_{j}^{n}=\int_{X}c_{1}(L)^{n}
$$
then we say that $\{(M_{j},\eta_{j})\}_{j=1}^{N}$ packs perfectly into $(X,L)$.
\end{defn}
Following \cite{WN15} we define the \emph{multipoint Okounkov domains} as the torus-invariant domains of $\mathbbm{C}^{n}$ given by $\Omega_{j}(L):=\mu^{-1}\big(\Delta_{j}(L)^{\mathrm{ess}}\big)$.
\begin{theorem}\footnote{the theorem holds even if $\nu^{p_{j}}$ is a family of faithful \emph{quasi-monomial} valuations respect to the same linearly independent vectors $\vec{\lambda}_{1},\dots,\vec{\lambda}_{n}\in \mathbbm{N}^{n}$.}
\label{ThmC}
Let $L$ be an ample line bundle. Then $\{(\Omega_{j}(L),\omega_{st})\}_{j=1,\dots,N}$ packs perfectly into $(X,L)$.
\end{theorem}
Note that for big line bundles a similar theorem holds, given a slightly different definition of \emph{packings} (see section \ref{subsection:BigCase}).\\
As a consequence of Theorems \ref{ThmB}, \ref{ThmC} (see Corollary \ref{cor:EcklGen}),
$$
\sqrt{\epsilon_{S}(||L||; p_{1},\dots,p_{N})}=\max\Big\{0,\sup\big\{r>0\, : \, \{(B_{r}(0),\omega_{std}\big)\}_{j=1}^{N}\, \mbox{packs into} \, (X,L)\,\big\}\Big\}.
%\max\big\{0,\sup\{r>0 \, : \, B_{r}(0)\subset D_{j}(L) \, \forall j=1,\dots,N\}\big\}.
$$
This result was known in dimension $2$ by the work of Eckl (\cite{Eckl17}), and for $N=1$ by \cite{WN15}.\\

Moving to particular cases, for toric manifolds we prove that, chosen torus-fixed points and the deglex order, the multipoint Okounkov bodies can be obtained \emph{subdiving} the polytope (Theorem \ref{thm:ToricSubdivision}). If we consider all torus-invariant points the subdivision is \emph{barycentric} (Corollary \ref{cor:Barycentric}). As a consequence we get that the multipoint Seshadri constant of $N$ torus-fixed points is in $\frac{1}{2}\mathbbm{N}$ (Corollary \ref{cor:HalfSesh}).

Finally in the surface case, we extend the result in \cite{KLM12} showing, for the lexicographical order, the polyhedrality of $\Delta_{j}(L)$ (Theorem \ref{thm:Surfaces}). Moreover for $\mathcal{O}_{\mathbbm{P}^{2}}(1)$ over $\mathbbm{P}^{2}$ we completely characterize $\Delta_{j}(\mathcal{O}_{\mathbbm{P}^{2}}(1))$ in function of $\epsilon_{S}(\mathcal{O}_{\mathbbm{P}^{2}}(1);N)$ obtaining an explicit formula for the restricted volume of $\mu^{*}\mathcal{O}_{\mathbbm{P}^{2}}(1)-t\mathbbm{E}$ for $t\in\mathbbm{Q}$ where $\mu:\tilde{X}\to X$ is the blow-up at $N$ very general points and $\mathbbm{E}:=\sum_{j=1}^{N}E_{j}$ is the sum of the exceptional divisors (Theorem \ref{thm:ProjPlane}). As a consequence we independently get a result present in \cite{DKMS15}: the ray $\mu^{*}\mathcal{O}_{\mathbbm{P}^{2}}(1)-t\mathbbm{E}$ meets at most two \emph{Zariski chambers}.
\subsection{Organization}
Section \ref{sec:Preliminaries} contains some preliminary facts on singular metrics, base loci of divisors and Okounkov bodies.\\
In section \ref{sec:MOB} we develop the theory of multipoint Okounkov bodies: the goal is to generalize some results in \cite{LM09} for $N\geq1$. We prove here Theorem \ref{ThmA}.\\
Section \ref{section:Packings} is dedicated to show Theorem \ref{ThmC}.\\
In section \ref{section:SeshadriConstant} we introduce the notion of \emph{moving} multipoint Seshadri constants. Moreover we prove Theorem \ref{ThmB}, connecting the moving multipoint Seshadri constant in a more analytical language in the spirit of \cite{Dem90}, and deduce the connection between the moving multipoint Seshadri constant and K\"ahler packings.\\
The last section \ref{section:ParticularCases} deals with the two aforementioned particular cases: toric manifolds and surfaces.
\subsection{Related works}
In addition to the already mentioned papers of Witt Nyström (\cite{WN15}), Eckl (\cite{Eckl17}), and K\"urona-Lozovanu (\cite{KL15a}, \cite{KL17}), during the final revision of this paper the work of Shin \cite{Sh17} appeared as a preprint. Starting from the same data of a big divisor over a projective manifold of dimension $n$ and the choice of $r$ different points, he gave a construction of an \emph{extended Okounkov Body} $\Delta_{Y_{\cdot}^{1},\dots,Y_{\cdot}^{r}}(D)\subset\mathbbm{R}^{rn}$ from a valuation associated to a family of admissible or infinitesimal flags $Y_{\cdot}^{1},\dots,Y_{\cdot}^{r}$. In the ample case thanks to the Serre's vanishing Theorem, the multipoint Okounkov bodies can be recovered from the extended Okounkov body as projections after suitable subdivisions. Precisely, with the notation given in \cite{Sh17}, we get
$$
F(\Delta_{j}(D))=\pi_{j}\Big(\Delta_{Y_{\cdot}^{1},\dots,Y_{\cdot}^{r}}(D)\cap H_{1,j}\cap\cdots\cap H_{j-1,j}\cap H_{j+1,j}\cap\cdots \cap H_{r,j}\Big)
$$
where $\pi_{j}:\mathbbm{R}^{rn}\to \mathbbm{R}^{n}, \pi_{j}(\vec{x}_{1},\dots,\vec{x}_{r}):=\vec{x}_{j}$, $H_{i,j}:=\{(\vec{x}_{1},\dots,\vec{x}_{r}){\in\mathbbm{R}^{rn}}\, : \, x_{i,1}\geq x_{j,1}\}$ and $F:\mathbbm{R}^{n}\to\mathbbm{R}^{n}, F(y_{1},\dots,y_{n}):=(\lvert y\rvert,y_{1},\dots,y_{n-1})$. Note that $x_{i,1}$ means the first component of the vector $\vec{x_{i}}$ while $|y|=y_{1}+\dots+y_{n}$. The same equality holds if $L:=\mathcal{O}_{X}(D)$ is big and $c_{1}(L)\in\mathrm{Supp}(\Gamma_{j}(X))^{\circ}$ (see section \ref{subsection:Variation}).
\subsection{Acknowledgements} I want to thank David Witt Nyström and Stefano Trapani for proposing the project to me and for their suggestions and comments. It is also a pleasure to thank Bo Berndtsson for reviewing this article, Valentino Tosatti for his interesting comments and Christian Schultes for pointing out a mistake in a previous version.
\section{Preliminaries}
\label{sec:Preliminaries}
\subsection{Singular metrics and (currents of) curvature}
\label{ssec:SingMetr}
Let $L$ be a holomorphic line bundle over a projective manifold $X$. A smooth (hermitian) metric $\varphi$ is the collection of an open cover $\{U_{j}\}_{j\in J}$ of $X$ and of smooth functions $\varphi_{j}\in\mathcal{C}^{\infty}(U_{j})$ such that on each not-empty intersection $U_{i}\cap U_{j}$ one has $\varphi_{i}=\varphi_{j}+\ln|g_{i,j}|^{2}$ where $g_{i,j}$ are the transition function defining the line bundle $L$. Note that if $s_{j}$ are nowhere zero local sections with respect to which the transition function are calculated then $|s_{j}|=e^{-\varphi_{j}}$. The curvature of a smooth metric $\varphi$ is given on each open $U_{j}$ by $dd^{c}\varphi_{j}$ where $d^{c}=\frac{i}{4\pi}(\partial-\bar{\partial})$ so that $dd^{c}=\frac{i}{2\pi}\partial\bar{\partial}$. It is a global $(1,1)-$form on $X$, so for convenience we use the notation $dd^{c}\varphi$. The metric is called \emph{positive} if the $(1,1)-$form $dd^{c}\varphi$ is a Kähler form, i.e. if the functions $\varphi_{j}$ are strictly plurisubharmonic. By the well-known Kodaira Embedding Theorem, a line bundle admits a positive metric iff it is ample.\\
Demailly in \cite{Dem90} introduced a weaker notion of metric: a (hermitian) singular metric $\varphi$ is given by a collection of data as before but with the weaker condition that $\varphi_{j}\in L^{1}_{loc}(U_{j})$. If the functions $\varphi_{j}$ are also plurisubharmonic, then we say that $\varphi$ is a singular positive metric. Note that $dd^{c}\varphi$ exists in the weak sense, indeed it is a closed positive $(1,1)-$current (we will call it the current of curvature of the metric $\varphi$). We say that $dd^{c}\varphi$ is a Kähler current if it dominates some Kähler form $\omega$. By Proposition $4.2.$ in \cite{Dem90} a line bundle is big iff it admits a singular positive metric whose current of curvature is a Kähler current.\\
In this paper we will often work with $\mathbbm{R}-$line bundles, i.e. formal linear combinations of line bundles. Moreover since we will only consider projective manifolds, we will often identify an $\mathbbm{R}-$line bundle as a class of $\mathbbm{R}-$divisors modulo linear equivalence and its first Chern class as a class of $\mathbbm{R}-$divisors modulo numerical equivalence.
\subsection{Base loci}
\label{ssec:BaseLoci}
We recall here the construction of the \emph{base loci} (see \cite{ELMNP06}).\\
Given a $\mathbbm{Q}-$divisor $D$, let $\mathbbm{B}(D):=\bigcap_{k\geq 1}\mathrm{Bs}(kD)$ be the \emph{stable base locus} of $D$ where $\mathrm{Bs}(kD)$ is the base locus of the linear system $|kD|$. The base loci $\mathbbm{B}_{+}(D):=\bigcap_{A}\mathbbm{B}(D-A)$ and $\mathbbm{B}_{-}(D):=\bigcup_{A}\mathbbm{B}(D+A)$, where $A$ varies among all ample $\mathbbm{Q}-$divisors, are called respectively \emph{augmented} and \emph{restricted base locus} of $D$. They are invariant under rescaling and $\mathbbm{B}_{-}(D)\subset \mathbbm{B}(D)\subset \mathbbm{B}_{+}(D)$. Moreover as described in a work of Nakamaye, \cite{Nak03}, the restricted and the augmented base loci are numerical invariants and can be extended to the Neron-Severi space (for a real class it is enough to consider only ample $\mathbbm{R}-$divisors $A$ such that $D\pm A$ is a $\mathbbm{Q}-$divisor). The stable base loci do not, see Example $1.1.$ in \cite{ELMNP06}, although by Proposition $1.2.6.$ in \cite{ELMNP06} the subset where the augmented and restricted base loci are equal is open and dense in the Neron-Severi space $\mathrm{N}^{1}(X)_{\mathbbm{R}}$.\\
Thanks to the numerical invariance of the restricted and augmented base loci, we will often talk of restricted and/or augmented base loci of a $\mathbbm{R}-$line bundle $L$. Moreover the restricted base locus can be thought as a measure of the \emph{nefness} since $D$ is nef iff $\mathbbm{B}_{-}(D)=\emptyset$, while the augmented base locus can be thought as a measure of the \emph{ampleness} since $D$ is ample iff $\mathbbm{B}_{+}(D)=\emptyset$. Moreover $\mathbbm{B}_{-}(D)=X$ iff $D$ is not pseudoeffective while $\mathbbm{B}_{+}(D)=X$ iff $D$ is not big.
\subsection{Additive Semigroups and their Okounkov bodies}
\label{ssec:Semigroups}
We briefly recall some notions about the theory of the Okounkov bodies constructed from additive semigroups (the main references are \cite{KKh12} and \cite{Bou14}, see also \cite{Kho93}).\\
Let $S\subset \mathbbm{Z}^{n+1}$ be an additive subsemigroup not necessarily finitely generated. We denote by $C(S)$ the closed cone in $\mathbbm{R}^{n+1}$ generated by $S$, i.e. the closure of the set of all linear combinations $\sum_{i}\lambda_{i}s_{i}$ with $\lambda_{i}\in\mathbbm{R}_{\geq 0}$ and $s_{i}\in S$. In this paper we will exclusively work with semigroups $S$ such that the pair $(S,\mathbbm{R}^{n}\times\mathbbm{R}_{\geq 0})$ is \emph{admissible}, i.e. $S\subset \mathbbm{R}^{n}\times \mathbbm{R}_{\geq 0}$, or \emph{strongly admissible}, i.e. it is admissible and $C(S)$ intersects the hyperplane $\mathbbm{R}^{n}\times \{0\}$ only in the origin (see section $\S 1.2$ in \cite{KKh12}). We recall that a closed convex cone $C$ with \emph{apex} the origin is called \emph{strictly convex} iff the biggest linear subspace contained in $C$ is the origin, so if $(S, \mathbbm{R}^{n}\times \mathbbm{R}_{\geq 0})$ is strongly admissible then $C(S)$ is strictly convex.
\begin{defn}
\label{defn:Okou}
Let $(S,\mathbbm{R}^{n}\times \mathbbm{R}_{\geq 0})$ be an admissible pair. Then
$$
\Delta(S):=\pi\big(C(S)\cap \{\mathbbm{R}^{n}\times \{1\}\}\big)
$$
is called \textbf{Okounkov convex set of} $(\mathbf{S},\mathbbm{R}^{\mathbf{n}}\times\mathbbm{R}_{\mathbf{\geq 0}})$, where $\pi:\mathbbm{R}^{n+1}\to \mathbbm{R}^{n}$ is the projection to the first $n$ coordinates. If $(S,\mathbbm{R}^{n}\times \mathbbm{R}_{\geq 0})$ is strongly admissible, $\Delta(S)$ is also called \textbf{Okounkov body} of $(\mathbf{S},\mathbbm{R}^{\mathbf{n}}\times\mathbbm{R}_{\mathbf{\geq 0}})$.
\end{defn}
\begin{rem}
\label{rem:Okou}
\emph{The convexity of $\Delta(S)$ is immediate, while it is not hard to check that it is compact iff the pair is strongly admissible. Moreover $S$ generates a subgroup of $\mathbbm{Z}^{n+1}$ of maximal rank iff $\Delta(S)$ has interior not-empty.}
\end{rem}
The following result is well-known and it has many interesting consequences.
\begin{thm}[\cite{KKh12}, Theorem 1.4]
\label{thm:KKh12F}
Let $S\subset \mathbbm{Z}^{n+1}$ be a finitely generated subsemigroup. Then there exists an element $\alpha_{0}\in S$ such that 
$$
\alpha_{0}+C(S)\cap G(S)\subset S
$$
where $G(S)\subset \mathbbm{Z}^{n+1}$ is the group generated by $S$.
\end{thm}
Defining $S^{k}:=\{\alpha\, :\, (k\alpha,k)\in S\}\subset \mathbbm{R}^{n}$ for $k\in \mathbbm{N}$, we get
\begin{prop}[\cite{WN15}]
\label{prop:First}
Let $(S,\mathbbm{R}^{n}\times \mathbbm{R}_{\geq 0})$ be an admissible pair. Then
$$
\Delta(S)=\overline{\bigcup_{k\geq 1}S^{k}}.
$$
Moreover if $K\subset \Delta(S)^{\circ}\subset \mathbbm{R}^{n}$ compact subset then $K\subset \mathrm{Conv} (S^{k})$ for $k\geq 1$ divisible enough, where $\mathrm{Conv}$ denotes the closed convex hull. In particular
$$
\Delta(S)^{\circ}=\bigcup_{k\geq 1}\mathrm{Conv}(S^{k})^{\circ}=\bigcup_{k\geq 1} \mathrm{Conv}(S^{k!})^{\circ}
$$
with $\mathrm{Conv}(S^{k!})$ non-decreasing in $k$.
\end{prop}
\begin{proof}
The inclusion $\Delta(S)\supset \overline{\bigcup_{k\geq 1}S^{k}}$ is immediate.\newline
To prove the reverse inequality and the rest of the statement we can assume $S$ finitely generated. Indeed for any $S$ not finitely generated, $\Delta(S)$ can be approximated by Okounkov convex sets $\Delta(S_{m})$ of an increasing sequence $\{S_{m}\}_{m\geq 1}$ of finitely generated subsemigroups of $S$. Clearly $\Delta(S_{m})\subset \Delta(S_{m+1})\subset \cdots\subset \Delta(S)$ and
$$
\overline{\bigcup_{m\geq 1}\Delta(S_{m})}=\Delta(S).
$$
Thus, assuming $S$ finitely generated and letting $\alpha_{0}\in S$ given by Theorem \ref{thm:KKh12F}, $(k\alpha,k)-\alpha_{0}\in C(S)\cap G(S)$ for any $\alpha\in\Delta(S)\cap \big(\frac{1}{k}\mathbbm{Z}\big)^{n}$ such that $(\alpha,1)$ has distance bigger than $|\alpha_{0}|/k$ from the boundary of $C(S)$. Therefore by Theorem \ref{thm:KKh12F} $(k\alpha,k)\in S$, i.e. $\alpha\in S^{k}$, which yields $\Delta(S)\subset \overline{\bigcup_{k\geq 1}S^{k}}$ varying $\alpha$ and $k\in \mathbbm{N}$. Moreover, for $K\subset \Delta(S)^{\circ}$ compact subset, letting $k_{0}\in\mathbbm{N}$ such that $|\alpha_{0}|/k_{0}< d\big(K,\partial \Delta(S)\big)/2$, we get $K\subset \mathrm{Conv}(S^{k})^{\circ}$ for any $k\geq k_{0}$. The Proposition follows.
\end{proof}
When a strong admissible pair $(S,\mathbbm{R}^{n}\times \mathbbm{R}_{\geq 0})$ satisfies the further hypothesis $\Delta(S)\subset\mathbb{R}^{n}_{\geq 0} $ then we denote with
$$
\Delta(S)^{\mathrm{ess}}:=\bigcup_{k\geq 1} \mathrm{Conv}(S^{k})^{\mathrm{ess}}
$$
the \emph{essential} Okounkov body where $\mathrm{Conv}(S^{k})^{ess}$ represents the interior of $\mathrm{Conv}(S^{k})$ as subset of $\mathbbm{R}^{n}_{\geq 0}$ with its induced topology (\cite{WN15}). Note that if $S$ is finitely generated then $\Delta(S)^{\mathrm{ess}}$ coincides with the interior of $\Delta(S)$ as subset of $\mathbbm{R}^{n}_{\geq 0}$, but in general they may be different since points in the hyperplanes $\{x_{i}=0\}$ may belong to $\Delta(S)$, and hence in its interior as subset of $\mathbbm{R}^{n}_{\geq 0}$, but not in $\Delta(S)^{ess}$.
\begin{prop}
\label{prop:Essential}
Let $(S,\mathbbm{R}^{n}\times \mathbbm{R}_{\geq 0})$ be a strongly admissible pair such that $\Delta(S)\subset \mathbbm{R}^{n}_{\geq 0}$, and let $K\subset \Delta(S)^{\mathrm{ess}}$ be a compact set. Then there exists $k\gg 1$ divisible enough such that $K\subset \mathrm{Conv}(S^{k})^{\mathrm{ess}}$. In particular
$$
\Delta(S)^{\mathrm{ess}}=\bigcup_{k\geq 1}\mathrm{Conv}(S^{k!})^{\mathrm{ess}}
$$
with $\mathrm{Conv}(S^{k!})^{\mathrm{ess}}$ non-decreasing in $k$, and $\Delta(S)^{\mathrm{ess}}$ is an open convex set of $\mathbbm{R}^{n}_{\geq 0}$.
\end{prop}
\begin{proof}
We may assume that $\Delta(S)^{\mathrm{ess}}\neq 0$ otherwise it is trivial. Therefore the subgroup of $\mathbbm{Z}^{n+1}$ generated by $S$ has maximal rank. Then the proof coincides with that of Proposition \ref{prop:First} exploiting again the strength of Theorem \ref{thm:KKh12F}. Indeed the unique difference is that $K$ may intersect the boundary of $\Delta(S)$ on some hyperplanes $\{x_{i}=0\}$ where with obvious notations $(x_{1},\dots,x_{n})$ denotes coordinates on $\mathbbm{R}^{n}_{\geq 0}$. But by definition this can only happen if such intersection points belong to $\mathrm{Conv}(S^{k})^{ess}$ for some $k$.
\end{proof}
We also recall the following key Theorem:
\begin{thm}[\cite{Bou14}, Théorème 1.12.; \cite{KKh12}, Theorem 1.14.]
\label{thm:AboutVolume}
Let $(S,\mathbbm{R}^{n}\times \mathbbm{R}_{\geq 0})$ be a strongly admissible pair, let $G(S)\subset \mathbbm{Z}^{n+1}$ be the group generated by $S$ and let $\mathrm{ind}_{1}$ and $\mathrm{ind}_{2}$ be respectively the index of the subgroups $\pi_{1}\big(G(S)\big)$ and $\pi_{2}\big(G(S)\big)$ in $\mathbbm{Z}^{n}$ and in $\mathbbm{Z}$ where $\pi_{1}$ and $\pi_{2}$ are respectively the projection to the first $n$-coordinates and to the last coordinate. Then
$$
\frac{\mathrm{Vol}_{\mathbbm{R}^{n}}\big(\Delta(S)\big)}{\mathrm{ind}_{1}\mathrm{ind}_{2}^{n}}=\lim_{m\to\infty, m\in \mathbbm{N}(S)}\frac{\# S^{m}}{m^{n}}
$$
where $\mathbbm{N}(S):=\{m\in \mathbbm{N}\,:\, S^{m}\neq \emptyset\}$ and the volume is respect to the Lebesgue measure.
\end{thm}
Finally we need to introduce the \emph{valuations}:
\begin{defn}
Let $V$ be an algebra over $\mathbbm{C}$. A \textbf{valuation} from $V$ to $\mathbbm{Z}^{n}$ equipped with a total additive order $>$ is a map $\nu: V\setminus \{0\}\to (\mathbbm{Z}^{n}, >)$ such that
\begin{itemize}
\item[i)] $\nu(f+g)\geq \min \{\nu(f),\nu(g)\}$ for any $f,g\in V\setminus \{0\}$ such that $f+g\neq 0$;
\item[ii)] $\nu(\lambda f)=\nu (f)$ for any $f\in V\setminus \{0\}$ and any $\mathbbm{C}\ni \lambda \neq 0$;
\item[iii)] $\nu(fg)=\nu(f)+\nu(g)$ for any $f,g\in V\setminus\{0\}$.
\end{itemize}
\end{defn}
Often $\nu$ is defined on the whole $V$ adding $+\infty$ to the group $\mathbbm{Z}^{n}$ and imposing $\nu(0):=+\infty$.\\
For any $\alpha\in \mathbbm{Z}^{n}$ the $\alpha-$leaf of the valuation is defined as the quotient of vector spaces
$$
\hat{V}_{\alpha}:=\frac{\{f\in V\setminus \{0\}\, : \, \nu(f)\geq \alpha\}\cup\{0\}}{\{f\in V\setminus \{0\}\, : \, \nu(f)> \alpha\}\cup\{0\}}.
$$
A valuation is said to have \emph{one-dimensional leaves} if the dimension of any leaf is at most $1$.
\begin{prop}[\cite{KKh12}, Proposition 2.6.]
Let $V$ be an algebra over $\mathbbm{C}$, and let $\nu:V\setminus\{0\}\to (\mathbbm{Z}^{n},>)$ be a valuation with one-dimensional leaves. Then for any no trivial subspaces $W\subset V$,
$$
\#\nu(W\setminus\{0\})=\dim_{\mathbbm{C}}W.
$$
\end{prop}
We will say that a valuation $\nu:V\setminus\{0\}\to (\mathbbm{Z}^{n},>)$ is \emph{faithful} if the field of fractions $K$ of $V$ has transcendental degree $n$ and  the extension $\nu: K\setminus \{0\}\to (\mathbbm{Z}^{n},>)$ defined as $\nu(f/g):=\nu(f)-\nu(g)$ (see Lemme $2.3$ in \cite{Bou14}) has the whole $\mathbbm{Z}^{n}$ as image. Note that any faithful valuation has one-dimensional leaves (see Remark $2.26.$ in \cite{Bou14}).
\subsection{The Okounkov body associated to a line bundle}
\label{ssec:Okounkov}
In this section we recall the construction and some known results of the Okounkov body associated to a line bundle $L$ and a point $p\in X$ (see \cite{LM09},\cite{KKh12} and \cite{Bou14}).\\
Consider the abelian group $\mathbbm{Z}^{n}$ equipped with a total additive order $>$, let $\nu:\mathbbm{C}(X)\setminus \{0\}\to (\mathbbm{Z}^{n},>)$ be a faithful valuation with \emph{center} $p\in X$. We recall that $p\in X$ is the (unique) center of $\nu$ if $\mathcal{O}_{X,p}\subset \{f\in \mathbbm{C}(X)\, :\, \nu(f)\geq 0\}$ and $\mathfrak{m}_{X,p}\subset \{f\in\mathbbm{C}(X)\,:\, \nu(f)>0\}$, and that the semigroup $\nu(\mathcal{O}_{X,p}\setminus \{0\})$ is well-ordered by the induced order (see $\S 2$ in \cite{Bou14}).\\
Assume that $L_{|U}$ is trivialized by a non--zero local section $t$. Then any section $s\in H^{0}(X,kL)$ can be written locally as $s=ft^{k}$ with $f\in\mathcal{O}_{X}(U)$. Thus we define $\nu(s):=\nu(f)$, where we identify $\mathbbm{C}(X)$ with the meromorphic function field and $\mathcal{O}_{X,p}$ with the stalk of $\mathcal{O}_{X}$ at $p$. We observe that $\nu(s)$ does not depend on the trivialization $t$ chosen since any other trivialization $t'$ of $L_{|V}$ differs from $t$ on $U\cap V$ by an unit $u\in \mathcal{O}_{X}(U\cap V)$. We define an additive semigroup associated to the valuation by
$$
\Gamma:=\{(\nu(s),k)\,:\, s\in H^{0}(X,kL)\setminus\{0\},k\geq 1\}\subset \mathbbm{Z}^{n}\times\mathbbm{Z}.
$$
The \textbf{Okounkov body} $\Delta(L)$ is the Okounkov convex set of $(\Gamma, \mathbbm{R}^{n}\times\mathbbm{R}_{\geq 0})$ (see Definition \ref{defn:Okou}), i.e.
$$
\Delta(L):=\pi\big(C(\Gamma)\cap\{\mathbbm{R}^{n}\times\{1\}\}\big)
$$
where $\pi:\mathbbm{R}^{n}\times \mathbbm{R}\to \mathbbm{R}^{n}$ is the projection to the first $n$ coordinated. By Proposition \ref{prop:First} it follows that
$$
\Delta(L)=\overline{\bigcup_{k\geq 1}\Big\{\frac{\nu(s)}{k}\, :\, s\in H^{0}(X,kL)\setminus\{0\}\Big\}}=\mathrm{Conv}\Big(\Big\{\frac{\nu(s)}{k}\, :\, s\in H^{0}(X,kL)\setminus\{0\},k\geq 1\Big\}\Big),
$$
and by construction $\Delta(L)$ is a convex set of $\mathbbm{R}^{n}$ with interior not-empty iff $\Gamma$ generates a subgroup of $\mathbbm{Z}^{n+1}$ of maximal rank (Remark \ref{rem:Okou}).\\
For a prime divisor $D\in \mathrm{Div}(X)$ we set $\nu(D)=\nu(f)$ for $f$ any local equation for $D$ near $p$, noting that the map $\nu: \mathrm{Div}(X)\to \mathbbm{Z}^{n}$ extends to a $\mathbbm{R}-$linear map from $\mathrm{Div}(X)_{\mathbbm{R}}$.
\begin{thm}[\cite{LM09},\cite{KKh12}]
\label{thm:OneOko}
The following statements hold:
\begin{itemize}
\item[i)] $\Delta(L)$ is a compact convex set lying in $\mathbbm{R}^{n}$;
\item[ii)] $n!\mathrm{Vol}_{\mathbbm{R}^{n}}\big(\Delta(L)\big)=\mathrm{Vol}_{X}(L)$, and in particular $L$ is big iff $\Delta(L)^{\circ}\neq \emptyset$, i.e $\Delta(L)$ is a convex body;
\item[iii)] if $L$ is big then $\Delta(L)=\overline{\{D\in\mathrm{Div}_{\geq 0}(X)_{\mathbbm{R}}\,:\, D\equiv_{num} L\}}$ and, in particular, the Okounkov body only depends on the numerical class of the big line bundle.
\end{itemize}
\end{thm}
\paragraph{Quasi-monomial valuation} \label{paragraph:ParticularValuations} Equip $\mathbbm{Z}^{n}$ of a total additive order $>$, fix $\vec{\lambda}_{1},\dots,\vec{\lambda}_{n}\in\mathbbm{Z}^{n}$ linearly independent and fix local holomorphic coordinates $\{z_{1},\dots,z_{n}\}$ around a fixed point $p$. Then we can define the \emph{quasi-monomial} valuation %\footnote{recall that there exists an unique valuation $\tilde{\nu}:\mathbbm{C}(X)\setminus\{0\}\to \mathbbm{Z}^{n}$ that extends $\nu$ defined by $\tilde{\nu}(f/g):=\nu(f)-\nu(g)$, for a more general construction we refer to the section $\S 2.3$ in \cite{Bou14}.}
$\nu:\mathcal{O}_{X,p}\setminus\{0\}\to \mathbbm{Z}^{n}$ by
$$
\nu(f):=\min\Big\{\sum_{i=1}^{n}\alpha_{i}\vec{\lambda}_{i} \, : \, a_{\alpha}\neq 0 \, \mathrm{where}\, \mathrm{locally} \, \mathrm{around}\, p, f=_{U}\sum_{\alpha\in\mathbbm{N}^{n}}a_{\alpha}z^{\alpha}\Big\}
$$
where the minimum is taken respect to the order $>$ fixed on $\mathbbm{Z}^{n}$. Note that such valuation is faithful iff $\det (\vec{\lambda}_{1},\dots,\vec{\lambda}_{n})=\pm 1$.\\
For instance if we equip $\mathbbm{Z}^{n}$ of the lexicographical order, for $\vec{\lambda}_{j}=\vec{e}_{j}$ ($j-$th vector of the canonical base of $\mathbbm{R}^{n}$) we get
$$
\nu(f):=\min_{lex}\Big\{\alpha \,:\, a_{\alpha}\neq 0 \,\mathrm{where}\,\mathrm{locally}\,\mathrm{around}\, p, f=_{U}\sum_{\alpha\in\mathbbm{N}^{n}}a_{\alpha}z^{\alpha}\Big\}.
$$
This is the valuation associated to an admissible flag $X=Y_{0}\supset Y_{1}\supset\cdots\supset Y_{n}=\{p\}$, in the sense of \cite{LM09}\footnote{$Y_{i}$ smooth irreducible subvariety of $X$ of codimension $i$ such that $Y_{i}$ is a Cartier divisor in $Y_{i-1}$ for any $i=1,\dots,n$.}, such that locally $Y_{i}:=\{z_{1}=\cdots=z_{i}=0\}$ (see also \cite{WN15}).\\
A change of coordinates with the same local flag produces the same valuation, i.e. the valuation depends uniquely on the local flag.\\
\textbf{Note:} \emph{In the paper a valuation associated to an admissible flag $Y_{\cdot}$ will be the valuation constructed by the local procedure starting from local holomorphic coordinates as just described.}\\
On the other hand if we equip $\mathbbm{Z}^{n}$ of the deglex order and we take $\vec{\lambda}_{i}=\vec{e}_{i}$, we get the valuation $\nu: \mathcal{O}_{X,p}\setminus\{0\}\to \mathbbm{Z}^{n}$,
$$
\nu(f):=\min_{deglex}\Big\{\alpha \,:\, a_{\alpha}\neq 0 \,\mathrm{where}\,\mathrm{locally}\,\mathrm{around}\, p, f=_{U}\sum_{\alpha\in\mathbbm{N}^{n}}a_{\alpha}z^{\alpha}\Big\}.
$$
This is the valuation associated to an \emph{infinitesimal} flag $Y_{\cdot}$ in $p$: given a flag of subspaces $T_{p}X=:V_{0}\supset V_{1}\supset \cdots\supset V_{n}=\{0\}$ such that $\dim_{\mathbbm{C}}V_{i}=n-i$, consider on $\tilde{X}:=\mathrm{Bl}_{p}X$ the flag
$$
\tilde{X}=:Y_{0}\supset\mathbbm{P}(T_{p}X)=\mathbbm{P}(V_{0})=:Y_{1}\supset \cdots\supset \mathbbm{P}(V_{n-1})=:Y_{n}=:\{\tilde{p}\}.
$$
Note that $Y_{\cdot}$ is an admissible flag around $\tilde{p}$ on the blow-up $\tilde{X}$. Indeed we recover the valuation on $\tilde{X}$ associated to this admissible flag considering $F\circ \nu$ where $F:(\mathbbm{Z}^{n},>_{deglex})\to(\mathbbm{Z}^{n},>_{lex})$ is the order-preserving isomorphism $F(\alpha):=(|\alpha|,\alpha_{1},\dots,\alpha_{n-1})$, i.e. considering the quasi-monomial valuation given by the lexicographical order and $\vec{\lambda}_{i}:=\vec{e}_{1}+\vec{e}_{i}$.\\
\textbf{Note:} \emph{In the paper a valuation associated to an infinitesimal flag $Y_{\cdot}$ will be the valuation $\nu$ constructed by the local procedure starting from local holomorphic coordinates as just described, and in particular the total additive order on $\mathbbm{Z}^{n}$ will be the deglex order in this case.}
\subsection{A moment map associated to an $(S^{1})^{n}-$action on a particular manifold}
\label{ssec:Moment}
In this brief subsection we recall some results regarding a moment map for an $(S^{1})^{n}-$action on a symplectic manifold $(M,\omega) $ constructed from a convex hull of a finite set $\mathcal{A}\subset\mathbbm{N}^{n}$ (see section $\S 3$ in \cite{WN15}).

Let $\mathcal{A}\subset \mathbbm{N}^{n}$ be a finite set, let $\mu:\mathbbm{C}^{n}\to\mathbbm{R}^{n}$ be the map $\mu(z_{1},\dots,z_{n}):=(|z_{1}|^{2},\dots,|z_{n}|^{2})$.\\
Then if $\mathrm{Conv}(\mathcal{A})^{\mathrm{ess}}\neq \emptyset$, we define
$$
\mathcal{D}_{\mathcal{A}}:=\mu^{-1}\big(\mathrm{Conv}(\mathcal{A})^{\mathrm{ess}}\big)=\mu^{-1}\big(\mathrm{Conv}(\mathcal{A})\big)^{\circ}
$$
where we have denoted by $\mathrm{Conv}(\mathcal{A})^{\mathrm{ess}}$ the interior of $\mathrm{Conv}(\mathcal{A})$ respect to the induced topology on $\mathbbm{R}^{n}_{\geq 0}$. Next we define $M_{\mathcal{A}}$ as the manifold we get removing from $\mathbbm{C}^{n}$ all submanifolds given by $\{z_{i_{1}}=\cdots=z_{i_{r}}=0\}$ which do not intersect $\mathcal{D}_{\mathcal{A}}$. We equip such manifold with the form $\omega_{\mathcal{A}}:=dd^{c}\phi_{\mathcal{A}}$ where
$$
\phi_{\mathcal{A}}(z):=\ln\Big(\sum_{\alpha\in\mathcal{A}}|\mathbf{z}^{\alpha}|^{2}\Big).
$$
Here $\mathbf{z}=\{z_{1},\dots,z_{n}\}$ and $\mathbf{z}^{\alpha}=z_{1}^{\alpha_{1}} \cdots z_{n}^{\alpha_{n}}$. Clearly, by construction, $\omega_{\mathcal{A}}$ is an $(S^{1})^{n}-$invariant Kähler form on $M_{\mathcal{A}}$, so in particular $(M_{\mathcal{A}},\omega_{\mathcal{A}})$ can be thought as a symplectic manifold. Moreover defining $f(w_{1},\dots,w_{n}):=(e^{w_{1}/2},\dots,e^{w_{n}/2})$, the function $u_{\mathcal{A}}(w):=\phi_{\mathcal{A}}\circ f(w)$ is plurisubharmonic and independent of the imaginary part $y_{i}$, and $f^{*}\omega_{\mathcal{A}}=dd^{c}u_{\mathcal{A}}$. Thus an easy calculation shows that
$$
dd^{c}u_{\mathcal{A}}=\frac{1}{4\pi}\sum_{j,k=1}^{n}\frac{\partial^{2}u_{\mathcal{A}}}{\partial x_{k}\partial x_{j}}dy_{k}\wedge dx_{j}
$$
which implies
$$
d\frac{\partial}{\partial x_{k}}u_{\mathcal{A}}=dd^{c}u_{\mathcal{A}}\Big(4\pi\frac{\partial}{\partial y_{k}}, \cdot\Big).
$$
Therefore, setting $H_{k}:=\frac{\partial u_{\mathcal{A}}}{\partial x_{k}}\circ f^{-1}$, since $(f^{-1})_{*}\big(2\pi\frac{\partial}{\partial \theta_{k}}\big)=4\pi \frac{\partial}{\partial y_{k}}$, we get
$$
dH_{k}=\omega_{\mathcal{A}}\Big(2\pi\frac{\partial}{\partial\theta_{k}},\cdot\Big).
$$
Hence $\mu_{\mathcal{A}}=(H_{1},\dots,H_{n})=\nabla u_{\mathcal{A}}\circ f^{-1}$ is a moment map for the $(S^{1})^{n}-$action on the symplectic manifold $(M_{\mathcal{A}},\omega_{\mathcal{A}})$. Furthermore it is not hard to check that $\mu_{\mathcal{A}}\big((\mathbbm{C}^{*})^{n}\big)=\mathrm{Conv}(\mathcal{A})^{\circ}$, that $\mu_{\mathcal{A}}(M_{\mathcal{A}})=\mathrm{Conv}(\mathcal{A})^{\mathrm{ess}}$ and that for any $U\subset M_{\mathcal{A}}$, setting $f^{-1}(U)=V\times (i\mathbbm{R}^{n})$,
$$
\int_{U}\omega_{\mathcal{A}}^{n}=\int_{V\times (i[0,4\pi])^{n}}(dd^{c}u_{\mathcal{A}})^{n}=n!\int_{V}\det (\mathrm{Hess}(u_{\mathcal{A}}))=n!\int_{\nabla u_{\mathcal{A}}(V)}dx=n! \mathrm{Vol}(\mu_{\mathcal{A}}(U)).
$$
Finally we quote here an useful result:
\begin{lem}[\cite{WN15}, Lemma $3.1.$]
\label{lem:Legendre}
Let $U$ be a relatively compact open subset of $\mathcal{D}_{\mathcal{A}}$. Then there exists a smooth function $g:M_{A}\to \mathbbm{R}$ with compact support such that $\omega=\omega_{\mathcal{A}}+dd^{c}g$ is Kähler and $\omega=\omega_{st}$ over $U$.
\end{lem}
\section{Multipoint Okounkov bodies}
\label{sec:MOB}
We fix an additive total order $>$ on $\mathbbm{Z}^{n}$ and a family of faithful valuations $\nu^{p_{j}}:\mathbbm{C}(X)\setminus\{0\}\to (\mathbbm{Z}^{n},>)$ centered at $p_{j}$, where recall that $p_{1},\dots,p_{N}$ are different points chosen on the $n-$dimensional projective manifold $X$ and $L$ is a line bundle on $X$.
\begin{defn}
We define $V_{\cdot,j}\subset R(X,L)$ as
$$
V_{k,j}=\{s\in H^{0}(X,kL)\setminus\{0\}\, :\, \nu^{p_{j}}(s)<\nu^{p_{i}}(s)\,\mathrm{for}\, \mathrm{any} \, i\neq j\}.
$$
\end{defn}
\begin{rem}
\emph{They are disjoint graded subsemigroups with respect to the multiplicative action since $\nu^{p_{j}}(s_{1}\otimes s_{2})=\nu^{p_{j}}(s_{1})+\nu^{p_{j}}(s_{2})$, but they are not necessarily closed under addition and $\cup_{j=1}^{N}V_{k,j}$ is typically strictly contained in $H^{0}(X,kL)\setminus\{0\}$ for some $k\geq 1$. Note that $V_{k,j}$ contains sections whose leading term at $p_{j}$ with respect to $\nu^{p_{j}}$ is strictly smaller than the leading term at $p_{i}$ with respect to $\nu^{p_{i}}$ for any $i\neq j$.}
\end{rem}
Clearly the properties of the valuations $\nu^{p_{j}}$ assure that
\begin{itemize}
\item[i)] $\nu^{p_{j}}(s)=+\infty$ iff $s=0$ (by extension $\nu^{p_{j}}(0):=+\infty$);
\item[ii)] for any $s\in V_{\cdot,j}$ and for any $0\neq a\in \mathbbm{C}$, $\nu^{p_{j}}(as)=\nu^{p_{j}}(s)$.
\end{itemize}
Thus we can define
$$
\Gamma_{j}:=\{(\nu^{p_{j}}(s),k)\,:\, s\in V_{k,j},k\geq 1\}\subset \mathbbm{Z}^{n}\times \mathbbm{Z}.
$$
\begin{lem}
\label{lem:Additive}
$\Gamma_{j}$ is an additive subsemigroup of $\mathbbm{Z}^{n+1}$ and $(\Gamma_{j},\mathbbm{R}^{n}\times \mathbbm{R})$ is a strongly admissible pair.
\end{lem}
\begin{proof}
The first part is an immediate consequence of the definition, while the second assertion follows from the inclusion $\Gamma_{j}\subset \Gamma_{p_{j}}:=\{(\nu^{p_{j}}(s),k)\,:\,s\in H^{0}(X,kL)\setminus\{0\},k\geq 1\}$ (see subsection \ref{ssec:Okounkov}).
\end{proof}
\begin{defn}
We call $\Delta_{j}(L):=\Delta(\Gamma_{j})$ the \textbf{multipoint Okounkov body} of $L$ at $p_{j}$.
\end{defn}
Observe that by Proposition \ref{prop:First} $\Delta_{j}(L)=\overline{\bigcup_{k\geq 1}\frac{\nu^{p_{j}}(V_{k,j})}{k}}$ and that these multipoint Okounkov bodies depend on the choice of the faithful valuations $\nu^{p_{1}},\dots,\nu^{p_{N}}$, but we omit the dependence to simplify the notations.
\begin{rem}
\label{rem:ParticularValuations}
\emph{If we fix local holomorphic coordinates $\{z_{j,1},\dots,z_{j,n}\}$ around $p_{j}$, we can consider any family of faithful quasi-monomial valuations $\nu^{p_{j}}$ with center $p_{1},\dots,p_{N}$ (see paragraph $\S \ref{paragraph:ParticularValuations}$) with respect to the same choice of a total additive order on $\mathbbm{Z}^{n}$ and to the choice of families of $\mathbbm{Z}-$linearly independent vectors $\vec{\lambda}_{1,j},\dots,\vec{\lambda}_{n,j}\in\mathbbm{Z}^{n}$ (these families of vectors may be different). For instance we can choose those associated to the family of admissible flags $Y_{j,i}:=\{z_{j,1}=\cdots=z_{j,i}=0\}$ (with $\mathbbm{Z}^{n}$ equipped of the lexicographical order) or those associated to the family of infinitesimal flags $Y_{j,\cdot}$ (with in this case $\mathbbm{Z}^{n}$ equipped pf the deglex order).}
\end{rem}
\begin{lem}
\label{lem:IntOk}
The following statements hold:
\begin{itemize}
\item[i)] $\Delta_{j}(L)$ is a compact convex set contained in $\mathbbm{R}^{n}$;
\item[ii)] if $p_{j}\notin \mathbbm{B}_{+}(L)$ then $\Gamma_{j}(L)$ generates $\mathbbm{Z}^{n+1}$ as a group. In particular $\Delta_{j}(L)^{\circ}\neq \emptyset$;
\item[iii)] if $\Gamma_{j}(L)$ is not empty then it generates $\mathbbm{Z}^{n+1}$ as a group. In particular $\Delta_{j}(L)^{\circ}\neq \emptyset$ iff $\Delta_{j}(L)\neq \emptyset$.
\end{itemize}
\end{lem}
\begin{proof}
The first point follows by construction (see Definition \ref{defn:Okou} and Remark \ref{rem:Okou}).\\
\emph{\textbf{Proof of (ii).}} Proceeding similarly to Lemma $2.2$ in \cite{LM09}, let $D$ be a big divisor such that $L=\mathcal{O}_{X}(D)$ and let $A,B$ be two fixed ample divisors such that $D=A-B$. Since $D$ is big there exists $\mathbbm{N}\ni k\gg 1$ such that $kD-B$ is linearly equivalent to an effective divisor $F$.\\
Moreover, since by hypothesis $p_{j}\notin\mathbbm{B}_{+}(L)$, by taking $k\gg 1$ big enough, we may assume that $p_{j}\notin \mathrm{Supp}(F)$ (see Corollary $1.6.$ in \cite{ELMNP06}), thus $F$ is described by a global section $f$ that is an unity in $\mathcal{O}_{X,p_{j}}$. Then, possibly adding a very ample divisor to $A$ and $B$ we may suppose that there exist sections $s_{0},\dots,s_{n}\in V_{1,j}(B)$ such that $\nu^{p_{j}}(s_{0})=\vec{0}$ and $\nu^{p_{j}}(s_{l})=\vec{\lambda}_{l}$ for any $l=1,\dots,n$ where $\vec{\lambda}_{1},\dots,\vec{\lambda}_{n}$ are linearly independent vectors in $\mathbbm{Z}^{n}$ which generate all $\mathbbm{Z}^{n}$ as a group (remember that the valuations $\nu^{p_{j}}$ are faithful). Thus, since $s_{i}\otimes f\in V_{1,j}(kL)$ for any $i=0,\dots,n$ and $\nu^{p_{j}}(f)=\vec{0}$, we get
$$
(\vec{0},k),(\vec{\lambda}_{1},k),\dots,(\vec{\lambda}_{n},k)\in\Gamma_{j}(L).
$$
And, since $(k+1)D-F$ is linearly equivalent to $A$ we may also assume that $(\vec{0},k+1)\in \Gamma_{j}(L)$, which concludes the proof of $(ii)$.\\
\emph{\textbf{Proof of (iii).}} Let $s\in V_{k,j}(L)$ such that $(\nu^{p_{j}}(s),k)\in \Gamma_{j}(L)$ and set $\vec{w}:=\nu^{p_{j}}(s)$. Then by Lemma $2.2$ in \cite{LM09} there exists $m\in \mathbbm{N}$ big enough and a vector $\vec{v}\in\mathbbm{Z}^{n}$ such that
\begin{equation}
\label{eqn:ValPoin}
(\vec{v},m),(\vec{v}+\vec{\lambda}_{1},m),\dots,(\vec{v}+\vec{\lambda}_{n},m), (\vec{v},m+1)\in \Gamma(L)
\end{equation}
where with $\Gamma(L)$ we denote the semigroup associated to $\nu^{p_{j}}$ for the one-point Okounkov body (see subsection \ref{ssec:Okounkov}) and where $\vec{\lambda}_{1},\dots,\vec{\lambda}_{n}$ are linearly independent vectors in $\mathbbm{Z}^{n}$ as in $(ii)$. The points in (\ref{eqn:ValPoin}) correspond to sections $t_{0},\dots,t_{n}\in H^{0}(X,mL)\setminus \{0\}, t_{n+1}\in H^{0}(X,(m+1)L)\setminus\{0\}$. Next by definition of $V_{\cdot,j}(L)$ there exists $N\gg 1$ big enough such that $s^{N}\otimes t_{j}\in V_{Nk+m,j}(L)$ for any $j=0,\dots,n$ and $s^{N}\otimes t_{n+1}\in V_{Nk+m+1}(L)$. Therefore
$$
(N\vec{w}+\vec{v},m),(N\vec{w}+\vec{v}+\vec{\lambda}_{1},m),\dots,(N\vec{w}+\vec{v}+\vec{\lambda}_{n},k), (N\vec{w}+\vec{v},m+1)\in \Gamma_{j}(L),
$$
which concludes the proof.
\end{proof}
\begin{rem}
\emph{Let $X$ be a curve, $L$ be a line bundle of degree $\deg L=c$, and $p_{1},\dots,p_{N}$ be different points on $X$. Then by the proof of Lemma \ref{lem:IntOk}, $\Delta_{j}(L)$ are intervals in $\mathbbm{R}$ containing the origin. Moreover if the points are very general and the faithful valuations $\nu^{p_{j}}$ are associated to admissible or to infinitesimal flags, then $\Delta_{j}(L)=[0,c/N]$ for any $j=1,\dots,N$ as a consequence of Theorem \ref{ThmA}.}
\end{rem}
\begin{rem}
\label{rem:Ex}
\emph{In higher dimension, however, the situation is more complicated. Indeed it may happen that $\Delta_{j}(L)=\emptyset$ for some $j$ as the following simple example shows.\\
Consider on $X=\mathrm{Bl}_{q}\mathbbm{P}^{2}$ two points $p_{1}\notin \mathrm{Supp}(E)$ and $p_{2}\in \mathrm{Supp}(E)$ ($E$ exceptional divisor), and consider the big line bundle $L:=H+aE$ for $a>1$. Clearly, if we consider the family of admissible flags given by any fixed holomorphic coordinates centered at $p_{1}$ and holomorphic coordinates $\{z_{1,2},z_{2,2}\}$ centered at $p_{2}$ where locally $E=\{z_{1,2}=0\}$, then $\Delta_{2}(L)= \emptyset$. Indeed by the theory of one-point Okounkov bodies for surfaces (see section $6.2$ in \cite{LM09}) $\Delta_{1}(L)\subset \Delta^{p_{1}}(L)=\Sigma$ (where $\Sigma$ is the standard $2-$simplex and $\Delta^{p_{1}}(L)$ the one-point Okounkov body) while $\Delta_{2}(L)\subset \Delta^{p_{2}}(L)=(a,0)+\Sigma^{-1}$ ($\Sigma^{-1}=Conv(\vec{0},\vec{e}_{1},\vec{e}_{1}+\vec{e}_{2})$ inverted simplex), and the conclusion follows by construction. Actually, from Theorem \ref{ThmA} we get $\Delta_{1}(L)=\Sigma$.\\
We refer to subsection \ref{subsection:Surfaces} for a detailed analysis on the multipoint Okounkov bodies on surfaces, and to subsection \ref{subsection:Toric} for the toric case.}
\end{rem}
\subsection{Proof of Theorem \ref{ThmA}}
\label{subsection:TheoremA}
The goal of this section is to prove Theorem \ref{ThmA}.
\begin{reptheorem}{ThmA}
Let $L$ be a big line bundle. Then
$$
n!\sum_{j=1}^{N}\mathrm{Vol}_{\mathbbm{R}^{n}}(\Delta_{j}(L))=\mathrm{Vol}_{X}(L)
$$
\end{reptheorem}
We first introduce $W_{\cdot,j}\subset R(X,L)$ as
\begin{multline*}
W_{k,j}:=\{s\in H^{0}(X,kL)\setminus\{0\}\, : \, \nu^{p_{j}}(s)\leq \nu^{p_{i}}(s)\, \mathrm{if}\, 1\leq i\leq j\, \mathrm{and} \,\nu^{p_{j}}(s)< \nu^{p_{i}}(s)\, \mathrm{if} \, j<i\leq N\}
\end{multline*}
and we set $\Gamma_{W,j}:=\{(\nu^{p_{j}}(s),k)\, : \, s\in W_{k,j}, k\geq 1\}$. It is clear $W_{\cdot,j}$ are graded subsemigroups of $R(X,L)$ and that Lemma \ref{lem:Additive} holds for $\Gamma_{W,j}$. Moreover they are closely related to $V_{\cdot,j}$ and $\bigsqcup_{j=1}^{N} W_{k,j}=H^{0}(X,kL)\setminus \{0\}$ for any $k\geq 1$, but they depend on the order chosen on the points.
\begin{lem}
\label{prop:Vol}
For every $k\geq 1$ we have that
$$
\sum_{j=1}^{N}\# \Gamma_{W,j}^{k}= h^{0}(X,kL),
$$
where we recall that $\Gamma_{W,j}^{k}:=\{\alpha\in\mathbbm{R}^{n} : (k\alpha,k)\in \Gamma_{W,j}\}$.
\end{lem}
\begin{proof}
We define a new valuation $\nu: \mathbbm{C}(X)\setminus\{0\}\to 
{\mathbbm{Z}^{n}\times\dots\times\mathbbm{Z}^{n}\simeq \mathbbm{Z}^{Nn}}$ given by $\nu(f):=(\nu^{p_{1}}(f),\dots,\nu^{p_{N}}(f))$, where we put on $\mathbbm{Z}^{Nn}$ the lexicographical order on the product of $N$ total ordered abelian groups $\mathbbm{Z}^{n}$, i.e.
$$
(\lambda_{1},\dots,\lambda_{N})<(\mu_{1},\dots,\mu_{N}) \, \mbox{if there exists} \,  j\in\{1,\dots,N\} \, \mathrm{s.t.} \, \lambda_{i}=\mu_{i} \, \forall i < j \,\mathrm{and} \, \lambda_{j}<\mu_{j}.
$$
%It is easy to see that $\nu$ actually defines a valuation.\\
Fix $k\in\mathbbm{N}$. For every $j=1,\dots, N$, let $ \Gamma_{W,j}^{k}=\{\alpha_{j,1},\dots, \alpha_{j,r_{j}}\} $ and $ s_{j,1},\dots, s_{j,r_{j}}\in W_{k,j} $ be a set of sections such that $ \nu^{p_{j}}(s_{j,l})=\alpha_{j,l} $ for any $l=1,\dots, r_{j}$.\\
We next prove that $\{s_{1,1},\dots,s_{N,r_{N}}\}$ is a basis of $H^{0}(X,kL)$.\\
Let $\sum_{i=1}^{r}\mu_{i}s_{i}=0$ be a linear relation in which $\mu_{i}\neq 0$, $s_{i}\in\{s_{1,1},\dots,s_{N,r_{N}}\}$ for all $i=1,\dots, r$ and $s_{i}\neq s_{j}$ if $i\neq j$.
By construction we know that $ \nu(s_{1}),\dots,\nu(s_{r}) $ are different points in $\mathbbm{Z}^{Nn}$. Thus without loss of generality we can assume that $\nu(s_{1})<\dots<\nu(s_{r})$, but the relation
$$
s_{1}=-\frac{1}{\mu_{1}}\sum_{i=2}^{N}\mu_{i}s_{i}
$$
implies that $\nu(s_{1})\geq \min \{\nu(s_{j}) \, : \, j=2,\dots,r\}$ which is the contradiction. Hence $\{s_{1,1},\dots,s_{N,r_{N}}\}$ is a system of linearly independent vectors, thus to conclude the proof it is enough to show that it generates all $H^{0}(X,kL)$.\\
Let $t_{0}\in H^{0}(X,kL)\setminus\{0\}$ be a section and set ${\lambda_{0}:=(\lambda_{0,1},\dots,\lambda_{0,N}):=\nu(t_{0})}$. By definition of $W_{\cdot,j}$ there exists an unique $j_{0}\in 1,\dots, N$ such that $t_{0}\in W_{k,j_{0}}$, which means that $\lambda_{0,i}\geq \lambda_{0,j_{0}} $ if $1\leq i\leq j_{0}$, and that $ \lambda_{0,i}> \lambda_{0,j_{0}} $ if $ j_{0}< i\leq N$. Therefore by construction there exists $l\in\{1,\dots,r_{j_{0}}\}$ such that $ \lambda_{0,j_{0}}=\nu^{p_{j_{0}}}(s_{j_{0},l})$, and we set $s_{0}:=s_{j_{0},l}$. But %observing that $G^{\geq}_{k,\lambda_{0,j_{0}}}:=\{s\in H^{0}(X,kL)\setminus\{0\} \, : \, \nu^{p_{j_{0}}}(s)\geq \lambda_{0,j_{0}}\}\cup\{0\}$ and $G^{>}_{k,\lambda_{0,j_{0}}}:=\{s\in H^{0}(X,kL)\setminus\{0\} \, : \, \nu^{p_{j_{0}}}(s)>\lambda_{0,j_{0}}\}\cup\{0\}$ are vector spaces, we get
$$
\mathrm{dim}\Bigg(\frac{\{s\in H^{0}(X,kL)\setminus\{0\}\, :\, \nu^{p_{j_{0}}}(s)\geq \lambda_{0,j_{0}}\}\cup\{0\}}{\{s\in H^{0}(X,kL)\setminus\{0\}\, :\, \nu^{p_{j_{0}}}(s)> \lambda_{0,j_{0}}\}\cup\{0\}}\Bigg)\leq 1,
$$
since $\nu^{p_{j_{0}}}$ has one-dimensional leaves, so there exists a coefficient $a_{0}\in\mathbbm{C}$ such that $\nu^{p_{j_{0}}}(t_{0}-a_{0}s_{0})>\lambda_{0,j_{0}}$. Thus if $t_{0}=a_{0}s_{0}$ we can conclude the proof, otherwise we set $t_{1}:=t_{0}-a_{0}s_{0}$ and we iterate the process setting ${\lambda_{1}:=(\lambda_{1,1},\dots,\lambda_{1,N}):=\nu(t_{1})}$. Observe that ${\min_{j}\lambda_{1,j}\geq\min_{j}\lambda_{0,j}=\lambda_{0,j_{0}}}$ and that the inequality is strict if $t_{1}\in W_{k,j_{0}}$.\\
Summarizing we obtain $t_{0},t_{1},\dots, t_{l}\in H^{0}(X,kL)\setminus\{0\}$ such that $t_{l}:=t_{l-1}-a_{l-1}s_{l-1}\in W_{k,j_{l}}$ for an unique $j_{l}\in \{1,\dots,N\}$ where $s_{l-1}\in\{s_{j_{l-1},1},\dots,s_{j_{l-1},r_{l-1}}\}$ satisfies $\nu^{p_{j_{l-1}}}(t_{l-1})=\nu^{p_{j_{l-1}}}(s_{l-1})$, and $\min_{j}\lambda_{l,j}\geq\min_{j}\lambda_{l-1,j}$ for $\nu(t_{l})=:\lambda_{l}$. Therefore we get a sequence of valuative points $\lambda_{l}$ such that $\min_{j}\lambda_{l,j}\geq\min_{j}\lambda_{l-1,j}\geq\dots \geq \min_{j}\lambda_{0,j} $ where by construction there is at least one strict inequality if $l>N$. Hence we deduce that the iterative process must conclude since that the set of all valuative points of $\nu$ is finite as easy consequence of the finite cardinality of $\Gamma_{W,j}^{k}$ for each $j=1,\dots,N$.
%Moreover by construction we have for any $l\geq 1$
%\begin{gather*}
%\lambda_{l,i}\geq \lambda_{l-1,j_{l-1}} \, \mathrm{for} \, \mathrm{any} \, 1\leq i < j_{l-1},\\
%\lambda_{l,i}> \lambda_{l-1,j_{l-1}} \, \mathrm{for} \, \mathrm{any} \, j_{l-1}\leq i\leq N 
%\end{gather*}
%Indeed $\lambda_{l,i}=\nu^{p_{i}}(t_{l})\geq\min \{\nu^{p_{i}}(t_{l-1}),\nu^{p_{i}}(s_{l-1})\}$ and since $t_{l-1},s_{l-1}\in W_{k,j_{l-1}}$ we have $\nu^{p_{i}}(t_{l-1})\geq \nu^{p_{j_{l-1}}}(t_{l-1})=\lambda_{l-1,j_{l-1}}$ if $1\leq i \leq j_{l-1}$ while we have a strict inequality for $j_{l-1}< i\leq N$ (similarly for $s_{l-1}$ using that $\nu^{p_{j_{l-1}}}(s_{l-1})=\lambda_{l-1,j_{l-1}}$). Finally by construction for $i=j_{l-1}$ we have $\lambda_{l,j_{l-1}}=\nu^{p_{j_{l-1}}}(t_{l})>\nu^{p_{j_{l-1}}}(t_{l-1})=\lambda_{l-1,j_{l-1}}$.\\
%Thus we get an increasing sequence of vectors $\lambda_{l,j_{l}}\geq \lambda_{l-1,j_{l-1}}\geq \dots\geq \lambda_{0,j_{0}}$ such that, by construction,  for any $l\geq N$ there is at least one strict inequality. % Indeed there exists $l\in \{1,\dots,N\}$ such that $j_{l}=j_{m}$ for $m<l$, i.e. there exist a cycle $\sigma=(j_{m},\dots,j_{l-1})$, so there is $d\in \{m,\dots,l\}$ such that $j_{d}\geq j_{d-1}$, therefore $\lambda_{d,j_{d}}> \lambda_{d-1,j_{d-1}}$.\\
%Hence ordering the set $\{\nu^{p_{1}}(s_{1,1}),\dots,\nu^{p_{N}}(s_{N,r_{N}})\}$, we get after a number finite $l$ of step that $t_{l}=a_{l}s_{l}$, i.e.
%$$
%t_{0}=\sum_{i=1}^{l}a_{i}s_{i}
%$$
%which concludes the proof.
\end{proof}
\begin{prop}
\label{prop:Homo}
Let $L$ be a big line bundle. Then ${\Delta_{j}(mL)=m\Delta_{j}(L)}$ and $\Delta_{j}^{W}(mL)=m\Delta_{j}^{W}(L)$ for any $m\in\mathbbm{N}$ and for any $j=1,\dots,N$ where $\Delta_{j}^{W}(L)$ is the Okounkov body associated to the additive semigroup $\Gamma_{W,j}(L)$.
\end{prop}
\begin{proof}
The proof proceeds similarly to the proof of Proposition $4.1.ii$ in \cite{LM09}, exploiting again the property of the total order on $\mathbbm{Z}^{n}$.\\
We may assume $\Delta_{j}(L)\neq\emptyset $, otherwise it would be trivial, and we can choose $r,t\in\mathbbm{N}$ such that $V_{r,j},V_{tm-r,j}\neq \emptyset$, i.e. there exist sections $e\in V_{r,j}$ and $f\in V_{tm-r,j}$. Thus we get the inclusions
$$
k\Gamma_{j}(mL)^{k}+\nu^{p_{j}}(e)+\nu^{p_{j}}(f)\subset (km+r)\Gamma_{j}(L)^{km+r}+\nu^{p_{j}}(f)\subset (k+t)\Gamma_{j}(mL)^{k+t}.
$$
Letting $k\to\infty$, we find $ \Delta_{j}(mL)\subset m\Delta_{j}(L)\subset \Delta_{j}(mL) $.\\
The same proof works for $\Delta_{j}^{W}(L)$.
\end{proof}
Proposition \ref{prop:Homo} naturally extends the definition of the multipoint Okounkov bodies to $\mathbbm{Q}$-line bundles.\\
We are now ready to prove Theorem \ref{ThmA}.
\begin{proof}[Proof of Theorem \ref{ThmA}]
By Lemma \ref{prop:Vol} and Theorem \ref{thm:AboutVolume} we get
\begin{equation}
\label{eqn:Vol}
n!\sum_{j=1}^{N}\frac{\mathrm{Vol}_{\mathbbm{R}^{n}}(\Delta_{j}^{W}(L))}{\mathrm{ind}_{1,j}(L)\mathrm{ind}_{2,j}(L)^{n}}=\lim_{k\in\mathbbm{N}(L),k\to\infty}\frac{n!\sum_{j=1}^{N}\#\Gamma_{W,j}^{k}}{k^{n}}=\lim_{k\in \mathbbm{N}(L), k\to \infty}\frac{h^{0}(X,kL)}{k^{n}/n!}=\mathrm{Vol}_{X}(L).
\end{equation}
where we keep the same notations of Theorem \ref{thm:AboutVolume} for the indexes $\mathrm{ind}_{1,j}(L), \mathrm{ind}_{2,j}(L)$ adding the $j$ subscript to keep track of the points and the dependence on the line bundle since we want to perturb it.\\
\emph{\textbf{Key point:}} We claim that
\begin{equation}
\label{eqn:Inter}
\Delta_{j}^{W}(L)^{\circ}=\Delta_{j}(L)^{\circ},
\end{equation}
for any ${j=1,\dots,N}$. Note that since $\Gamma_{V,j}\subset \Gamma_{W,j}$ we only need to prove that $\Delta_{j}^{W}(L)^{\circ}\subset \Delta_{j}(L)^{\circ}$.\\
Let $A$ be a fixed ample line bundle $A$ such that there exist ${s_{1},\dots,s_{N}\in H^{0}(X,A)}$ with $s_{i}\in V_{1,i}(A)$ and $\nu^{p_{i}}(s_{i})=0$. Thus we get ${\Delta_{j}^{W}(mL-A)\subset \Delta_{j}(mL)}$ for each $m\in\mathbbm{N}$ and for any $j=1,\dots,N$ since $s\otimes s_{j}^{k}\in V_{k,j}(mL)$ for any $s\in W_{k,j}(mL-A)$. Hence
\begin{equation}
\label{eqn:FirstEqn}
\Delta_{j}^{W}\Big(L-\frac{1}{m}A\Big)\subset \Delta_{j}(L)\subset \Delta_{j}^{W}(L)
\end{equation}
by Proposition \ref{prop:Homo}.\\
Moreover since $m\to \mathrm{ind}_{1,j}(L-\frac{1}{m}A)$ and $m\to \mathrm{ind}_{2,j}(L-\frac{1}{m}A)$ are decreasing functions, (\ref{eqn:Vol}) implies
\begin{multline}
\label{eqn:FinallyA}
\limsup_{m\to \infty} n! \sum_{j=1}^{N}\frac{\mathrm{Vol}_{\mathbbm{R}^{n}}\big(\Delta_{j}^{W}(L-\frac{1}{m}A)\big)}{\mathrm{ind}_{1,j}(L)\mathrm{ind}_{2,j}(L)^{n}}\geq \limsup_{m\to \infty}n! \sum_{j=1}^{N}\frac{\mathrm{Vol}_{\mathbbm{R}^{n}}\big(\Delta_{j}^{W}(L-\frac{1}{m}A)\big)}{\mathrm{ind}_{1,j}(L-\frac{1}{m}A)\mathrm{ind}_{2,j}(L-\frac{1}{m}A)^{n}}=\\
=\limsup_{m\to \infty}\mathrm{Vol}_{X}\Big(L-\frac{1}{m}A\Big)=\mathrm{Vol}_{X}(L)=n! \sum_{j=1}^{N}\frac{\mathrm{Vol}_{\mathbbm{R}^{n}}\big(\Delta_{j}^{W}(L)\big)}{\mathrm{ind}_{1,j}(L)\mathrm{ind}_{2,j}(L)^{n}}
\end{multline}
where we used the continuity of the volume function on line bundles. Thus since $\Delta_{j}^{W}(L-\frac{1}{m}A)\subset\Delta_{j}^{W}(L-\frac{1}{l}A)$ if $l>m$ for any $j=1,\dots,N$, from (\ref{eqn:FinallyA}) we deduce that $m\to \mathrm{Vol}_{\mathbbm{R}^{n}}(\Delta_{j}^{W}(L-\frac{1}{m}A))$ is a continuous increasing function converging to $\mathrm{Vol}_{\mathbbm{R}^{n}}(\Delta_{j}^{W}(L))$ for any $j=1,\dots,N$. Hence (\ref{eqn:Inter}) follows from (\ref{eqn:FirstEqn}).\\
\emph{\textbf{Conclusion.}} Finally combining (\ref{eqn:Inter}) and Lemma \ref{lem:IntOk}$.(ii)$ we find out that $\mathrm{ind}_{1,j}(L)=\mathrm{ind}_{2,j}(L)=1$ if $\mathrm{Vol}_{\mathbbm{R}^{n}}\big(\Delta_{j}^{W}(L)\big)\neq 0$. Finally (\ref{eqn:Vol}) yields
$$
n!\sum_{j=1}^{N}\mathrm{Vol}_{\mathbbm{R}^{n}}\big(\Delta_{j}(L)\big)=n!\sum_{j=1}^{N}\frac{\mathrm{Vol}_{\mathbbm{R}^{n}}\big(\Delta_{j}^{W}(L)\big)}{\mathrm{ind}_{1,j}(L)\mathrm{ind}_{2,j}(L)^{n}}=\mathrm{Vol}_{X}(L),
$$
which concludes the proof.
\end{proof}
%\begin{rem}
%\emph{During the proof we have showed that the unique case when $\Delta_{j}(\Gamma_{W,j})$ and $\Delta_{j}(L)$ are not necessarly equal is when $\Delta_{j}(\Gamma_{W,j})^{\circ}=\emptyset$ but it is not empty while $\Delta_{j}(L)=\emptyset$. This actually may happen as the example of the Remark \ref{rem:Interior} for $a=1$ shows.}
%\end{rem}
\subsection{Variation of multipoint Okounkov bodies}
\label{subsection:Variation}
Similarly to the section $\S 4$ in \cite{LM09}, we prove that for fixed faithful valuations $\nu^{p_{j}}$ centered a $N$ different points the construction of the multipoint Okounkov Bodies is a numerical invariant, i.e. $\Delta_{j}(L)$ depends only from the first Chern class $c_{1}(L)\in\mathrm{N}^{1}(X)$ of the big line bundle $L$, where we have denoted by $\mathrm{N}^{1}(X)$ the Neron-Severi group. Recall that $\rho(X):=\dim \mathrm{N}^{1}(X)_{\mathbbm{R}}<\infty$ where $\mathrm{N}^{1}(X)_{\mathbbm{R}}:=\mathrm{N}^{1}(X)\otimes_{\mathbbm{Z}} \mathbbm{R}$.
\begin{prop}
\label{prop:cohomologicalconstruction}
Let $L$ be a big line bundle. Then $\Delta_{j}(L)$ is a numerical invariant.
\end{prop}
\begin{proof}
Assume $\Delta_{j}(L)^{\circ}\neq \emptyset$, which by Lemma \ref{lem:IntOk} is equivalent to $\Delta_{j}(L)\neq \emptyset$, and let $L'$ such that $L'=L+P$ for $P$ numerically trivial. Fix also an ample line bundle $A$. Then for any $m\in\mathbbm{N}$ there exist $k_{m}\in\mathbbm{N}$ and ${s_{m}\in H^{0}(X,k_{m}m(P+\frac{1}{m}A))}$ such that $s_{m}(p_{i})\neq 0$ for any $i=1,\dots, N$ since $P+\frac{1}{m}A$ is a ample $\mathbbm{Q}-$line bundle. Hence we get $\Delta_{j}(L)\subset \Delta_{j}(L'+\frac{1}{m}A)$ by homogeneity (Proposition \ref{prop:Homo}) because $s\otimes s_{m}^{k}\in V_{k,j}(k_{m}mL'+k_{m}A)$ for any section $s\in V_{k,j}(k_{m}mL)$. Therefore similarly to the proof of Theorem \ref{ThmA}, letting $m\to \infty$, we obtain $\Delta_{j}(L)\subset \Delta_{j}(L')$. Replacing $L$ by $L+P$ and $P$ by $-P$, Lemma \ref{lem:IntOk} concludes the proof.
\end{proof}
%\begin{rem}
%\emph{In the Proposition just showed the bigness of $L$ is necessary because for $N=1$ there exist counterexamples, see Remark $3.13.$ in \cite{CHPW16}, but it would be interesting to understand if it holds for any $L$ big and any choice of the points (i.e. if $\Delta_{j}(L)$ is a numerical invariant for $L$ big even if $\Delta_{j}(L)^{\circ}=\emptyset$).}
%\end{rem}
Setting $r:=\rho(X)$ for simplicity, fix $L_{1},\dots, L_{r}$ line bundles such that $\{c_{1}(L_{1}),\dots, c_{1}(L_{r})\}$ is a $\mathbbm{Z-}$basis of $\mathrm{N}^{1}(X)$: this lead to natural identifications $\mathrm{N}^{1}(X)\simeq \mathbbm{Z}^{r}$, $\mathrm{N}^{1}(X)_{\mathbbm{\mathbbm{R}}}\simeq\mathbbm{R}^{r}$. Moreover by Lemma $4.6.$ in \cite{LM09} we may choose $L_{1},\dots,L_{r}$ such that the pseudoeffective cone is contained in in the positive orthant of $\mathbbm{R}^{r}$.
\begin{defn}
Letting
\begin{equation*}
\Gamma_{j}(X):=\Gamma_{j}(X;L_{1},\dots,L_{r}):=\{(\nu^{p_{j}}(s),\vec{m}) \, : \, s\in V_{\vec{m},j}(L_{1},\dots, L_{r}))\setminus\{0\},\\ \vec{m}\in\mathbbm{N}^{r}\}\subset \mathbbm{Z}^{n}\times \mathbbm{N}^{r}
\end{equation*}
be the \emph{global multipoint semigroup of $X$ at $p_{j}$ with $p_{1},\dots\hat{p_{j}},\dots,p_{N}$ fixed} (it is an addittive subsemigroup of $\mathbbm{Z}^{n+r}$) where $V_{\vec{m},j}(L_{1},\dots,L_{r}):=\{s\in H^{0}(X,\vec{m}\cdot(L_{1},\dots,L_{r}))\setminus\{0\} \, : \, \nu^{p_{j}}(s)<\nu^{p_{i}}(s) \, \mathrm{for}\, \mathrm{any} \, i\neq j\}$, we define
$$
\Delta_{j}(X):=C(\Gamma_{j}(X))
$$
as the closed convex cone in $\mathbbm{R}^{n+r}$ generated by $\Gamma_{j}(X)$, and call it the \textbf{global multipoint Okounkov body} at $p_{j}$.
\end{defn}
\begin{lem}
The semigroup $\Gamma_{j}(X)$ generates a subgroup of $\mathbbm{Z}^{n+r}$ of maximal rank.
\end{lem}
\begin{proof}
Since the ample cone $\mathrm{Amp}(X)$ is open non-empty set in $\mathrm{N}^{1}(X)_{\mathbbm{R}}$, we can fix $F_{1},\dots, F_{r}$ ample line bundles generating $\mathrm{N}^{1}(X)$ as free $\mathbbm{Z}-$module. Moreover, by the assumptions done for $L_{1},\dots,L_{r}$ we know that for every $i=1,\dots,r$ there exists $\vec{a}_{i}$ such that $F_{i}=\vec{a}_{i}\cdot (L_{1},\dots,L_{r})$. Thus, for any $i=1,\dots,r$, the graded semigroup $\Gamma_{j}(F_{i})$ sits in $\Gamma_{j}(X)$ in a natural way and it generates a subgroup of $\mathbbm{Z}^{n}\times \mathbbm{Z}\cdot \vec{a}_{i}$ of maximal rank by point $ii)$ in Lemma \ref{lem:IntOk} since $\mathbbm{B}_{+}(F_{i})=\emptyset$. We conclude observing that $\vec{a}_{1},\dots,\vec{a}_{r}$ span $\mathbbm{Z}^{r}$.
\end{proof}
Next we need a further fact about additive semigroups and their cones. Let $\Gamma\subset \mathbbm{Z}^{n}\times\mathbbm{N}^{r}$ be an additive semigroup, and let $C(\Gamma)\subset\mathbbm{R}^{n}\times\mathbbm{R}^{r}$ be the closed convex cone generated by $\Gamma$. We call the \emph{support} of $\Gamma$ respect to the last $r$ coordinates, $\mathrm{Supp}(\Gamma)$, the closed convex cone $ C(\pi(\Gamma))\subset \mathbbm{R}^{r} $ where $\pi:\mathbbm{R}^{n}\times\mathbbm{R}^{r}\to\mathbbm{R}^{r}$ is the usual projection. Then, given $\vec{a}\in\mathbbm{N}^{r}$, we set $\Gamma_{\mathbbm{N}\vec{a}}:=\Gamma\cap(\mathbbm{Z}^{n}\times \mathbbm{N}\vec{a})$ and denote by $C(\Gamma_{\mathbbm{N}\vec{a}})\subset\mathbbm{R}^{n}\times\mathbbm{R}\vec{a}$ the closed convex cone generated by $\Gamma_{\mathbbm{N}\vec{a}}$ when we consider it as an additive semigroup of $\mathbbm{Z}^{n}\times\mathbbm{Z}\vec{a}\simeq\mathbbm{Z}^{n+1}$.
\begin{prop}[\cite{LM09}, Proposition $4.9.$]
\label{prop:FibreSemigroupJ}
Assume that $\Gamma$ generates a subgroup of finite index in $\mathbbm{Z}^{n}\times\mathbbm{Z}^{r}$, and let $\vec{a}\in\mathbbm{N}^{r}$ be a vector lying in the interior of $\mathrm{Supp}(\Gamma)$. Then
$$
C(\Gamma_{\mathbbm{N}\vec{a}})=C(\Gamma)\cap(\mathbbm{R}^{n}\times\mathbbm{R}\vec{a})
$$
\end{prop}
Now we are ready to prove the main theorem of this section:
\begin{thm}
\label{thm:GlobalMOBJ}
The global multipoint Okounkov body $\Delta_{j}(X)$ is characterized by the property that in the following diagram
$$
\begin{tikzcd}[column sep=
tiny, row sep=large
]
\Delta_{j}(X) \arrow[dr] 
& \subset & \mathbbm{R}^{n}\times\mathbbm{R}^{r}\simeq\mathbbm{R}^{n}\times \mathrm{N}^{1}(X)_{\mathbbm{R}} \arrow[dl, "\mathrm{pr}_{2}"] \\
& \mathbbm{R}^{r}\simeq\mathrm{N}^{1}(X)_{\mathbbm{R}}
\end{tikzcd}
$$
the fiber of $\Delta_{j}(X) $ over any cohomology class $c_{1}(L)$ of a big $\mathbbm{Q}-$line bundle $L$ such that $c_{1}(L)\in \mathrm{Supp}(\Gamma_{j}(X))^{\circ}$ is the multipoint Okounkov body associated to $L$ at $p_{j}$, i.e $\Delta_{j}(X)\cap \mathrm{pr}_{2}^{-1}(c_{1}(L))=\Delta_{j}(L)$. Moreover $\mathrm{Supp}\big(\Gamma_{j}(X)\big)^{\circ}\cap N^{1}(X)_{\mathbbm{Q}}=\{c_{1}(L)\, : \, \Delta_{j}(L)\neq \emptyset,\, L \, \mathbbm{Q}\mathrm{-line}\,\mathrm{bundle}\}$.
\end{thm}
\begin{rem}
\label{rem:GlobalMOBJRem}
\emph{%By Proposition $1.2.$ in \cite{KL15a} the set $B_{+}(p_{j}):=\{\{\alpha\in \mathrm{N}^{1}(X)_{\mathbbm{R}}\, : \, p\in\mathbbm{B}_{+}(\alpha)\}$ is a closed set respect to the metric topology on $N^{1}(X)_{\mathbbm{R}}$, and moreover by Corollary an easy consequence of Proposition 1.5. in \cite{ELMNP06} is that for any $p\in X$ the open (respectively closed) set $\big(B_{+}(p)\big)^{C}$.
It is unclear how $\mathrm{Supp}(\Gamma_{j}(X))^{\circ}$ can be described. By second point in Lemma \ref{lem:IntOk}, it contains the open convex set $B_{+}(p_{j})^{C}$ where $B_{+}(p_{j}):=\{\alpha\in \mathrm{N}^{1}(X)_{\mathbbm{R}}\, : \, p\in\mathbbm{B}_{+}(\alpha)\}$ is closed respect to the metric topology on $N^{1}(X)_{\mathbbm{R}}$ by Proposition $1.2.$ in \cite{KL15a} and its complement is convex as easy consequence of Proposition 1.5. in \cite{ELMNP06}. But in general $\mathrm{Supp}(\Gamma_{j}(X))^{\circ}$ may be bigger: for instance if $N=1$ $\mathrm{Supp}(\Gamma_{j}(X))^{\circ}$ coincides with the big cone, and we can easily construct an example with $p_{1}, p_{2}\in \mathbbm{B}_{-}(L)$ and $\Delta_{j}(L)^{\circ}\neq \emptyset$ for $j=1,2$. For instance take $X=\mathrm{Bl}_{q}\mathbbm{P}^{2}$, $L:=H+E$ where $E$ is the exceptional divisor and $p_{1},p_{2}\in \mathrm{Supp}(E)$ different points. Then given two valuations associated to admissible flags $Y_{\cdot,j}$ for $j=1,2$ centered at $p_{1},p_{2}$ such that $Y_{1,j}=E$ for any $j=1,2$, it is easy to check that $\Delta_{j}(L)^{\circ}\neq \emptyset$ for $j=1,2$ where by Lemma \ref{lem:IntOk} this is equivalent to $\Delta_{j}(L)\neq\emptyset$.}%. Indeed it depends on the number of points chosen and on their position, but we do not know if $\mathrm{Supp}(\Gamma_{j}(X))$ depends on the choice of the family of valuations.}
\end{rem}
\begin{proof}
For any vector $\vec{a}\in\mathbbm{N}^{r}$ such that $L:=\vec{a}\cdot(L_{1},\dots,L_{r})$ is a big line bundle in $\mathrm{Supp}(\Gamma_{j}(X))^{\circ}$, we get $\Gamma_{j}(X)_{\mathbbm{N}\vec{a}}=\Gamma_{j}(L)$, and so the base of the cone $C(\Gamma_{j}(X)_{\mathbbm{N}\vec{a}})=C(\Gamma_{j}(L))\subset\mathbbm{R}^{n}\times\mathbbm{R}\vec{a}$ is the multipoint Okounkov body $\Delta_{j}(L)$, i.e.
$$
\Delta_{j}(L)=\pi\Big(C(\Gamma_{j}(X)_{\mathbbm{N}\vec{a}})\cap\big(\mathbbm{R}^{n}\times\{1\}\big)\Big).
$$
Then Proposition \ref{prop:FibreSemigroupJ} implies that the right side of the last equality coincides with the fiber $\Delta_{j}(X)$ over $c_{1}(L)$. Both sides rescale linearly, so the equality extends to $\mathbbm{Q}$-line bundles.\\
Next by Lemma \ref{lem:IntOk} it follows that $c_{1}(L)\in \mathrm{Supp}\big(\Gamma_{j}(X)\big)$ for any $\mathbbm{Q}$-line bundle $L$ such that $\Delta_{j}(L)\neq\emptyset$. On the other hand, by the first part of the proof we get
\begin{equation}
\label{eqn:Ref}
\mathrm{Supp}\big(\Gamma_{j}(X)\big)^{\circ}\cap \mathrm{N}^{1}(X)_{\mathbbm{Q}}\subset \{c_{1}(L)\, : \Delta_{j}(L)\neq \emptyset, L \,\, \mathbbm{Q}\mathrm{-line}\, \mathrm{bundle}\}.
\end{equation}
Thus it remains to prove that the right hand in $(\ref{eqn:Ref})$ is open in $\mathrm{N}^{1}(X)_{\mathbbm{Q}}$, which is equivalent to show that $\Delta_{j}(L-\frac{1}{k}A)\neq \emptyset$ for $k\gg 1$ big enough if $A$ is a fixed very ample line bundle since the ample cone is open and not empty in $\mathrm{N}^{1}(X)_{\mathbbm{R}}$ and $\mathrm{N}^{1}(X)_{\mathbbm{R}}$ is a finite dimensional vector space. Considering the multiplication by a section $s\in H^{0}(X,A)$ such that $s(p_{i})\neq 0$ for any $i=1,\dots, N$, we obtain $\Delta_{i}(L-\frac{1}{k}A)\subset \Delta_{i}(L)$ for any $i=1,\dots,N$. Therefore by Theorem \ref{ThmA} and Lemma \ref{lem:IntOk} we necessarily have $\Delta_{j}(L-\frac{1}{k}A)^{\circ}\neq \emptyset$ for $k\gg1$ big enough since $\mathrm{Vol}_{X}(L-\frac{1}{k}A)\nearrow \mathrm{Vol}_{X}(L)$ and $\Delta_{j}(L)^{\circ}\neq \emptyset$. This concludes the proof.
\end{proof}
As a consequence of Theorem \ref{thm:GlobalMOBJ}, we can extend the definition of multipoint Okounkov bodies to $\mathbbm{R}$-line bundles. Indeed we can define $\Delta_{j}(L)$ as the limit (in the Hausdorff sense) of $\Delta_{j}(L_{k})$ if $c_{1}(L)\in\mathrm{Supp}(\Gamma_{j}(X))^{\circ}=\overline{\{c_{1}(L)\, : \, \Delta_{j}(L)\neq \emptyset, L \, \mathbbm{Q}\mathrm{-line}\, \mathrm{bundle}\}}^{\circ}$ where $\{L_{k}\}_{k\in\mathbbm{N}}$ is any sequence of $\mathbbm{Q}$-line bundles such that $c_{1}(L_{k})\to c_{1}(L)$, and $\Delta_{j}(L)=\emptyset$ otherwise. This extension is well-defined and coherent with Lemma \ref{lem:IntOk}, since we obtain $\Delta_{j}(L)^{\circ}\neq \emptyset$ iff $\Delta_{j}(L)\neq\emptyset$.
\begin{cor}
The function $\mathrm{Vol}_{\mathrm{R}^{n}}:\mathrm{Supp}(\Gamma_{j}(X))^{\circ}\to \mathbbm{R}_{>0}$, $c_{1}(L)\to \mathrm{Vol}_{\mathrm{R}^{n}}(\Delta_{j}(L))$ is well-defined, continuous, homogeneous of degree $n$ and log-concave, i.e.
$$
\mathrm{Vol}_{\mathrm{R}^{n}}(\Delta_{j}(L+L'))^{1/n}\geq\mathrm{Vol}_{\mathrm{R}^{n}}(\Delta_{j}(L))^{1/n}+\mathrm{Vol}_{\mathrm{R}^{n}}(\Delta_{j}(L'))^{1/n}
$$
\end{cor}
\begin{proof}
The fact that it is well-defined and its homogeneity follow directly from Propositions \ref{prop:Homo} and \ref{prop:cohomologicalconstruction}, while the other statements are standard in convex geometry, using the Brunn-Minkowski Theorem and Theorem \ref{thm:GlobalMOBJ}.
\end{proof}
Finally we note that the Theorem \ref{thm:GlobalMOBJ} provides a description of the multipoint Okounkov bodies similarly to that of Theorem \ref{thm:OneOko}.(iii):
\begin{cor}
\label{cor:AnotherEqMOB}
If $L=\mathcal{O}_{X}(D)$ is a big line bundle such that $c_{1}(L)\in\mathrm{Supp}(\Gamma_{j}(X))^{\circ}$, then
$$
\Delta_{j}(L)=\overline{\nu^{p_{j}}\{ D'\in \mathrm{Div}_{\geq 0}(X)_{\mathbbm{R}} :  D'\equiv_{num}D \, \mathrm{and} \,\nu^{p_{j}}(D')<\nu^{p_{i}}(D') \, \forall i\neq j\}}.
$$
In particular every rational point in $\Delta_{j}(L)^{\circ}$ is valuative and if it contains a small $n$-symplex with valuative vertices then any rational point in the $n$-symplex is valuative.
\end{cor}
\begin{proof}
The first part follows directly from Theorem \ref{thm:GlobalMOBJ} since $
{D'\equiv_{num} D}$ iff $c_{1}(L)=c_{1}(\mathcal{O}_{X}(D'))$ by definition (considering the $\mathbbm{R}-$line bundle $\mathcal{O}_{X}(D')$). The second part about $\Delta_{j}(L)^{\circ}$ follows combining Lemma \ref{lem:IntOk}$.(iii)$ and the multiplicative property of $\nu^{p_{j}}$ with Theorem \ref{thm:KKh12F} (see also the proof of Proposition \ref{prop:First}).
\end{proof}
\subsection{Geometry of multipoint Okounkov bodies}
\label{subsection:Consequence}
To investigate the geometry of the multipoint Okounkov bodies we need to introduce the following important invariant:
\begin{defn}
Let $L$ be a line bundle, $V\subset X$ a subvariety of dimension $d$ and $H^{0}(X\vert V,kL):=\mathrm{Im}\Big(H^{0}(X,kL)\to H^{0}(V,kL_{\vert V})\Big)$. Then the quantity
$$
\mathrm{Vol}_{X\vert V}(L):=\limsup_{k\to\infty} \frac{\dim H^{0}(X\vert V,kL)}{k^{d}/d!}
$$
is called the \textbf{restricted volume of $L$ along $V$}.
\end{defn}
We refer to \cite{ELMNP09} and reference therein for the theory about this new object.\\
In the repeatedly quoted paper \cite{LM09}, given a valuation $\nu^{p}(s)=(\nu^{p}(s)_{1},\dots,\nu^{p}(s)_{n})$ associated to an admissible flag $Y_{\cdot}=(Y_{1},\dots,Y_{n})$ such that $Y_{1}=D$ and a line bundle $L$ such that $D\not\subset \mathbbm{B}_{+}(L)$, the authors also defined the one-point Okounkov body of the graded linear sistem $H^{0}(X\vert D, kL)\subset H^{0}(D,kL_{\vert D})$ by
$$
\Delta_{X\vert D}(L):=\Delta(\Gamma_{X\vert D})
$$
with $\Gamma_{X\vert D}:=\{(\nu^{p}(s)_{2},\dots,\nu^{p}(s)_{n},k)\in \mathbbm{N}^{n-1}\times\mathbbm{N}: {s\in H^{0}(X\vert D, kL)\setminus\{0\}},{k\geq 1\}}$ and they proved the following result.
\begin{thm}[\cite{LM09}, Theorem $4.24$, Corollary $4.25$]
\label{thm:Diff}
Let $D\not\subset\mathbbm{B}_{+}(L)$ be a prime divisor with $L$ big $\mathbbm{R}-$line bundle and let $Y.$ be an admissible flag such that $Y_{1}=:D$. Let $C_{\max}:=\sup\{\lambda\geq 0 \, : \, L-\lambda D \, \mathrm{is} \, \mathrm{big}\}$. Then for any $0\leq t< C_{max}$
\begin{gather*}
\Delta(L)_{x_{1}\geq t}=\Delta(L-tD)+t \vec{e}_{1}\\
\Delta(L)_{x_{1}=t}=\Delta_{X\vert D}(L-tD)
\end{gather*}
Moreover
\begin{itemize}
\item[i)] $\mathrm{Vol}_{\mathrm{R}^{n-1}}(\Delta(L)_{x_{1}=t})=\frac{1}{(n-1)!}\mathrm{Vol}_{X\vert D}(L-tD)$;
\item[ii)] $ \mathrm{Vol}_{X}(L)-\mathrm{Vol}_{X}(L-tD)=n\int_{0}^{t}\mathrm{Vol}_{X\vert D}(L-\lambda D)d\lambda $;
%\item[iii)] the function $t\to \mathrm{Vol}_{X}(L-tE)$ is differentiable at $t=0$, and
%$$
%\frac{d}{dt}\Big(\mathrm{Vol}_{X}(L-tE)\Big)_{\vert t=0}=n\mathrm{Vol}_{X\vert E}(L)
%$$
\end{itemize}
\end{thm}
%We refer to \cite{ELMNP09} for the theory of the restricted volumes of the points $i)$ and $ii)$.\\
In this section we suppose to have fixed a family of valuations $\nu^{p_{j}}$ associated to a family of admissible flags $Y.=(Y_{\cdot,1},\dots,Y_{\cdot,N})$ on a projective manifold $X$, centered respectively in $p_{1},\dots,p_{N}$ (see paragraph \ref{paragraph:ParticularValuations} and Remark \ref{rem:ParticularValuations}).
Given a big line bundle $L$, and prime divisors $D_{1},\dots,D_{N}$ where $D_{j}=Y_{1,j}$ for any $j=1,\dots,N$, we set
$$
\mu(L;\mathbbm{D}):=\sup\{t\geq 0 \, : \, L-t\mathbbm{D} \, \mathrm{is} \, \mathrm{big}\}
$$
where $\mathbbm{D}:=\sum_{i=1}^{N}D_{i}$, and
$$
\mu(L;D_{j}):=\sup\{t\geq 0 \, : \, \Delta_{j}(L-t\mathbbm{D})^{\circ}\neq \emptyset\}.
$$
\begin{thm}
\label{thm:Slices}
Let $L$ a big $\mathbbm{R}-$line bundle, $\nu^{p_{j}}$ a family of valuations associated to a family of admissible flags $Y.$ centered at $p_{1},\dots,p_{N}$. Then, letting $(x_{1},\dots,x_{n})$ be fixed coordinates on $\mathbbm{R}^{n}$, for any ${j\in\{1,\dots,N\}}$ such that $\Delta_{j}(L)^{\circ}\neq\emptyset$ the followings hold:
\begin{itemize}
\item[i)] $\Delta_{j}(L)_{x_{1}\geq t}=\Delta_{j}(L-t\mathbbm{D})+ t\vec{e}_{1}$ for any $0\leq t < \mu(L;D_{j})$, for any $j=1,\dots,N$;
\item[ii)] $ \Delta_{j}(L)_{x_{1}=t}=\Delta_{X\vert D_{j}}(L-t\mathbbm{D}) $ for any $0\leq t<\mu(L;\mathbbm{D})$, $t\neq \mu(L;D_{j})$ and for any $j=1,\dots,N$;
\item[iii)] $ \mathrm{Vol}_{\mathbbm{R}^{n-1}}(\Delta_{j}(L)_{x_{j}=t})=\frac{1}{(n-1)!}\mathrm{Vol}_{X\vert D_{j}}(L-t\mathbbm{D}) $ for any ${0\leq t<\mu(L;\mathbbm{D})}$, $t\neq \mu(L;D_{j})$ for any $j=1,\dots,N$, and in particular $\mu(L;D_{j})=\sup\{t\geq 0\, : D_{j}\not\subset \mathbbm{B}_{+}(L-t\mathbbm{D})\}$.
\end{itemize}
Moreover
\begin{itemize}
\item[iv)] $ \mathrm{Vol}_{X}(L)-\mathrm{Vol}_{X}(L-t\mathbbm{D})=n\int_{0}^{t}\sum_{i=1}^{N}\mathrm{Vol}_{X\vert D_{i}}\Big(L-\lambda\mathbbm{D}\Big)d\lambda $ for any $0\leq t<\mu(L;\mathbbm{D})$.
\end{itemize}
\end{thm}
\begin{proof}
\emph{\textbf{Proof of (i).}} The first point follows as in Proposition $4.1.$ in \cite{LM09}, noting that if $L$ is a big line bundle and $0\leq t<\mu(L;D_{j})$ integer then $\{s\in V_{k,j}(L) \, : \, \nu^{p_{j}}(s)_{1}\geq kt\}\simeq V_{k,j}(L-t\mathbbm{D})$ for any $k\geq 1$. Therefore $\Gamma_{j}(L)_{x_{1}\geq t} =\varphi_{t}(\Gamma_{j}(L-t\mathbbm{D}))$ where $\varphi_{t}:\mathbbm{N}^{n}\times \mathbbm{N}\to \mathbbm{N}^{n}\times\mathbbm{N}$ is given by $\varphi_{t}(\vec{x},k):=(\vec{x}+tk\vec{e}_{1},k)$. Passing to the cones we get $C(\Gamma_{j}(L)_{x_{1}\geq t})=\varphi_{t,\mathbbm{R}}\big(C(\Gamma_{j}(L-t\mathbbm{D}))\big)$ where $\varphi_{t,\mathbbm{R}}$ is the linear map between vector spaces associated to $\varphi_{t}$. Hence, taking the base of the cones, the equality $\Delta_{j}(L)_{x_{1}\geq t}=\Delta_{j}(L-t\mathbbm{D})+t\vec{e}_{1}$ follows. Finally, since both sides in $i)$ rescale linearly by Proposition \ref{prop:Homo}, the equality holds for any $L$ $\mathbbm{Q}-$line bundle and $t\in\mathbbm{Q}$. Both sides in $(i)$ are clearly continuous in $t$ if ${0\leq t<\mu(L;D_{j})}$ so it remains to extend it to $\mathbbm{R}$-line bundles $L$. We fix a decreasing sequence of $\mathbbm{Q}$-line bundles $\{L_{k}\}_{k\in\mathbbm{N}}$ such that $L_{k}\searrow L$, where for decreasing we mean $L_{k}-L_{k+1}$ is an pseudoeffective line bundle and where the convergence is in the Neron-Severi space $\mathrm{N}^{1}(X)_{\mathbbm{R}}$. Then, as a consequence of Theorem \ref{thm:GlobalMOBJ} $0\leq t< \mu(L_{k},D_{j})$ for any $k\in\mathbbm{N}$ big enough where $t$ is fixed as in $(i)$, and $\{\Delta_{j}(L_{k})\}_{k\in\mathbbm{N}}$ continuously approximates $\Delta_{j}(L)$ in the Hausdorff sense. Hence we obtain $(i)$ letting $k\to \infty$.\\
\emph{\textbf{Proof of (ii).}} Assuming first $L$ $\mathbbm{Q}-$line bundle and ${0\leq t<\mu(L;D_{j})}$ rational.\\
We consider the additive semigroups
\begin{gather*}
\Gamma_{j,t}(L)=\{(\nu^{p_{j}}(s),k)\in\mathbbm{N}^{n}\times \mathbbm{N} \, : \, s\in V_{k,j}(L)\,\, \mathrm{and}\,\, \nu^{p_{j}}(s)_{1}=kt\}\\
\Gamma_{X\vert D_{j}}(L-t\mathbbm{D}):=\{(\nu^{p_{j}}(s)_{2},\dots,\nu^{p_{j}}(s)_{n},k)\in\mathbbm{N}^{n-1}\times\mathbbm{N} : \\ s\in H^{0}(X\vert D_{j},k(L-t\mathbbm{D}))\setminus\{0\}, k\geq 1\}
\end{gather*}
and, setting $\psi_{t}:\mathbbm{N}^{n-1}\times\mathbbm{N}\to\mathbbm{N}^{n}\times\mathbbm{N}$ as $\psi_{t}(\vec{x},k):=(kt,\vec{x},k)$, we easily get $\Gamma_{j,t}(L)\subset\psi_{t}\big(\Gamma_{X\vert D_{j}}(L-t\mathbbm{D})\big)$. Thus passing to the cones we have
$$
C(\Gamma_{j}(L))_{x_{1}=t}=C\big(\Gamma_{j,t}(L)\big)\subset\psi_{t,\mathbbm{R}}\Big(C\big(\Gamma_{X\vert D_{j}}(L-t\mathbbm{D})\big)\Big)
$$
where the equality follows from Proposition $A.1$ in \cite{LM09}. Hence $\Delta_{j}(L)_{x_{1}=t}\subset \Delta_{X|D_{j}}(L-t\mathbbm{D})$ for any $0\leq t<\mu(L;D_{j})$ rational. Moreover it is trivial that the same inclusion holds for any ${\mu(L;D_{j})<t<\mu(L;\mathbbm{D})}$.\\
Next let $0\leq t<\mu(L;\mathbbm{D})$ fixed and let $A$ be a fixed ample line bundle such that there exists $s_{j}\in V_{1,j}(A)$ with $\nu^{p_{j}}(s_{j})=\vec{0}$ and $\nu^{p_{i}}(s_{j})_{1}>0$ for any $i\neq j$. Thus since to any section ${s\in H^{0}(X\vert D_{j},k(L-t\mathbbm{D}))\setminus\{0\}}$ we can associate a section $\tilde{s}\in H^{0}(X,kL)$ with $\nu^{p_{j}}(\tilde{s})=(kt,\nu^{p_{j}}(s)_{2},\dots,\nu^{p_{j}}(s)_{n})$ and $\nu^{p_{i}}(\tilde{s})_{1}\geq kt$ for any $i\neq j$, we get that $\tilde{s}^{m}\otimes s_{j}^{k}\in V_{k,j}(mL+A)$ for any $m\in\mathbbm{N}$. By homogeneity this implies
$$
\frac{\nu^{p_{j}}(\tilde{s}^{m}\otimes s_{j}^{k})}{mk}=\frac{\nu^{p_{j}}(\tilde{s})}{k}=\Big(t,\frac{\nu^{p_{j}}(s)}{k}\Big)=:x\in \Delta_{j}\Big(L+\frac{1}{m}A\Big)_{x_{1}=t}
$$
for any $m\in\mathbbm{N}$. Hence since $\Delta_{j}(L)^{\circ}\neq \emptyset$ we get $0\leq t\leq \mu(L;D_{j})$ and $x\in \Delta_{j}(L)_{x_{1}=t}$ by the continuity of $m\to \Delta_{j}(L+\frac{1}{m}A)$ (Theorem \ref{thm:GlobalMOBJ}).\\
%using Theorem \ref{ThmA}, the continuity of the volume and the inclusion $\Delta_{j}(L+\frac{1}{m}A)\subset \Delta_{j}(L+\frac{1}{l}A)$ if $m>l$ we get that $l\to \mathrm{Vol}_{\mathbbm{R}^{n}}(\Delta_{j}(L+\frac{1}{m}A))$ is a continuous and decreasing function, which necessarily implies $t\leq \mu(L;E_{j})$ and $x\in \Delta_{j}(L)_{x_{1}=t}$ using the convexity of $\Delta_{j}(L)$ and the hypothesis $\Delta_{j}(L)^{\circ}\neq \emptyset$.
Summarizing we have showed that both sides of $ii)$ are empty if $\mu(L;D_{j})<t<\mu(L;\mathbbm{D})$ and that they coincides for any rational ${0\leq t<\mu(L;D_{j})}$. Moreover since by Theorem \ref{thm:Diff}
$$
\Delta_{X\vert D_{j}}(L-t\mathbbm{D})=\Delta\Big(L-t\sum_{i=1,i\neq j}^{N}D_{i}\Big)_{x_{1}=t}
$$
with respect to the valuation $\nu^{p_{j}}$, we can proceed similarly as in $(i)$ to extend the equality in $(ii)$ first to $t$ real and then to $\mathbbm{R}$-line bundles using the continuity derived from Theorem \ref{thm:GlobalMOBJ} and Theorem $4.5$ in \cite{LM09}.\\
\emph{\textbf{Proof of (iii), (iv).}} The third point is an immediate consequence of $ii)$ using Theorem \ref{thm:Diff}.i) and Theorem $A$ and $C$ in \cite{ELMNP09}, while last the point follows by integration using our Theorem \ref{ThmA}.
\end{proof}
We observe that Theorem \ref{thm:Slices} may be helpful when one fixes a big line bundle $L$ and a family of valuations associated to a family of infinitesimal flags centered at $p_{1},\dots,p_{N}\notin\mathbbm{B}_{+}(L)$. Indeed, similarly as stated in the paragraph $\S$ \ref{paragraph:ParticularValuations}, componing with $F:\mathbbm{R}^{n}\to\mathbbm{R}^{n}$, $F(x)=(\lvert x \rvert, x_{1},\dots,x_{n-1})$, Theorem \ref{thm:Slices} holds and in particular, for any $j=1,\dots,N$, we get
\begin{itemize}
\item[i)] $F\big(\Delta_{j}(L)\big)_{x_{j,1}\geq t}=\Delta_{j}\big(f^{*}L-t\mathbbm{E}\big) + t\vec{e}_{1}$ for any $0\leq t<\mu(f^{*}L;E_{j})$;
\item[ii)] $ F\big(\Delta_{j}(L)\big)_{x_{j,1}=t}=\Delta_{\tilde{X}\vert E_{j}}(f^{*}L-t\mathbbm{E}) $ for any $0\leq t<\mu(f^{*}L;\mathbbm{E})$, $t\neq\mu(f^{*}L;E_{j})$;
\item[iii)] $ \mathrm{Vol}_{\mathbbm{R}^{n-1}}\big(F(\Delta_{j}(L))_{x_{j,1}=t}\big)=\frac{1}{(n-1)!}\mathrm{Vol}_{\tilde{X}\vert E_{j}}(f^{*}L-t\mathbbm{E}) $ for any ${0\leq t<\mu(f^{*}L;\mathbbm{E})}$, $t\neq \mu(f^{*}L;E_{j})$;
\end{itemize}
where we have set $f:\tilde{X}\to X$ for the blow-up at $Z=\{p_{1},\dots,p_{N}\}$ and we have denoted with $E_{j}$ the exceptional divisors. Note that $\mathbbm{E}=\sum_{j=1}^{N}E_{j}$ and that the multipoint Okounkov body on the right side in $i)$ is calculated from the family of valutions $\{\tilde{\nu}^{\tilde{p}_{j}}\}_{j=1}^{N}$ (it is associated to the family of admissible flags on $\tilde{X}$ given by the family of infinitesimal flags on $X$).\\
This yields a new tool to study the \emph{multipoint Seshadri constant} as stated in the Introduction (see Theorem \ref{ThmB}). An application in the surfaces case is provided in subsection $\S$ \ref{subsection:Surfaces}.
\section{Kähler Packings}
\label{section:Packings}
Recall that the \emph{essential} multipoint Okounkov body is defined as
$$
\Delta_{j}(L)^{\mathrm{ess}}:=\bigcup_{k\geq 1} \Delta_{j}^{k}(L)^{\mathrm{ess}}=\bigcup_{k\geq 1}\Delta_{j}^{k!}(L)^{\mathrm{ess}}
$$
where $ \Delta_{j}^{k}(L)^{\mathrm{ess}}:=\mathrm{Conv}(\Gamma_{j}^{k})^{\mathrm{ess}}=\frac{1}{k}\mathrm{Conv}(\nu^{p_{j}}(V_{k,j}))^{\mathrm{ess}}$ is the interior of $ \Delta_{j}^{k}(L):=\mathrm{Conv}(\Gamma_{j}^{k})$ as subset of $\mathbbm{R}^{n}_{\geq 0}$ with its induced topology (see subsection $\S$ \ref{ssec:Semigroups}).\\% We recall that $ \Delta_{j}^{k}(L)^{ess}\subset  \Delta_{j}^{km}(L)^{ess}$ for any $m\in\mathbbm{N}$ and that $ \Delta_{j}(L)^{ess}$ is an open convex subset of $\mathbbm{R}^{n}_{\geq 0}$ described by $ \Delta_{j}(L)^{ess}=\bigcup_{k\geq 1} \Delta_{j}^{k!}(L)^{ess}$ with $ \Delta_{j}^{k!}(L)^{ess}$ non-decreasing sequence in $k$ of convex open subset of $\mathbbm{R}^{n}_{\geq 0}$ (see Corollary \ref{cor:Essential}).\\
Fix a family of local holomorphic coordinates $\{z_{j,1},\dots,z_{j,n}\}$ for ${j=1,\dots,N}$ respectively centered at $p_{1},\dots,p_{N}$ and assume that the faithful valuations $\nu^{p_{1}},\dots,\nu^{p_{N}}$ are quasi-monomial respect to the same additive total order $>$ on $\mathbbm{Z}^{n}$ and respect to the same vectors $\vec{\lambda}_{1},\dots,\vec{\lambda}_{n}\in \mathbbm{N}$ (see Remark \ref{rem:ParticularValuations}). Thus similarly to the Definition $2.7.$ in \cite{WN15}, we give the following
\begin{defn}
For every $j=1,\dots,N$ we define $\Omega_{j}(L):=\mu^{-1}(\Delta_{j}(L)^{\mathrm{ess}})$ for $\mu(w_{1},\dots,w_{n}):=(\lvert w_{1}\rvert^{2},\dots, \lvert w_{n}\rvert^{2} )$, and call them the \textbf{multipoint Okounkov domains}.
\end{defn}
Note that we get $n!\mathrm{Vol}_{\mathbbm{R}^{n}}(\Delta_{j}(L))=\mathrm{Vol}_{\mathbbm{C}^{n}}(\Omega_{j}(L))$ for any $j=1,\dots,N$ (see subsection $\S$ \ref{ssec:Moment}).\\
We will construct \emph{K\"ahler packings} (see Definition \ref{defn:KPAmple} and \ref{defn:KPBig}) of the multipoint Okounkov domains with the standard metric into $(X,L)$ for $L$ big line bundle. We will first address the ample case and then we will generalize to the big case in the subsection $\S$ \ref{subsection:BigCase}.
\subsection{Ample case}
\begin{defn}
\label{defn:KPAmple}
We say that a finite family of $n-$dimensional K\"ahler manifolds $ \{(M_{j},\eta_{j})\}_{j=1}^{N} $ packs into $ (X,L) $ for $L$ ample if for every family of relatively compact open set $ U_{j} \Subset M_{j} $ there are a holomorphic embedding $ f: \bigsqcup_{j=1}^{N}U_{j}\to X $ and a K\"ahler form $ \omega $ lying in $ c_{1}(L) $ such that $ f_{*}\eta_{j}=\omega_{\vert f(U_{j})} $. If, in addition,
$$
\sum_{j=1}^{N}\int_{M_{j}}\eta_{j}^{n}=\int_{X}c_{1}(L)^{n}
$$
then we say that $\{(M_{j},\eta_{j})\}_{j=1}^{N}$ packs perfectly into $(X,L)$.
\end{defn}
Letting $\mu:\mathbbm{C}^{n}\to\mathbbm{R}^{n}$ be the map $\mu(\mathbf{z_{j}}):=(\lvert z_{j,1}\rvert^{2},\dots,\lvert z_{j,n}\rvert^{2})$ where $\mathbf{z_{j}}=\{z_{j,1},\dots,z_{j,n}\}$ are usual coordinates on $\mathbbm{C}^{n}$ and letting
$$
\mathcal{D}_{k,j}:=\mu^{-1}(k\Delta_{j}^{k}(L))^{\circ}=\mu^{-1}(k\Delta_{j}^{k}(L)^{\mathrm{ess}}),
$$
we define $ M_{k,j} $ like the manifold we get by removing from $\mathbbm{C}^{n}$ all the submanifolds of the form $ \{z_{j,i_{1}}=\dots=z_{j,i_{m}}=0\} $ which do not intersect $\mathcal{D}_{k,j}$.\\
Thus
$$
\phi_{k,j}:=\ln \Big( \sum_{\alpha_{j}\in\nu^{p_{j}}(V_{k,j})}\lvert \mathbf{z_{j}}^{\mathbf{\alpha_{j}}}\rvert^{2} \Big)
$$
is a strictly plurisubharmonic function on $M_{k,j}$ and we denote by $\omega_{k,j}:=dd^{c}\phi^{k,j}$ the K\"ahler form associated (recall that $dd^{c}=\frac{i}{2\pi}\partial \bar{\partial}$, see subsection $\S$ \ref{ssec:SingMetr}).
\begin{lem}[\cite{And13}, Lemma 5.2.]
\label{lem:AndersonOrdering}
For any finite set $\mathcal{A}\subset \mathbbm{N}^{m}$ with a fixed additive total order $>$, there exists $\gamma\in (\mathbbm{N}_{>0})^{m}$ such that
$$
\alpha < \beta \quad \mathrm{iff}\quad \alpha\cdot \gamma < \beta \cdot \gamma
$$
for any $\alpha,\beta\in \mathcal{A}$.
\end{lem}
\begin{thm}
\label{thm:EmbeddingsGrandi}
If $L$ is ample then for $k>0$ big enough $ \{(M_{k,j},\omega_{k,j})\}_{j=1}^{N} $ packs into $(X,kL)$.
\end{thm}
Using the idea of the Theorem $A$ in \cite{WN15} we want to construct a Kähler metric on $kL$ such that locally around the points $p_{1},\dots,p_{N}$ approximates the metrics $\phi_{k,j}$ after a suitable \emph{zoom}. We observe that for any $\mathbf{\gamma}\in \mathbbm{N}^{n}$ and any section $s\in H^{0}(X,kL)$ with leading term $\mathbf{\alpha}\in\mathbbm{N}^{n}$ around a point $p\in X$ we have ${s(\tau^{\gamma_{1}}z_{1},\dots,\tau^{\gamma_{n}}z_{n})/\tau^{\gamma\cdot \alpha}\sim z_{1}^{\alpha_{1}}}\cdots z_{n}^{\alpha_{n}}$ for $\mathbbm{R}_{>0}\ni \tau$ converging to zero. Therefore locally around $p_{j}$ we have $\ln \Big(\sum_{\alpha_{j}\in  \nu^{p_{j}}(V_{k,j})}\lvert \frac{s_{\alpha_{j}}(\tau^{\gamma}\mathbf{z_{j}})}{\tau^{\gamma\cdot \alpha_{j}}} \rvert^{2}\Big)\sim \phi_{k,j}$ where $s_{\alpha_{j}}$ are sections in $V_{k,j}$ with leading terms of their expansion at $p_{j}$ equal to $\alpha_{j}\in\mathbbm{N}^{n}$. Thus the idea is to consider the metric on $kL$ given by $\ln(\sum_{i=1}^{N}\sum_{\alpha_{i}\in \nu^{p_{i}}(V_{k,i})}\lvert \frac{s_{\alpha_{i}}}{\tau^{\gamma\cdot\alpha_{i}}} \rvert^{2}))$ and define an opportune factor $\gamma$ such that this metric approximates the local plurisubharmonic functions around the points $p_{1},\dots,p_{N}$ after the uniform zoom $\tau^{\gamma}$ for $\tau$ small enough. This will be possible thanks to Lemma \ref{lem:AndersonOrdering} and the definition of $V_{k,j}$. Finally a standard regularization argument will conclude the proof.
% that the Kähler metrics $\phi_{k,j}$ can be approximated for $\mathbbm{R}_{>0}\ni\tau\to 0$ by the family of local plurisubharmonic functions $\ln(\sum_{\alpha_{j}\in  \nu^{p_{j}}(V_{k,j})}\lvert \frac{s_{\alpha_{j}}(\tau^{\gamma}z_{j})}{\tau^{\gamma\cdot \alpha_{j}}} \rvert^{2})$ if $\tau\in \mathbbm{R}_{>0}$  where $s_{\alpha_{j}}$ are sections in $V_{k,j}$ with leading terms of their expansion at $p_{j}$ equal to $\alpha_{j}$, $\gamma\in\mathbbm{N}^{n}$, and for $\tau^{\gamma}z_{j}$ we mean $(\tau^{\gamma_{1}}z_{j,1},\dots \tau^{\gamma}z_{j,n})$. Therefore using the Lemma \ref{lem:AndersonOrdering} and the characterization of $V_{k,j}$ we will be able to choice an \emph{uniform rescaling factor} $\gamma$ such that the metric $\ln(\sum_{i=1}^{N}\sum_{\alpha_{i}\in \nu^{p_{i}}(V_{k,i})}\lvert \frac{s_{\alpha_{i}}}{\tau^{\gamma\cdot\alpha_{i}}} \rvert^{2}))$ will approximate (by rescaling for $\tau^{\gamma})$ the plurisubharmonic function $\ln(\sum_{\alpha_{j}\in  \nu^{p_{j}}(V_{k,j})}\lvert \frac{s_{\alpha_{j}}(\tau^{\gamma}z_{j})}{\tau^{\gamma\cdot \alpha_{j}}} \rvert^{2})$ around $p_{j}$ if $\tau$ is small enough, i.e. the contribute of all sections not in $V_{k,j}$ can be made negligible. A standard regularization argument will conclude the proof.
\begin{proof}
\emph{\textbf{Step 1: Pick sections.}} We assume that the local holomorphic coordinates $\mathbf{z_{j}}=\{z_{j,1},\dots,z_{j,n}\} $ centered a $p_{j}$ contains the unit ball $ B_{1}\subset \mathbbm{C}^{n} $ for every $j=1,\dots,n$.\\
Set $\mathcal{A}_{j}:=\nu^{p_{j}}(V_{k,j})$ and $\mathcal{B}_{i}^{j}:=\nu^{p_{i}}(V_{k,j})$ for $i\neq j$ to simplify the notation, let $k$ be large enough so that $ \Delta_{j}^{k}(L)^{\mathrm{ess}}\neq \emptyset $ for any $j=1,\dots,N$ (by Lemma \ref{lem:IntOk} and Proposition \ref{prop:Essential}) and let $ \{U_{j}\}_{j=1}^{N} $ be a family of relatively compact open set (respectively) in $ \{M_{k,j}\}_{j=1}^{N} $. Pick $ \gamma\in\mathbbm{N}^{n} $ as in Lemma \ref{lem:AndersonOrdering} for $\mathcal{S}=\bigcup_{j=1}^{N}\big(\mathcal{A}_{j}\cup\bigcup_{i\neq j}\mathcal{B}^{j}_{i}\big) $ ordered with the total additive order $>$ induced by the family of quasi-monomial valuations, i.e. $\alpha>\beta$ iff $\alpha \cdot \gamma>\beta\cdot\gamma$.\\
Next, for any $j=1,\dots,N$, by construction we can choice a family of sections $s_{\mathbf{\alpha_{j}}}$ in $V_{k,j}$, parametrized by $\mathcal{A}_{j}$, such that locally
\begin{gather*}
s_{\mathbf{\alpha_{j}}}(\mathbf{z_{j}})=\mathbf{z_{j}}^{\mathbf{\alpha_{j}}}+\sum_{\mathbf{\eta_{j}}>\mathbf{\alpha_{j}}}a_{j,\eta_{j}}\mathbf{z_{j}}^{\mathbf{\eta_{j}}}\\
s_{\mathbf{\alpha_{j}}}(\mathbf{z_{i}})=a_{i,j}\mathbf{z_{i}}^{\mathbf{\beta^{j}_{i}}}+\sum_{\mathbf{\eta_{i}}>\mathbf{\beta^{j}_{i}}}a_{i,\eta_{i}}\mathbf{z_{i}}^{\mathbf{\eta_{i}}}
\end{gather*}
with $a_{i,j}\neq 0$ and $\mathbf{\alpha_{j}}<\mathbf{\beta^{j}_{i}}$ for any $i\neq j$.\\
\emph{\textbf{Step 2: A suitable zoom.}} If we define $\tau^{\gamma}\mathbf{z_{j}}:=(\tau^{\gamma_{1}}z_{j,1}\dots,\tau^{\gamma_{n}}z_{j,n})$ for $\tau\in\mathbbm{R}_{\geq 0}$, then we get for any $ \mathbf{\alpha_{j}}\in\mathcal{A}_{j}$
\begin{gather}
\label{equation:Embedding11}
s_{\mathbf{\alpha_{j}}}(\tau^{\gamma}\mathbf{z_{j}})=\tau^{\gamma\cdot\mathbf{\alpha_{j}}}(\mathbf{z_{j}}^{\mathbf{\alpha_{j}}}+O(\lvert \tau\rvert)) \qquad \forall\, \tau^{\gamma}\mathbf{z_{j}}\in B_{1} \\
\label{equation:Embedding21}
s_{\mathbf{\alpha_{j}}}(\tau^{\gamma}\mathbf{z_{i}})=\tau^{\gamma\cdot\mathbf{\beta^{j}_{i}}}(a_{i,j}\mathbf{z_{j}}^{\mathbf{\beta^{j}_{i}}}+O(\lvert \tau\rvert)) \qquad \forall\, \tau^{\gamma}\mathbf{z_{i}}\in B_{1}
\end{gather}
Let, for any $j=1,\dots,N$, $g_{j}:M_{k,j}\to[0,1]$ be a smooth function such that $ g_{j}\equiv 0 $ on $U_{j}$ and $ g_{j}\equiv 1 $ on $K_{j}^{C} $ for some smoothly bounded compact set $K_{j}$ such that $ U_{j}\Subset K_{j}\subset M_{k,j}  $. Furthermore let $U_{j}'$ be a relatively compact open set in $M_{k,j}$ such that $K_{j}\subset U_{j}'$.\\
Then pick $0<\delta\ll 1$ such that $ \phi_{j}:=\phi_{k,j}-4\delta g_{j}$ is still strictly plurisubharmonic for any $j=1,\dots,N$.\\
Now we claim that for any $j$ there is a real positive number $0<\tau_{j}=\tau_{j}(\delta)\ll1$ such that for every $0<\tau\leq\tau_{j}$ the following statements hold:
\begin{gather*}
\tau^{\gamma}\mathbf{z_{j}}\in B_{1} \qquad \forall \, \mathbf{z_{j}}\in U_{j}', \\
\label{equation:Estimates1}
\phi_{j}>\ln\Big(\sum_{i=1}^{N} \sum_{\mathbf{\alpha_{i}}\in\mathcal{A}_{i}}\lvert \frac{s_{\mathbf{\alpha_{i}}}(\tau^{\gamma}\mathbf{z_{j}})}{\tau^{\gamma\cdot\mathbf{\alpha_{i}}}} \rvert^{2} \Big)-\delta \quad \mathrm{on} \, \, U_{j},\\
\label{equation:Estimates2}
\phi_{j}<\ln\Big(\sum_{i=1}^{N} \sum_{\mathbf{\alpha_{i}}\in\mathcal{A}_{i}}\lvert \frac{s_{\mathbf{\alpha_{i}}}(\tau^{\gamma}\mathbf{z_{j}})}{\tau^{\gamma\cdot\mathbf{\alpha_{i}}}} \rvert^{2} \Big)-3\delta \quad \mathrm{near} \, \partial K_{j}.
\end{gather*}
Indeed it is sufficient that each request is true for $ \tau\in(0,a) $ with $a$ positive real number. The first request is clear, while the others follow from the equations (\ref{equation:Embedding11}) and (\ref{equation:Embedding21}) since $g_{j}\equiv 0$ on $U_{j}$ and $g_{j}\equiv 1$ on $K_{j}^{C}$ (recall that $g_{j}$ is smooth and that $\gamma\cdot \mathbf{\alpha_{i}}<\gamma\cdot\mathbf{\beta^{j}_{i}}$ if $\mathbf{\alpha_{i}}\in\mathcal{A}_{i}$ for any $j\neq i$).\\
So, since $p_{1},\dots,p_{N}$ are distinct points on $X$, we can choose ${0<\tau_{k}\ll 1}$ such that the requests above hold for every $j=1,\dots,N$ and that $ W_{j}\cap W_{i}=\emptyset $ if $j\neq i$ where $ W_{j}:=\varphi_{j}^{-1}(\tau_{k}^{\gamma}U_{j}') $ for $\varphi_{j}$ coordinate map giving the local holomorphic coordinates centered at $p_{j}$.\\
\emph{\textbf{Step 3: Gluing}} We define, for any $j=1,\dots,N$,
$$
\phi_{j}':=\max_{reg}\Bigg(\phi_{j},\ln\Big(\sum_{i=1}^{N} \sum_{\mathbf{\alpha_{i}}\in\mathcal{A}_{i}}\lvert \frac{s_{\mathbf{\alpha_{i}}}(\tau^{\gamma}\mathbf{z_{j}})}{\tau^{\gamma\cdot\mathbf{\alpha_{i}}}} \rvert^{2} \Big)-2\delta\Bigg)
$$
where $\max_{reg}(x,y)$ is a smooth convex function such that $\max_{reg}(x,y)=\max(x,y)$ whenever $ \lvert x-y\lvert >\delta$. Therefore, by construction, we observe that $\phi_{j}'$ is smooth and strictly plurisubharmonic on $ M_{k,j} $, identically equal to $ \ln\Big(\sum_{i=1}^{N} \sum_{\mathbf{\alpha_{i}}\in\mathcal{A}_{i}}\lvert \frac{s_{\mathbf{\alpha_{i}}}(\tau^{\gamma}\mathbf{z_{j}})}{\tau^{\gamma\cdot\mathbf{\alpha_{i}}}} \rvert^{2} \Big)-2\delta $ near $\partial K_{j}$ and identically equal to $ \phi_{k,j} $ on $U_{j} $. So
$$
\omega_{j}:=dd^{c}\phi_{j}'
$$
is equal to $\omega_{k,j}$ on $ U_{j}$. Thus since for $k\gg 1$ big enough $ \ln\Big(\sum_{i=1}^{N} \sum_{\mathbf{\alpha_{i}}\in\mathcal{A}_{i}}\lvert \frac{s_{\mathbf{\alpha_{i}}}}{\tau^{\gamma\cdot\mathbf{\alpha_{i}}}} \rvert^{2} \Big)-2\delta  $ extends as a positive hermitian metric of $kL$, with abuse of notation and unless restricting further $\tau$, we get that $\{\omega_{j}\}_{j=1}^{N}$ extend to a Kähler form $\omega$ such that
$$
\omega_{f(U_{j})}=f_{*}(\omega_{j\vert U_{j}})=f_{*}\omega_{k,j}
$$
where we have set $f:\bigsqcup_{j=1}^{N}U_{j}\to X, f_{\vert U_{j}}:=\varphi_{j}^{-1}\circ \tau^{\gamma}$ (the \emph{uniform rescaled} embedding).\\
Since $\{U_{j}\}_{j=1}^{N}$ are arbitrary, this shows that $\{(M_{k,j},\omega_{k,j})\}_{j=1}^{N}$ packs into $(X,kL)$.
\end{proof}
\begin{reptheorem}{ThmC}[Ample Case]
Let $L$ be an ample line bundle. Then $\{(\Omega_{j}(L),\omega_{st})\}_{j=1}^{N}$ packs perfectly into $(X,L)$.
\end{reptheorem}
\begin{proof}
If $U_{1},\dots, U_{N}$ are relatively compact open sets, respectively, in $\Omega_{j}(L)$ then by Proposition \ref{prop:Essential} there exists $k>0$ divisible enough such that $U_{j}$ is compactly contained in $\mu^{-1}(\mathrm{Conv}(\Delta_{j}^{k}(L))^{\circ}$ for any $j=1,\dots, N$, i.e. $\sqrt{k}U_{j}\Subset \mathcal{D}_{k,j}\Subset M_{k,j}$ for any $j=1,\dots, N$. \\
By Lemma \ref{lem:Legendre} there exist smooth functions $g_{j}:M_{k,j}\to \mathbbm{R}$ with support on relatively compact open sets $U_{j}'\supset \sqrt{k}U_{j}$ such that ${\tilde{\omega}_{j}:=\omega_{k,j} +dd^{c}g_{j}}$ is K\"ahler and $\tilde{\omega}_{j}=\omega_{st}$ holds on $\sqrt{k}U_{j}$.\\
Furthermore, fixing relatively compact open sets $V_{j}\subset M_{k,j}$ such that $U_{j}'\Subset V_{j}$ for any $j=1,\dots,N$, by Theorem \ref{thm:EmbeddingsGrandi} we can find a holomorphic embedding $f': \bigsqcup_{j=1}^{N}V_{j}\to X$ and a K\"ahler form $\omega'$ in $c_{1}(kL)$ such that $\omega'_{\vert f'(V)}=f'_{*}\omega_{k,j}$ for any $j=1,\dots,N$.\\
Next, let $\chi_{j}$ be smooth cut-off functions on $X$ such that $\chi_{j}\equiv 1$ on $f'(U_{j}')$ and $\chi_{j}\equiv 0$ outside $\overline{f'(V_{j})}$. Thus, since $f'(V_{j})\cap f'(V_{i})=\empty$ for every $j \neq i$ and since $g_{j}\circ f'^{-1}_{\vert f'(V_{j})}$ has compact support in $f'(U_{j}')$, the function $g=\sum_{j=1}^{N}\chi_{j}g_{j}\circ f'^{-1}$, extends to $0$ outside $\bigcup_{j=1}^{N}\overline{f'(V_{j})}$ and $g_{\vert f'(V_{j})}=g_{j}\circ f^{-1}_{\vert f'(V_{j})}$.\\% while $g\equiv 0$ outside $\bigcup_{j=1}^{N}\overline{f'(U_{j}')}$.\\
Finally defining $f: \bigsqcup_{j=1}^{N}U_{j}\to X$ by $f_{\vert U_{j}}(z_{j}):=f'_{\vert \sqrt{k}U_{j}}(\sqrt{k}z_{j})$, we get
$$
(\omega'+dd^{c}g)_{\vert f(U_{j})}=f'_{*}(\omega_{k,j}+dd^{c}g_{j})_{\vert \sqrt{k}U_{j}}=kf_{*}\omega_{st\vert U_{j}}
$$
by construction. Hence $\omega:=\frac{1}{k}(\omega'+dd^{c}g)$ is a K\"ahler form with class $c_{1}(L)$ that satisfies the requests since by Theorem \ref{ThmA}
$$
\sum_{j=1}^{N}\int_{\Omega_{j}(L)}\omega_{st}^{n}=n!\sum_{j=1}^{N}\mathrm{Vol}_{\mathbbm{R}^{n}}(\Delta_{j}(L))=\mathrm{Vol}_{X}(L)=\int_{X}\omega^{n}.
$$
\end{proof}
\begin{rem}
\emph{If the family of valuations fixed is associated to a family of admissible flags $Y_{j,i}=\{z_{j,1}=\dots=z_{j,i}=0\}$ then each associated embedding $f:\bigsqcup_{j=1}^{N}U_{j}\to X$ can be chosen so that
$$
f^{-1}_{\vert f(U_{j})}(Y_{j,i})=\{z_{j,1}=\cdots=z_{j,i}=0\}
$$
In particular if $N=1$ we recover the Theorem $A$ in \cite{WN15}.}
\end{rem}
\subsection{The big case}
\label{subsection:BigCase}
\begin{defn}
\label{defn:KPBig}
If $L$ is big, we say that a finite family of $n-$dimensional K\"ahler manifolds $ \{(M_{j},\eta_{j})\}_{j=1}^{N} $ packs into $ (X,L) $ if for every family of relatively compact open set $ U_{j} \Subset M_{j} $ there is a holomorphic embedding $ f: \bigsqcup_{j=1}^{N}U_{j}\to X $ and there exists a k\"ahler current with analytical singularities $ T $ lying in $ c_{1}(L) $ such that $ f_{*}\eta_{j}=T_{\vert f(U_{j})} $. If, in addition,
$$
\sum_{j=1}^{N}\int_{M_{j}}\eta_{j}^{n}=\int_{X}c_{1}(L)^{n}
$$
then we say that $\{(M_{j},\eta_{j})\}_{j=1}^{N}$ packs perfectly into $(X,L)$.
\end{defn}
Reasoning as in the previous section we prove the following result.
\begin{reptheorem}{ThmC}[Big Case]
Let $L$ be a big line bundle. Then $ \{(\Omega_{j}(L),\omega_{st})\}_{j=1}^{N}$ packs perfectly into $(X,L)$.% where $\Delta_{l}(L)^{\circ}= \emptyset$ if $l\notin\{i_{1},\dots,i_{q}\}$ while $\Delta_{i_{m}}(L)^{\circ}\neq \emptyset$ for any $m=1,\dots,q$.
\end{reptheorem}
\begin{proof}
By Lemma \ref{lem:IntOk}, $\Omega_{j}(L)=\emptyset$ for any $j$ such that $\Delta_{j}(L)^{\circ}=\emptyset$. So, unless removing some of the points we may assume that $\Delta_{j}(L)^{\circ}\neq \emptyset$ for any $j=1,\dots,N$. \\
Thus letting $k\gg 0$ big enough such that $\Delta_{j}^{k}(L)^{ess}\neq\emptyset$ for any $j$ (Proposition \ref{prop:Essential}) we can proceed similarly to the proof of Theorem \ref{thm:EmbeddingsGrandi} with the unique difference that $ \ln\Big(\sum_{i=1}^{N}\sum_{\mathbf{\alpha_{i}}\in\mathcal{A}_{i}}\lvert \frac{s_{\mathbf{\alpha_{i}}}}{\tau^{\gamma\cdot\mathbf{\alpha_{i}}}} \rvert^{2}\Big) $ extends to a positive singular hermitian metric, hence we get a (current of) curvature $T$ that is a K\"ahler current with analytical singularities. Next, as in the ample case, we can show that $\{(\Omega_{j}(L),\omega_{st})\}_{j=1}^{N}$ packs perfectly into $(X,L)$.% using Theorem \ref{ThmA}.
%Finally by  we get
%$$
%\sum_{j=1}^{N}\int_{D_{j}(L)}\omega_{st}^{n}=n!\sum_{j=1}^{N}\mathrm{Vol}_{\mathbbm{R}^{n}}(\Delta_{j}(L))=\mathrm{Vol}_{X}(L)=\int_{X}c_{1}(L)^{n}
%$$
\end{proof}
\begin{rem}
\emph{If the family of valuations fixed is associated to a family of admissible flags $Y_{j,i}=\{z_{j,1}=\dots=z_{j,i}=0\}$ then each associated embedding $f:\bigsqcup_{j=1}^{N}U_{j}\to X$ can be chosen so that
$$
f^{-1}_{\vert f(U_{j})}(Y_{j,i})=\{z_{j,1}=\cdots=z_{j,i}=0\}
$$
In particular if $N=1$ we recover the Theorem $C$ in \cite{WN15}.}
\end{rem}
\section{Local Positivity}
\label{section:SeshadriConstant}
\subsection{Moving Multipoint Seshadri Constant}
\begin{defn}
\label{defn:MultipointSeshadriConstant}
Let $L$ be a nef line bundle on $X$. The quantity
$$
\epsilon_{S}(L;p_{1},\dots,p_{N}):=\inf \frac{L\cdot C}{\sum_{i=1}^{N}\mathrm{mult}_{p_{i}}C}
$$
where the infimum is over all irreducible curves $C\subset X$ passing through at least one of the points $p_{1},\dots,p_{N}$ is called the \textbf{multipoint Seshadri constant at }$\mathbf{p_{1},\dots,p_{N}}$\textbf{ of} $\mathbf{L}$.
\end{defn}
This constant has played an important role in the last three decades and it is the natural extension of the Seshadri constant introduced by Demailly in \cite{Dem90}.\\
The following Lemma is well-known and its proof can be found for instance in \cite{Laz04}, \cite{BDRH$^{+}$09}:
\begin{lem}
\label{lem:SeshadriNef}
Let $L$ be a nef line bundle on $X$. Then
\begin{equation*}
\epsilon_{S}(L;p_{1},\dots,p_{N})=\sup\Big\{t\geq 0 \, : \, \mu^{*}L-t\sum_{i=1}^{N}E_{i} \,\, \mathrm{is} \, \, \mathrm{nef}\Big\}=\inf \Big(\frac{L^{\dim V}\cdot V}{\sum_{j=1}^{N}\mathrm{mult}_{p_{j}}V}\Big)^{\frac{1}{\dim V}}
\end{equation*}
where $\mu:\tilde{X}\to X$ is the blow-up at $Z=\{p_{1},\dots,p_{N}\}$, $E_{i}$ is the exceptional divisor above $p_{i}$ and where the infimum on the right side is over all positive dimensional irreducible subvarieties $V$ containing at least one point among $p_{1},\dots,p_{N}$.
\end{lem}
%\begin{proof}
%We denote the quantity on the right side with $\bar{\epsilon}$ and $\epsilon_{S}(L;p_{1},\dots,p_{N})$ with $\epsilon_{S}$.\\
%By nefness we get that $(\mu^{*}L-\bar{\epsilon}\sum_{i=1}^{N}E_{i})\cdot\tilde{C}\geq 0$ for any $\tilde{C}$ proper trasform of a curve $C$ passing through at least one of the points $p_{1},\dots,p_{N}$, which implies immediately $\bar{\epsilon}\leq \epsilon_{S}$.\\
%Viceversa we consider $L':=\mu^{*}L-\epsilon_{S}\sum_{i=1}^{N}E_{i} $ and we want to show it is nef to get $\epsilon_{S}\leq \bar{\epsilon}$.\\
%For any $\tilde{C}$ irreducible curve on $\tilde{X}$ disjoint from the support of $E_{1},\dots,E_{N}$ we get $L'\cdot \tilde{C}=L\cdot C\geq 0$, where $C:=\mu(\tilde{C})$, by nefness of $L$. While if $\tilde{C}\subset E_{i}$ then $L'\cdot \tilde{C}=\bar{\epsilon}(-E_{i})\cdot \tilde{C}>0$ where the inequality follows from $\mathcal{O}_{\tilde{C}}(-E_{i})\simeq \mathcal{O}_{\mathbbm{P}^{n-1}}(1)_{\vert \tilde{C}}$ which is ample. Finally if $\tilde{C}$ intersect at least one of the exceptional divisors $E_{1},\dots,E_{N}$ but it is not contained in any of them, then
%$$
%L'\cdot \tilde{C}=L\cdot C-\epsilon_{S}\sum_{i=1}^{N}\mathrm{mult}_{p_{i}}C\geq 0
%$$
%by definition of the Seshadri constant $\epsilon_{S}$.
%\end{proof}
The characterization of Lemma \ref{lem:SeshadriNef} allows to extend the definition to nef $\mathbbm{Q}-$line bundles by homogeneity and to nef $\mathbbm{R}-$line bundles by continuity.

Here we describe a possible generalization of the multipoint Seshadri constant for big line bundles:
\begin{defn}
Let $L$ be a big $\mathbbm{R}-$line bundle, we define the \textbf{moving multipoint Seshadri constant at }$\mathbf{p_{1},\dots,p_{N}}$\textbf{ of} $\mathbf{L}$ as
$$
\epsilon_{S}(||L||;p_{1},\dots,p_{N}):=\sup_{f^{*}L=A+E}\epsilon_{S}(A;f^{-1}(p_{1}),\dots,f^{-1}(p_{N}))
$$
if $ p_{1},\dots,p_{N}\notin \mathbbm{B}_{+}(L) $ and as $ \epsilon_{S}(||L||;p_{1},\dots,p_{N}):=0 $ otherwise, where the supremum is taken over all projective morphisms $f:Y\to X$ with $Y$ smooth such that $f$ is an isomorphism around $p_{1},\dots,p_{N}$ and over all decomposition $f^{*}L=A+E$ where $A$ is an ample $\mathbbm{Q}-$divisor and $E$ is effective with $f^{-1}(p_{j})\notin\mathrm{Supp}(E)$ for any $j=1,\dots,N$.
\end{defn}
%Nakayame in \cite{Nak03} introduced a generalization of the Seshadri constant for big $\mathbbm{Q}-$line bundles, called \emph{moving Seshadri constant}. Successively, in the paper \cite{ELMNP09}, the notion was extended to $\mathbbm{R}-$line bundles and some properties of the moving Seshadri function, like the homogeneity, the concavity on $B_{+}(p)^{C}$ and the continuity on the whole Neron-Severi $\mathrm{N}^{1}(X)_{\mathbbm{R}}$ were studied (see section \S $ 6$ in that paper, where recall that $B_{+}(p)=\{\alpha\in\mathrm{N}^{1}(X)_{\mathbbm{R}}\, : p\in\mathbbm{B}_{+}(\alpha)\}$). \\
For $N=1$, we retrieve the definition given in \cite{ELMNP09}.\\

The following properties are well-known for the one-point case.
\begin{prop}
\label{prop:Properties}
Let $L, L'$ be big $\mathbbm{R}-$line bundles. Then
\begin{itemize}
\item[i)] $\epsilon_{S}(||L||; p_{1},\dots,p_{N})\leq \big(\frac{\mathrm{Vol}_{X}(L)}{N}\big)^{1/n}$;
\item[ii)] if $c_{1}(L)=c_{1}(L')$ then $\epsilon_{S}(||L||; p_{1},\dots,p_{N})=\epsilon_{S}(||L'||; p_{1},\dots,p_{N})$;
\item[iii)] $ \epsilon_{S}(||\lambda L||; p_{1},\dots,p_{N})=\lambda \epsilon_{S}(||L||; p_{1},\dots,p_{N})$ for any $\lambda \in \mathbbm{R}_{>0}$;
\item[iv)] if $p_{1},\dots,p_{N}\notin \mathbbm{B}_{+}(L)\cup \mathbbm{B}_{+}(L')$ then $\epsilon_{S}(||L+L'||; p_{1},\dots,p_{N})\geq \epsilon_{S}(||L|| ; p_{1},\dots,p_{N})+\epsilon_{S}(||L'|| ; p_{1},\dots,p_{N})$.
\end{itemize}
\end{prop}
\begin{proof}
Combining the definition of the moving multipoint Seshadri constant with Lemma \ref{lem:SeshadriNef}, the first three points are immediate since ampleness and nefness are numerical conditions and the ample, nef classes form the so-called ample and nef cones. More precisely, for a projective morphism $f:Y\to X$ and an ample $\mathbbm{Q}$-divisor as in the definition, the homogeneity of $\epsilon_{S}\big(A;f^{-1}(p_{1}),\dots,f^{-1}(p_{N})\big)$ is given by the first equality of Lemma \ref{lem:SeshadriNef} while the second equality of the same Lemma \ref{lem:SeshadriNef} yields $\epsilon_{S}\big(A;f^{-1}(p_{1}),\dots,f^{-1}(p_{N})\big)\leq \big(A^{n}/N\big)^{1/n}\leq \big(\mathrm{Vol}_{X}(L)/N\big)^{1/n}$.\newline
Regarding the last point, fix $A, A'$ ample $\mathbbm{Q}$-divisors as in the definition of the moving multipoint Seshadri constant for $L, L'$ with respect to projective morphisms $f:Y\to X, f':Y'\to X$. Then taking projective morphisms $g:W\to Y, g':W\to Y'$ which are isomorphism around $p_{1},\dots,p_{N}$, there exist effective divisors $F,F'$ on $W$ such that $g^{-1}f^{-1}(p_{j})\notin \mathrm{Supp}(F)$ (and similarly for $F'$) and decreasing sequences converging to $0$ of positive rational numbers $\{a_{m}\}_{m\in\mathbbm{N}}, \{a_{m}'\}_{m\in\mathbbm{N}}\subset \mathbbm{Q}_{>0}$ such that $B_{m}:=g^{*}A-a_{m}F$, $B_{m}':=g'^{*}A'-a_{m}'F'$ are ample $\mathbbm{Q}$-divisors for any $m\in \mathbbm{N}$. Then we claim that
\begin{equation}
\label{eqn:Approx}
\epsilon_{S}\big(B_{m}; g^{-1}f^{-1}(p_{1}),\dots,g^{-1}f^{-1}(p_{N})\big)\to \epsilon_{S}\big(A; f^{-1}(p_{1}),\dots,f^{-1}(p_{N})\big)
\end{equation}
as $m\to \infty$, and similarly for $A',B'_{m}$. In fact letting $\mu:\tilde{Y}\to Y$, $\nu: \tilde{W}\to W$ be the blow-ups at $f^{-1}Z, g^{-1}f^{-1}Z$ where $Z=\{p_{1},\dots,p_{N}\}$, there is a commutative diagram
\begin{tikzcd}
\tilde{W} \arrow{r}{\nu} \arrow[swap]{d}{\tilde{g}} & W \arrow{d}{g} \\%
\tilde{Y} \arrow{r}{\mu}& Y
\end{tikzcd}
for a suitable projective morphism $\tilde{g}:\tilde{W}\to \tilde{Y}$ because the blow-ups are local projective morphisms and $g$ is an isomorphism around $f^{-1}Z$. Therefore the convergence (\ref{eqn:Approx}) follows using the first characterization of Lemma \ref{lem:SeshadriNef} since there exists an uniform constant $K>0$ such that $|\nu^{*}F\cdot C|\leq K$ for any irreducible curve $C\subset \tilde{W}$ and nefness is preserved under pullback.\\
Thus, for any $\delta>0$ fixed, we can choose $A,A',f,f',m,m'$ such that
\begin{gather*}
\epsilon_{S}(\|L\|;p_{1},\dots,p_{N})+\epsilon_{S}(\|L'\|;p_{1},\dots,p_{N})\leq\\
\leq \epsilon_{S}\big(A;f^{1}(p_{1}),\dots,f^{-1}(p_{N})\big)+\epsilon_{S}\big(A';f'^{-1}(p_{1}),\dots,f'^{-1}(p_{N})\big)-\delta\leq\\
\leq \epsilon_{S}\big(B_{m};g^{-1}f^{-1}(p_{1}),\dots,g^{-1}f^{-1}(p_{N})\big)+\epsilon_{S}\big(B_{m}';g'^{-1}f'^{-1}(p_{1}),\dots,g'^{-1}f'^{-1}(p_{N})\big)-2\delta\leq\\
\leq \epsilon_{S}\big(B_{m}+B_{m}';g^{-1}f^{-1}(p_{1}),\dots,g^{-1}f^{-1}(p_{N})\big)-2\delta
\end{gather*}
where the last inequality is an easy consequence of the convexity of the nef cone. Hence since $(f'\circ g')^{*}(L+L')=B_{m}+B_{m}'+G$ where $G
$ for an effective divisor $G$ such that $(f'\circ g')^{-1}(p_{j})\notin \mathrm{Supp}(G)$, by definition we deduce that
$$
\epsilon_{S}(\|L\|;p_{1},\dots,p_{N})+\epsilon_{S}(\|L'\|;p_{1},\dots,p_{N})\leq \epsilon_{S}(\|L+L'\|;p_{1},\dots,p_{N})-2\delta
$$
which clearly concludes the proof.
\end{proof}
We check that the moving multipoint Seshadri constant is an effective generalization of the multipoint Seshadri constant:
\begin{prop}
Let $L$ be a big and nef $\mathbbm{Q}-$line bundle. Then
$$
\epsilon_{S}(||L||;p_{1},\dots,p_{N})=\epsilon_{S}(L; p_{1},\dots,p_{N})
$$
\end{prop}
\begin{proof}
By homogeneity we can assume $L$ line bundle and ${p_{1},\dots,p_{N}\notin\mathbbm{B}_{+}(L)}$ since if $p_{j}\in\mathbbm{B}_{+}(L)$ for some $j$ then by Proposition $1.1.$ and Corollary $5.6.$ in \cite{ELMNP09} there exists an irreducible positive dimensional component $ V\subset \mathbbm{B}_{+}(L)$ with $p_{j}\in V$ such that $L^{\dim V}\cdot V=0$ and Lemma \ref{lem:SeshadriNef} gives the equality.\\
Thus, for a fixed a projective morphism $f:Y\to X$ as in the definition, we get
$$
\frac{L\cdot C}{\sum_{i=1}^{N}\mathrm{mult}_{p_{i}}C}=\frac{f^{*}L\cdot \tilde{C}}{\sum_{i=1}^{N}\mathrm{mult}_{f^{-1}(p_{i})}\tilde{C}}\geq \frac{A\cdot \tilde{C}}{\sum_{i=1}^{N}\mathrm{mult}_{f^{-1}(p_{i})}\tilde{C}}
$$
since $f^{-1}(p_{1}),\dots,f^{-1}(p_{N})\notin \mathrm{Supp}(E)$, and $\epsilon_{S}(||L||;p_{1},\dots,p_{N})\leq \epsilon_{S}(L; p_{1},\dots,p_{N})$ follows.\\
For the reverse inequality, we can write $L=A+E$ with $A$ ample $\mathbbm{Q}-$line bundle and $E$ effective such that $p_{1},\dots, p_{N}\notin\mathrm{Supp}(E)$, noting that $L=A_{m}+ \frac{1}{m}E$ for any $m\in\mathbbm{N}$ where $A_{m}:=\frac{1}{m}A+(1-\frac{1}{m})L$ is an ample $\mathbbm{Q}-$line bundle. Thus $ \epsilon_{S}(||L||;p_{1},\dots,p_{N})\geq\epsilon_{S}(A_{m}; p_{1},\dots,p_{N}) $ and letting $m\to\infty$ the inequality requested follows from the continuity of $\epsilon_{S}(\cdot;p_{1},\dots,p_{N})$ in the nef cone.
\end{proof}
The following Proposition justifies the name given as generalization of the definition in \cite{Nak03}:
\begin{prop}
\label{prop:EquiivalenceNak}
If $L$ is a big $\mathbbm{Q}-$line bundle such that ${p_{1},\dots,p_{N}\notin\mathbbm{B}(L)}$ then
\begin{equation*}
\epsilon_{S}(||L||; p_{1},\dots,p_{N})=\lim_{k\to \infty} \frac{\epsilon_{S}(M_{k}; \mu_{k}^{-1}(p_{1}),\dots,\mu_{k}^{-1}(p_{N}))}{k}=\sup_{k\to \infty} \frac{\epsilon_{S}(M_{k}; \mu_{k}^{-1}(p_{1}),\dots,\mu_{k}^{-1}(p_{N}))}{k}
\end{equation*}
where $M_{k}:=\mu_{k}^{*}(kL)-E_{k}$ is the \emph{moving part} of $\vert mL \vert $ given by a resolution of the base ideal $\mathfrak{b}_{k}:=\mathfrak{b}(\vert kL \vert)$ (or set $M_{k}=0$ if ${H^{0}(X,kL)=\{0\}}$).
\end{prop}
%Note that, since $p_{1},\dots,p_{N}\notin \mathbbm{B}(L)$, for $k\gg 0$ divisible enough $\mu_{k}:X_{k}\to X$ is an isomorphism around $p_{j}$ for any $j=1,\dots,N$ and recall that $\mathcal{O}_{X_{k}}(-E_{k})=\mu^{-1}_{k}\mathfrak{b}_{k}$.\\
Note that $\epsilon_{S}(M_{k};\mu_{k}^{-1}(p_{1}),\dots,\mu_{k}^{-1}(p_{N})))$ does not depend on the resolution chosen and given $k_{1},k_{2}$ divisible enough we may choose resolutions such that $M_{k_{1}+k_{2}}=M_{k_{1}}+M_{k_{2}}+E$ where $E$ is an effective divisor with $p_{1},\dots,p_{N}\notin \mathrm{Supp}(E)$, so the existence of the limit in the definition follows from Proposition \ref{prop:Properties}.(iv).
\begin{proof}[Proof of Proposition \ref{prop:EquiivalenceNak}]
By homogeneity we can assume $L$ big line bundle, $\mathbbm{B}(L)=\mathrm{Bs}(\vert L\vert)$ and that the rational map ${\varphi: X\setminus \mathrm{Bs}(\lvert L\lvert)\to \mathbbm{P}^{N}}$ associated to the linear system $\vert L\vert$ has image of dimension $n$.\\
Suppose first that there exist $j\in\{1,\dots,N\}$ and an integer $k_{0}\geq 1$ such that ${\mu_{k_{0}}^{-1}(p_{j})\in \mathbbm{B}_{+}(M_{k_{0}})}$. Thus for any $\mathbbm{N}\ni k\geq k_{0}$ we get ${\mu_{k}^{-1}(p_{j})\in \mathbbm{B}_{+}(M_{k})}$. Then, since $M_{k}$ is big and nef, there exists a subvariety $V$ of dimension $d\geq 1$ such that $M_{k}^{d}\cdot V =0$ and $V\ni \mu_{k}^{-1}(p_{j})$ (Corollary $5.6.$ in \cite{ELMNP09}), thus ${\epsilon_{S}(M_{k}; \mu_{k}^{-1}(p_{1}),\dots,\mu_{k}^{-1}(p_{N}))=0}$ by Lemma \ref{lem:SeshadriNef} and the equality follows.\\
Therefore we may assume $\mu_{k}^{-1}(p_{1}),\dots,\mu_{k}^{-1}(p_{N})\notin \mathbbm{B}_{+}(M_{k})$ for any $k\geq 1$ and we can write $M_{k}=A+E$ for $A$ ample and $E$ effective such that $\mu_{k}^{-1}(p_{1}),\dots,\mu_{k}^{-1}(p_{N})\notin\mathrm{Supp}(E)$. Clearly for any $m\in\mathbbm{N}$, setting $A_{m}:=\frac{1}{m}A+(1-\frac{1}{m})M_{k}$, the equality $M_{k}=A_{m}+\frac{1}{m}E$ holds. Hence, since by definition $\epsilon_{S}(||L||; p_{1},\dots,p_{N})\geq\frac{1}{k}\epsilon_{S}(A_{m};\mu_{k}^{-1}(p_{1}),\dots,\mu_{k}^{-1}(p_{N}))$ for any $m\in \mathbbm{N}$, we get $\epsilon_{S}(||L||; p_{1},\dots,p_{N})\geq\frac{1}{k}\epsilon_{S}(M_{k}; \mu_{k}^{-1}(p_{1}),\dots, \mu_{k}^{-1}(p_{N}))$ letting $m\to\infty$.\\
For the reverse inequality, let $f:Y\to X$ be a projective morphism as in the definition of the moving multipoint Seshadri constant, i.e. ${f^{*}L=A+E}$ with $A$ ample $\mathbbm{Q}-$divisor and $E$ effective divisor with ${p_{1},\dots,p_{N}\notin\mathrm{Supp}(E)}$, and let $k\gg 1$ big enough such that $kA$ is very ample. Thus, unless taking a log resolution of the base locus of $f^{*}(kL)$ that is an isomorphism around $f^{-1}(p_{1}),\dots,f^{-1}(p_{N})$, we can suppose ${f^{*}(kL)=M_{k}+E_{k}}$ with $p_{1},\dots,p_{N}\notin\mathrm{Supp}(E_{k})$ for $E_{k}$ effective and $M_{k}$ nef and big. Then, since $kA$ is very ample, $M_{k}=kA+E_{k}'$ with $E_{k}'$ effective and $E_{k}'\leq kE$. Hence we get $f^{-1}(p_{1}),\dots,f^{-1}(p_{N})\notin \mathrm{Supp}(E_{k}')$ and $\frac{1}{k}\epsilon_{S}(M_{k};f^{-1}(p_{1}),\dots,f^{-1}(p_{N}))\geq \epsilon_{S}(A; f^{-1}(p_{1}),\dots,f^{-1}(p_{N}))$ by homogeneity, which concludes the proof.
\end{proof}
\begin{prop}
\label{prop:PerConSC}
Let $L$ be a big $\mathbbm{Q}-$line bundle. Then
$$
\epsilon_{S}(||L||; p_{1},\dots,p_{N})=\inf \Bigg(\frac{\mathrm{Vol}_{X\vert V}(L)}{\sum_{j=1}^{N}\mathrm{mult}_{p_{j}}V} \Bigg)^{1/\dim V}
$$
where the infimum is over all positive dimensional irreducible subvarities $V$ containing at least one of the points $p_{1},\dots,p_{N}$.
\end{prop}
\begin{proof}
We may assume $p_{1},\dots,p_{N}\notin\mathbbm{B}_{+}(L)$ since otherwise the equality is a consequence of Corollary $5.9.$ in \cite{ELMNP09}. Thus $V\not\subset \mathbbm{B}_{+}(L)$ for any positive dimensional irreducible subvariety that pass through at least one of the points $p_{1},\dots,p_{N}$, hence by Theorem $2.13.$ in \cite{ELMNP09} it is sufficient to show that
$$
\epsilon_{S}(||L|| ; p_{1},\dots,p_{N})=\inf \Bigg( \frac{\parallel L^{\dim V}\cdot V\parallel}{\sum_{j=1}^{N}\mathrm{mult}_{p_{j}}V} \Bigg) ^{1/\dim V}
$$
where the infimum is over all positive dimensional irreducible subvarities $V$ that contain at least one of the points $p_{1},\dots,p_{N}$. We recall that the \emph{asymptotic intersection number} is defined as
$$
\parallel L^{\mathrm{dim}V}\cdot V \parallel:=\lim_{k\to \infty} \frac{M_{k}^{\dim V}\cdot \tilde{V}_{k}}{k^{\dim V}}=\sup_{k}\frac{M_{k}^{\dim V}\cdot \tilde{V}_{k}}{k^{\dim V}}
$$
where $M_{k}$ is the moving part of $\mu_{k}^{*}(kL)$ as in Proposition \ref{prop:EquiivalenceNak} and $\tilde{V}_{k}$ is the proper trasform of $V$ through $\mu_{k}$ (the last equality follows from Remark 2.9. in \cite{ELMNP09}). \\
Then Lemma \ref{lem:SeshadriNef} and Proposition \ref{prop:EquiivalenceNak} ($M_{k}$ is nef) imply
\begin{gather*}
\epsilon_{S}(||L||;p_{1},\dots,p_{N})=\sup_{k}\frac{\epsilon_{S}(M_{k};\mu_{k}^{-1}(p_{1}),\dots,\mu_{k}^{-1}(p_{N}))}{k}=\\
=\sup_{k}\inf_{V}\frac{1}{k} \Bigg( \frac{\big(M_{k}^{\dim V}\cdot \tilde{V}_{k}\big)}{\sum_{j=1}^{N}\mathrm{mult}_{p_{j}}V}\Bigg)^{1/\dim V}\leq \inf_{V}\Bigg(\frac{\parallel L^{\dim V}\cdot V\parallel }{\sum_{j=1}^{N}\mathrm{mult}_{p_{j}}V}\Bigg)^{1/\dim V}.
\end{gather*}
Conversely by the approximate Zariski decomposition showed in \cite{Tak06} (Theorem $3.1.$) for any $0< \epsilon<1$ there exists a projective morphism ${f:Y_{\epsilon}\to X}$ that is an isomorphism around $p_{1},\dots,p_{N}$, $f^{*}L=A_{\epsilon}+E_{\epsilon}$ where $A_{\epsilon}$ ample and $E_{\epsilon}$ effective with $f^{-1}(p_{1}),\dots,f^{-1}(p_{N})\notin \mathrm{Supp}(E_{\epsilon})$, and
$$
A_{\epsilon}^{\dim V}\cdot \tilde{V}\geq (1-\epsilon)^{\dim V}\parallel L^{\dim V} \cdot V \parallel
$$
for any $V\not\subset \mathbbm{B}_{+}(L)$ positive dimensional irreducible subvariety ($\tilde{V}$ proper trasform of $V$ through $f$). Therefore, taking the infimum over all positive dimensional irreducible subvarieties passing through at least one of the points $p_{1},\dots,p_{N}$ we get
\begin{equation*}
\epsilon_{S}(||L||;p_{1},\dots,p_{N})\geq\epsilon_{S}(A_{\epsilon};f^{-1}(p_{1}),\dots,f^{-1}(p_{N}))\geq (1-\epsilon)\inf \Bigg( \frac{\parallel L^{\dim V} \cdot V \parallel}{\sum_{j=1}^{N}\mathrm{mult}_{p_{j}}V} \Bigg)^{1/\dim V}
\end{equation*}
which concludes the proof.
%Let $k\gg 1$ big enough such that $\mathbbm{B}(L)=\mathrm{Bs}(|kL|)$ and $A$ ample line bundle on $X$, then for any $q\in \mathbbm{N}$ we fix $ \delta_{qk}\geq 0 $ small enough such that $ \mu_{qk}^{*}A-\delta_{qk}F_{qk} $ is ample where $F_{qk}$ is exceptional and $s\geq 1 $ such that $ sq(M_{k}+A-\delta_{qk}F_{qk})$ is very ample (with abuse of notation, working on $\tilde{X}_{qk}$ where $\mu_{qk}:\tilde{X}_{qk}\to X$ is a resolution of $\mathfrak{b}(|qkL|)$ which factorizes as $f_{q}\circ\mu_{k}$ where $\mu_{k}:\tilde{X}_{k}\to X$ is a resolution of $\mathfrak{b}(|kL|)$).\\
%Therefore there exists an effective divisor $E_{qk}^{'}$ such that $sM_{qk}=sq(M_{k}+A-\delta_{qk}F_{qk})+E_{qk}^{'}$. Moreover we have $E_{qk}^{'}\leq sqF_{qk}$ since 
%We conclude the proof thanks to the Proposition \ref{prop:EquiivalenceNak}.
\end{proof}
\begin{thm}
\label{thm:ContSC}
Let $p_{1},\dots,p_{N}\in X$ be different points. Then the function $
{\mathrm{N}^{1}(X)_{\mathbbm{R}}\ni L\to \epsilon_{S}(||L||;p_{1},\dots,p_{N})\in \mathbbm{R}}$ is continuous.
\end{thm}
\begin{proof}
The homogeneity and the concavity described in Proposition \ref{prop:Properties} implies the locally uniform continuity of $\epsilon_{S}(||L||; p_{1},\dots,p_{N})$ on the open convex subset $\big(\bigcup_{j=1}^{N}B_{+}(p_{j})\big)^{C} $ (see Remark \ref{rem:GlobalMOBJRem}). Thus it is sufficient to show that $\lim_{L'\to L}\epsilon_{S}(||L'||; p_{1},\dots,p_{N})=0$ if $c_{1}(L)\in \bigcup_{j=1}^{N}B_{+}(p_{j})$. But, letting $V\subset X$ be an irreducible component of $\mathbbm{B}_{+}(L)$ containing at least one of the points $p_{1},\dots,p_{N}$, we have
$$
\lim_{L'\to L}\epsilon_{S}(\|L'\|;p_{1},\dots,p_{N})\leq \lim_{L'\to L}\Bigg(\frac{\mathrm{Vol}_{X|V}(L')}{\sum_{j=1}^{N}\mathrm{mult}_{p_{j}}V}\Bigg)^{1/\dim V}=0
$$
where the inequality follows from Proposition \ref{prop:PerConSC} while the convergence is a consequence of the continuity of the restricted volume (see \cite{ELMNP09}, in particular Theorem $5.7.$ in \cite{ELMNP09}).
\end{proof}
To conclude the section we recall that for a line bundle $L$ and for an integer $s\in\mathbbm{Z}_{\geq 0}$, we say that $L$ \emph{generates the $s-$jets at $p_{1},\dots,p_{N}$} if the map
$$
H^{0}(X,L)\twoheadrightarrow \bigoplus_{j=1}^{N}H^{0}(X,L\otimes\mathcal{O}_{X,p_{j}}/\mathfrak{m}_{p_{j}}^{s+1})
$$
is surjective where we have set $\mathfrak{m}_{p_{j}}$ for the maximal ideal in $\mathcal{O}_{X,p_{j}}$. We report an useful last characterization of the moving multipoint Seshadri constant.
\begin{prop}[\cite{Ito13}, Lemma $3.10.$]
\label{prop:EqSeshadri}
Let $L$ be a big line bundle.
Then
$$
\epsilon_{S}(||L|| ;p_{1},\dots,p_{N})=\sup_{k>0}\frac{s(kL;p_{1},\dots,p_{N})}{k}=\lim_{k\to\infty}\frac{s(kL;p_{1},\dots,p_{N})}{k}
$$
where $ s(kL;p_{1},\dots,p_{N}) $ is $0$ if $kL$ does not generate then $s-$jets at $p_{1},\dots,p_{N}$ for any $s\in \mathbbm{Z}_{\geq 0}$, otherwise it is the biggest non-negative integer such that $ kL $ generates the $ s(kL;p_{1},\dots,p_{N})-$jets at $ p_{1},\dots,p_{N}$.
\end{prop}
\subsection{Proof of Theorem \ref{ThmB}}
We prove here one of our main results.
\begin{reptheorem}{ThmB}
Let $L$ be a big line bundle and let $>$ be the deglex order. Then
$$
\epsilon_{\epsilon}(\|L\|; p_{1},\dots,p_{N})=\max\big\{0,\xi(L;p_{1},\dots,p_{N})\big\}
$$
where $\xi(L;p_{1},\dots,p_{N}):=\sup\{t\geq 0\, : \, t\Sigma_{n}\subset \Delta_{j}(L)^{ess}\, \mbox{for any}\, j=1,\dots,N\}$ and $\Sigma_{n}$ is the unit $n$-symplex.
\end{reptheorem}
Namely we construct the multipoint Okounkov bodies $\Delta_{j}(L)$ from a family of valuations $\nu^{p_{j}}$ associated to a family of infinitesimal flags centered at $p_{1},\dots,p_{N}$ (see paragraphs $\S$ \ref{paragraph:ParticularValuations} and $\S$ \ref{rem:ParticularValuations}).\newline
Observe that for $N=1$, Theorem \ref{ThmB} recovers Theorem C in \cite{KL17}.\\

Before proceeding with the proof, in the spirit of the aforementioned work of Demailly \cite{Dem90}, we need to describe the moving multipoint Seshadri constant $\epsilon(||L||;p_{1},\dots,p_{N})$ in a more analytical language.
\begin{defn}
We say that a singular metric $\varphi$ of a line bundle $L$ has \emph{isolated logarithmic poles} at $p_{1},\dots, p_{N}$ of coefficient $\gamma$ if $\min\{\nu(\varphi,p_{1}),\dots,\nu(\varphi,p_{N})\}=\gamma$ and $\varphi$ is finite and continuous in a small punctured neighborhood $V_{j}\setminus \{p_{j}\}$ for every $j=1,\dots,N$. We have denoted by $\nu(\varphi,p_{j})$ the \emph{Lelong number} of $\varphi$ at $p_{j}$,
$$
\nu(\varphi,p_{j}):=\liminf_{z\to x}\frac{\varphi_{j}(z)}{\ln\lvert z-x\rvert^{2}}
$$
where $\varphi_{j}$ is the local plurisubharmonic function defining $\varphi$ around ${p_{j}=x}$.\\
We set $\gamma(L;p_{1},\dots,p_{N}):=\sup\{\gamma\in\mathbbm{R}\, : \, L $ \emph{has a positive singular metric with isolated logarithmic poles at} $ p_{1},\dots,p_{N} $ \emph{of coefficient} $ \gamma\}$
\end{defn}
Note that for $N=1$ we recover the definition given in \cite{Dem90}.
\begin{prop}
\label{prop:Poles}
Let $L$ be a big $\mathbbm{Q}-$line bundle. Then
$$
\gamma(L;p_{1},\dots,p_{N})=\epsilon_{S}(||L||;p_{1},\dots,p_{N})
$$
\end{prop}
\begin{proof}
By homogeneity we can assume $L$ to be a line bundle, and we fix a family of local holomorphic coordinates $\{z_{j,1},\dots,z_{j,n}\}$ in open coordinated sets $U_{1},\dots,U_{N}$ centered respectively at $p_{1},\dots,p_{N}$.\\
Setting $z_{j}:=(z_{j,1},\dots,z_{j,N})$ and $s:=s(kL;p_{1},\dots,p_{N})$ for $k\geq 1$ natural number, we can find holomorphic section $f_{\alpha}$, parametrized by all $\alpha=(\alpha_{1},\dots,\alpha_{N})\in\mathbbm{N}^{Nn}$ such that $\lvert \alpha_{j}\rvert = s$ and $f_{\alpha\vert U_{j}}=z_{j}^{\alpha_{j}}$ for any $j=1,\dots,N$. In other words, we can find holomorphic sections of $kL$ whose jets at $p_{1},\dots,p_{N}$ generates all possible combination of monomials of degree $s$ around the points chosen. Thus the positive singular metric $\varphi$ on $L$ given by
$$
\varphi:=\frac{1}{k}\log\Big(\sum_{\alpha}\lvert f_{\alpha} \rvert^{2}\Big)
$$
has isolated logarithmic poles at $p_{1},\dots,p_{N}$ of coefficient $s/k$. Hence $\gamma(L;p_{1},\dots,p_{N})\geq s(kL;p_{1},\dots,p_{N})/k$, and letting $k\to\infty$ Proposition \ref{prop:EqSeshadri} yields $\gamma(L;p_{1},\dots,p_{N})\geq \epsilon_{S}(||L|| ;p_{1},\dots,p_{N})$.\\
Conversely, assuming $\gamma(L;p_{1},\dots,p_{N})>0$, let $\{\gamma_{t}\}_{t\in\mathbbm{N}}\subset\mathbbm{Q}$ be an increasing sequence of rational numbers converging to $\gamma(L;p_{1},\dots,p_{N})$ and let $\{k_{t}\}_{t\in\mathbbm{N}}$ be an increasing sequence of natural numbers such that $\{k_{t}\gamma_{t}\}_{t\in \mathbbm{N}}$ converges to $+\infty$. Moreover let $A$ be an ample line bundle such that $A-K_{X}$ is ample, and let $\omega=dd^{c}\phi$ be a K\"ahler form in the class $c_{1}(A-K_{X})$.\\
Thus for any positive singular metric $\varphi_{t}$ of $L$ with isolated logarithmic poles at $p_{1},\dots,p_{N}$ of coefficient $\geq \gamma_{t}$, $k_{t}\varphi_{t}+\phi$ is a positive singular metric of $k_{t}L+A-K_{X}$ with K\"ahler current $dd^{c}(k_{t}\varphi_{t})+\omega$ as curvature and with isolated logarithmic poles at $p_{1},\dots,p_{N}$ of coefficient $\geq k_{t}\gamma_{t}$. Therefore, for $t\gg 1$ big enough, $k_{t}L_{t}+A$ generates all $(k_{t}\gamma_{t}-n)-$jets at $p_{1},\dots,p_{N}$ by Corollary 3.3. in \cite{Dem90}, and thanks to Proposition \ref{prop:EqSeshadri} we obtain
$$
\epsilon_{S}\Big(\big\|L+\frac{1}{k_{t}}A\big\|;p_{1},\dots,p_{N}\Big)\geq\frac{k_{t}\gamma_{t}-n}{k_{t}}=\gamma_{t}-\frac{n}{k_{t}}.
$$
Letting $t\to \infty$ we get $\epsilon_{S}(||L||;p_{1},\dots,p_{N})\geq \gamma(L;p_{1},\dots,p_{N})$ exploiting the continuity of Theorem \ref{thm:ContSC}.
\end{proof}
%\begin{rem}
%\emph{For the reverse inequality, in the nef case, we could have proceeded in the following way. Suppose $L$ endowed of a positive singular metric $\varphi$ with isolated logarithmic pole at $p_{1},\dots,p_{N}$ of coefficient $\gamma$ and let $C$ be a irreducible curve passing at least one points of $p_{1},\dots,p_{N}$.\\
%Then, since $\varphi$ is locally integrable along $C$ around $p_{j}$ for any $j=1,\dots,N$, we get
%$$
%L\cdot C \geq \sum_{j=1}^{N}\int_{C\cap U_{j}}dd^{c}\varphi\geq\sum_{j=1}^{N}\nu(\varphi,p_{j})\mathrm{mult}_{p_{j}}C\geq\gamma\sum_{j=1}^{N}\mathrm{mult}_{p_{j}}C
%$$
%where, as in the proof of the Proposition \ref{prop:EqSeshadri}, the second inequality follows from the comparison principle for the generalized Lelong number of the current $[C]$ respect to the plurisubharmonic function $\varphi$ (see \cite{Dem87}). Thus $\epsilon_{S}(L;p_{1},\dots,p_{N})\geq \gamma$, and by arbitrariety of $\gamma$ we get $\epsilon_{S}( L;p_{1},\dots,p_{N})\geq \gamma(L;p_{1},\dots,p_{N})$.}
%\end{rem}
\begin{rem}
\emph{Observe that the same result cannot hold if we restrict ourselves to considering metrics with logarithmic poles at $p_{1},\dots,p_{N}$ not necessarily isolated. Indeed Demailly in \cite{Dem93} showed that for any nef and big $\mathbbm{Q}-$line bundle $L$ over a projective manifold, for any different points $p_{1},\dots,p_{N}$, and for any $\tau_{1},\dots,\tau_{N}$ positive real numbers with $\sum_{j=1}^{N}\tau_{j}^{n}<(L^{n})$ there exists a positive singular metric $\varphi$ with logarithmic poles at any $p_{j}$ of coefficients, respectively, $\tau_{j}$. We thus conclude that the result in Proposition \ref{prop:Poles} holds considering metrics with logarithmic poles at $p_{1},\dots,p_{N}$ not necessarily isolated if and only if the multipoint Seshadri constant is maximal, i.e. $\epsilon_{S}(||L||,p_{1},\dots,p_{N})=(\mathrm{Vol}_{X}(L)/N)^{1/n}$.}
%We recall that Tosatti in \cite{Tos16} extended this result to any big and nef class $\alpha$ (in the Bott-Chern cohomology, in the sense of Demailly) on any compact complex manifold.
\end{rem}
%\begin{reptheorem}{ThmB}
%Let $L$ be a big $\mathbbm{R}-$line bundle, then
%$$
%\max\big\{\xi(L;p_{1},\dots,p_{N}),0\big\}=\epsilon_{S}(||L||;p_{1},\dots,p_{N}).
%$$
%\end{reptheorem}
\begin{proof}[Proof of Theorem \ref{ThmB}]
By the continuity of Theorems \ref{thm:GlobalMOBJ}, \ref{thm:ContSC} and by the homogeneity of both sides we can assume $L$ big line bundle. Moreover we may also assume $\Delta_{j}(L)^{\circ}\neq \emptyset$ for any $j=1,\dots,N$ since otherwise the statement is a consequence of Lemma \ref{lem:IntOk}.(ii).\\
Let $\{\lambda_{m}\}_{m\in\mathbbm{N}}\subset \mathbbm{Q}_{>0}$ be an increasing sequence convergent to $\xi(L;p_{1},\dots,p_{N})>0$. By Proposition \ref{prop:Essential}, for any $m\in\mathbbm{N}$ there exist $k_{m}\gg 1$ such that $\lambda_{m}\Sigma_{n}\subset\Delta_{j}^{k_{m}}(L)^{\mathrm{ess}}$ for any ${j=1,\dots,N}$. Therefore, chosen a set of section $\{s_{j,\alpha}\}_{j,\alpha} \subset H^{0}(X,k_{m}L)$ parametrized in a natural way by all valuative points in $\Delta_{j}^{k_{m}}(L)^{\mathrm{ess}}\setminus \lambda_{m}\Sigma_{n}^{\mathrm{ess}}$ for any $j=1,\dots,N$ (i.e. $s_{j,\alpha}\in V_{k_{m},j}$, $\nu^{p_{j}}(s_{j,\alpha})=\alpha$ and $\alpha\notin \lambda_{m}\Sigma_{n}^{\mathrm{ess}}$) the metric
$$
\varphi_{k_{m}}:=\frac{1}{k_{m}}\ln\Big(\sum_{j=1}^{N}\sum_{\alpha}\lvert s_{j,\alpha}\rvert^{2}\Big)
$$
is a positive singular metric on $L$ such that $\nu(\varphi_{k_{m}},p_{j})\geq\lambda_{m}$ while $\varphi_{k_{m}}$ is continuos and finite on a punctured neighborhood $V_{j}\setminus\{p_{j}\}$ for any $j=1,\dots,N$ by Corollary \ref{cor:AnotherEqMOB}. Hence letting $m\to \infty,$ we get $\epsilon_{S}(||L||;p_{1},\dots,p_{N})=\gamma(L;p_{1},\dots,p_{N})\geq \xi(L;p_{1},\dots,p_{N})$, where the equality is the content of Proposition \ref{prop:Poles}.\\
On the other hand, letting $\{\lambda_{m}\}_{m\in\mathbbm{N}}\subset\mathbbm{Q}$ be a increasing sequence converging to $\epsilon_{S}(||L|| ;p_{1},\dots,p_{N})>0$, Proposition \ref{prop:EqSeshadri} implies that for any $m\in\mathbbm{N}$ there exists $k_{m}\gg 0$ divisible enough such that $s(tk_{m}L;p_{1},\dots,p_{N})\geq tk_{m}\lambda_{m}$ for any $t\geq 1$. Thus, since the family of valuation is associated to a family of infinitesimal flags, we get
$$
\frac{\lceil tk_{m}\lambda_{m}\rceil}{tk_{m}}\Sigma_{n}\subset \Delta_{j}^{k_{m}}(L)^{\mathrm{ess}}\subset\Delta_{j}(L)^{\mathrm{ess}} \,\, \forall \,j=1,\dots,N \, \mathrm{and} \,\,\forall\, t\geq 1.
$$
Hence $ \lambda_{m}\Sigma_{n}\subset \Delta_{j}(L)^{\mathrm{ess}}$ for any $j=1,\dots,N $, which concludes the proof.
\end{proof}
\begin{rem}
\emph{In the case $L$ ample line bundle, to prove the inequality ${\epsilon_{S}(L;p_{1},\dots,p_{N})\geq \xi(L;p_{1},\dots,p_{N})}$ we could have used Theorem \ref{ThmC}. In fact by definition $\sqrt{\xi(L;p_{1},\dots,p_{N})}={\sup\{r>0} \, : \, B_{r}(0)\subset \Omega_{j}(L) \, \mathrm{for} \, \mathrm{any} \, j=1,\dots,N \}$ where we recall that $\Omega_{j}(L):=\mu^{-1}\big(\Delta_{j}(L)^{ess}\big)$ for $\mu(z_{1},\dots,z_{n}):=(|z_{1}|^{2},\dots,|z_{n}|^{2})$. Thus Theorem \ref{ThmC} implies that $\{(B_{\sqrt{\xi(L;p_{1},\dots,p_{N})}-\epsilon}(0),\omega_{st})\}_{j=1}^{N}$ packs into $(X,L)$ for any $\epsilon>0$ small enough, and so the symplectic blow-up procedure for K\"ahler manifold (see section \S 5.3. in \cite{MP94}, or Lemma 5.3.17. in \cite{Laz04}) yields $\xi(L;p_{1},\dots,p_{N})\leq \epsilon_{S}(L;p_{1},\dots,p_{N})$.}
\end{rem}
\begin{rem}
\label{rem:SeshadriFitted}
\emph{The proof of the Theorem \ref{ThmB} shows that $\xi(L;p_{1},\dots,p_{N})$ is independent from the choice of the family of valuations given by the associated infinitesimal flags.}
\end{rem}
The following corollary, which is an immediate consequence of Theorems \ref{ThmB}, \ref{ThmC} extends Theorem $0.5$ in \cite{Eckl17} to all dimensions (as Eckl claimed in his paper) and to big line bundles.
\begin{cor}
\label{cor:EcklGen}
Let $L$ be a big line bundle. Then
\begin{multline*}
\sqrt{\epsilon_{S}(||L|| ;p_{1},\dots,p_{N})}=\max\Big\{0,\sup\{r>0 \, : \, B_{r}(0)\subset \Omega_{j}(L)\,\,\forall \, j=1,\dots,N\}\Big\}=\\
=\max\Big\{0,\sup\big\{r>0\, : \, \{(B_{r}(0),\omega_{std}\big)\}_{j=1}^{N}\, \mbox{packs into} \, (X,L)\,\big\}.
\end{multline*}
\end{cor}
For $N=1$ it is the content of Theorem $1.3.$ in \cite{WN15}.
\section{Some particular cases}
\label{section:ParticularCases}
\subsection{Projective toric manifolds}
\label{subsection:Toric}
In this section $X=X_{\Delta}$ is a smooth projective toric variety associated to a fan $\Delta $ in $N_{\mathbbm{R}}\simeq \mathbbm{R}^{n}$, so that the torus $T_{N}:=N\otimes_{\mathbbm{Z}}\mathbbm{C}^{*}\simeq (\mathbbm{C}^{*})^{n}$ acts on $X$ ($N\simeq \mathbbm{Z}^{n}$ denote a lattice of rank $n$ with dual $M:=\mathrm{Hom}_{\mathbbm{Z}}(N,\mathbbm{Z})$, see \cite{Ful93}, \cite{Cox11} for notation and basic fact about toric varieties).\\
It is well-known that there is a correspondence between toric manifolds $X$ polarized by $T_{N}-$invariant ample divisors $D$ and lattice \emph{polytopes} $P\subset M_{\mathbbm{R}}$ of dimension $n$. Indeed to any such divisor ${D=\sum_{\rho\in \Delta(1)}a_{\rho}D_{\rho}}$, denoting by $\Delta(k)$ the cones of dimension $k$, the polytope $P_{D}$ is given by $ P_{D}:=\bigcap_{\rho\in\Delta(1)}\{m\in M_{\mathbbm{R}} \, : \, \langle m,v_{\rho}\rangle\geq -a_{\rho}\} $ where $v_{\rho}$ represents the generator of $\rho\cap N$. Conversely any such polytope $P$ can be described as $P:=\bigcap_{F \, \mathrm{facet}}\{m\in M_{\mathbbm{R}} \, : \, \langle m,n_{F}\rangle\geq -a_{F}\} $ where a \emph{facet} is a $1-$codimensional face of $P$ and $n_{F}\in N$ is the unique primitive element that is normal to $F$ and that points toward the interior of $P$. Thus the \emph{normal fan associated to $P$} is $\Delta_{P}:=\{\sigma_{\mathcal{F}}\, : \, \mathcal{F} \, \mathrm{face} \, \mathrm{of}\, P \}$ where $\sigma_{\mathcal{F}}$ is the cone in $\mathbbm{N}_{\mathbbm{R}}$ generated by all normal elements $n_{F}$ as above for any facet containing the face $\mathcal{F}$. In particular vertices of $P$ correspond to $T_{N}-$invariant points on the toric manifold $X_{P}$ associated to $\Delta_{P}$ while facets of $P$ correspond to $T_{N}-$invariant divisors on $X_{P}$. Finally the polarization is given by $D_{P}:=\sum_{F \, \mathrm{facet}}a_{F}D_{F}$.\\

Thus, given an ample toric line bundle $L=\mathcal{O}_{X}(D)$ on a projective toric manifold $X$ we can fix local holomorphic coordinates around a $T_{N}-$invariant point $p\in X$ (corresponding to a vertex $x_{\sigma}\in P$) such that $\{z_{i}=0\}=D_{i\vert U_{\sigma}}$ for $D_{i}$ $T_{N}-$invariant divisors and we can assume $ D_{\vert U_{\sigma}}=0 $.
\begin{prop}[\cite{LM09},Proposition 6.1.(i)]
\label{prop:OBToric}
In the setting as above, the equality
$$
\phi_{\mathbbm{R}^{n}}(P_{D})= \Delta(L)
$$
holds, where $\phi_{\mathbbm{R}}$ is the linear map associated to $\phi: M\to \mathbbm{Z}^{n}, \phi(m):=(\langle m, v_{1} \rangle,\dots,\langle m, v_{n} \rangle)$, for $v_{i}\in \Delta_{P_{D}}(1)$ generators of the ray associated to $D_{i}$, and $\Delta(L)$ is the one-point Okounkov body associated to the admissible flag given by the local holomorphic coordinates chosen.
\end{prop}
%Moreover we can describe the blow-up around one $T_{N}-$invariant point $x_{\sigma}$ directly from the fan in the following way (recall that a \emph{toric morphism} $f:X_{\Delta'}\to X_{\Delta}$ correspond to a group homomorphism $f_{N}:N\to N$ such that $f_{N,\mathbbm{R}}: N_{\mathbbm{R}}\to N_{\mathbbm{R}}$ maps each cone in $\Delta'$ into a cone of $\Delta$):
%\begin{prop}[\cite{Ful93},Sec.2.4]
%\label{prop:Fulton}
%Let $X=X_{\Delta}$ be a projective toric manifold and let $\sigma\in \Delta$ be a cone of maximal dimension $n$ corresponding to the $T_{N}-$fixed point $x_{\sigma}$. Then the blow-up of $X$ in $x_{\sigma}$ is given by a toric morphism $X_{\Delta'}\to X_{\Delta}$ where $\Delta'$ is constructed from $\Delta$ by subdividing $\sigma$ into $n$ cones $\sigma_{i}$ generated by
%$$
%v_{1},\dots, v_{i-1},v_{1}+\dots,v_{n},v_{i+1},\dots, v_{n}
%$$
%where $v_{1},\dots,v_{n}\in N$ is a $\mathbbm{Z}$-basis of the lattice $N$ which spans $\sigma$. The exceptional divisor on $X_{\Delta'}$ is $T_{N}-$invariant and corresponds to the ray $\tau$ generated by $v_{1}+\dots+ v_{n}$
%\end{prop}
Moreover we recall that it is possible to describe the positivity of the toric line bundle at a $T_{N}-$invariant point $x_{\sigma}$ corresponding to a vertex in $P$ directly from the polytope.
\begin{lem}\emph{(Lemma 4.2.1, \cite{BDRH$^{+}$09})}
Let $(X,L)$ be a toric polarized manifold, and let $P$ be the associated polytope with vertices $x_{\sigma_{1}},\dots,x_{\sigma_{l}}$. Then $L$ generates the $k-$jets at $x_{\sigma_{j}}$ iff the length $\lvert e_{j,i}\rvert $ is bigger than $k$ for any $i=1,\dots,n$ where $e_{j,i}$ is the edge connecting $x_{\sigma_{j}}$ to another vertex $x_{\sigma_{\tau(i)}}$.
\end{lem}
\begin{rem}
\label{rem:Delzant}
\emph{By assumption, we know that $P$ is a} Delzant \emph{polypote, i.e. there are exactly $n$ edges originating from each vertex, and the first integer points on such edges form a lattice basis (for} integer \emph{we mean a point belonging in M). Moreover if one fixes the first integer points on the edges starting from a vertex $x_{\sigma}$ (i.e. a basis for $M\mathop{\simeq}\mathbbm{Z}^{n}$), then the length of an edge starting from $x_{\sigma}$ is defined as the usual length in $\mathbbm{R}^{n}$. Observe that the length of any edge is an integer since the polytope is a \emph{lattice} polytope.}
\end{rem}
Similarly to Proposition \ref{prop:OBToric}, chosen $R$ $T_{N}-$invariants points corresponding to $R$ vertices of the polytope $P$, we retrieve the multipoint Okounkov bodies of the corresponding $R$ $T_{N}-$invariant points on $X$ \emph{directly} from the polytope:
\begin{thm}
\label{thm:ToricSubdivision}
Let $(X,L)$ be a toric polarized manifold, and let $P$ be the associated polytope with vertices $x_{\sigma_{1}},\dots,x_{\sigma_{l}}$ corresponding, respectively, to the $T_{N}-$points $p_{1},\dots,p_{l}$. Then for any choice of $R$ different points ($R\leq l$) $p_{i_{1}},\dots, p_{i_{R}}$ among $p_{1},\dots,p_{l}$, there exists a subdivision of $P$ into $R$ polytopes (a priori not \emph{lattice} polytopes) $P_{1},\dots,P_{R}$ such that $\phi_{\mathbbm{R}^{n},j}(P_{j})=\Delta_{j}(L)$ for a suitable choice of a family of valuations associated to infinitesimal (toric) flags centered at $p_{i_{1}},\dots,p_{i_{R}}$, where $\phi_{\mathbbm{R}^{n},j}$ is the map given in the Proposition \ref{prop:OBToric} for the point $x_{\sigma_{j}}$.
\end{thm}
\begin{proof}
Unless reordering, we can assume that the $T_{N}-$invariants points $p_{1},\dots, p_{R}$ correspond to the vertices $x_{\sigma_{1}},\dots,x_{\sigma_{R}}$.\\
Next for any $j=1,\dots,R$, after the identification $M\simeq\mathbbm{Z}^{n}$ given by the choice of a lattice basis $m_{j,1},\dots,m_{j,n}$ as explained in Remark \ref{rem:Delzant}, we retrieve the Okounkov Body $\Delta(L)$ at $p_{j}$ associated to an infinitesimal flag given by the coordinates $\{z_{1,j},\dots,z_{n,j}\}$ as explained in Proposition \ref{prop:OBToric} composing with the map $\phi_{\mathbbm{R}^{n},j}$. Thus, by construction, we know that any valuative point lying in the diagonal face of the $n-$symplex $\delta\Sigma_{n}$ for $\delta\in \mathbbm{Q}_{\geq 0}$ corresponds to a section $s\in H^{0}(X,kL)$ such that $\mathrm{ord}_{p_{j}}(s)=k\delta$. Working directly on the polytope $P$, the diagonal face of the $n-$symplex $\delta\Sigma_{n}$ corresponds to the intersection of the polytope $P$ with the hyperplane $H_{\delta,j}$ parallel to the hyperplane passing for $m_{1,j},\dots,m_{n,j}$ and whose distance from the point $x_{\sigma_{j}}$ is equal to $\delta$ (the \emph{distance} is calculated from the identification $M\simeq \mathbbm{Z}^{n}$).\\
Therefore defining
\begin{equation*}
P_{j}:=\overline{\bigcup_{(\delta_{1},\dots,\delta_{n})\in\mathbbm{Q}^{n}_{\geq 0}, \delta_{j}<\delta_{i} \, \forall i\neq j }H_{\delta_{1},1}\cap\dots\cap H_{\delta_{R},R}\cap P}=\overline{\bigcup_{(\delta_{1},\dots,\delta_{n})\in\mathbbm{Q}^{n}_{\geq 0}, \delta_{j}\leq\delta_{i} \, \forall i\neq j }H_{\delta_{1},1}\cap\dots\cap H_{\delta_{R},R}\cap P}
\end{equation*}
we get by Proposition \ref{prop:First} $ \phi_{\mathbbm{R}^{n},j}(P_{j})=\Delta_{j}(L) $ since any valuative point in $H_{\delta_{1},1}\cap\cdots\cap H_{\delta_{R},R}\cap P$ belongs to $\Delta_{j}(L)$ if $\delta_{j}<\delta_{i}$ for any $i\neq j$, while on the other hand any valuative point in $\Delta_{j}(L)$ belongs to $H_{\delta_{1},1}\cap\cdots\cap H_{\delta_{R},R}\cap P$ for certain rational numbers $\delta_{1},\dots,\delta_{R}$ such that $\delta_{j}\leq \delta_{i}$.
\end{proof}
\begin{rem}
\emph{As an easy consequence, we get that for any polarized toric manifold $(X,L)$ and for any choice of $R$ $T_{N}-$invariants points $p_{1},\dots,p_{R}$, the multipoint Okounkov bodies constructed from the infinitesimal flags as in Theorem \ref{thm:ToricSubdivision} are polyhedral.}
\end{rem}
\begin{cor}
\label{cor:Barycentric}
In the same setting of the Theorem \ref{thm:ToricSubdivision}, if $R=l$, then the subdivision is \emph{barycenteric}. Namely, for any fixed vertex $x_{\sigma_{j}}$, if $F_{1},\dots,F_{n}$ are the facets containing $x_{\sigma_{j}}$ and $b_{1},\dots,b_{n}$ are their respective barycenters, then the polytope $P_{j}$ is the convex body defined by the intersection of $P$ with the $n$ hyperplanes $H_{O,j}$ passing through the baricenter $O$ of $P$ and the barycenters $b_{1},\dots,b_{j-1},b_{j+1},\dots,b_{n}$.
\end{cor}
Finally we retrieve and extend Corollary 2.3. in \cite{Eckl17} as consequence of Theorem \ref{thm:ToricSubdivision} and Theorem \ref{ThmB}:
\begin{cor}
\label{cor:HalfSesh}
In the same setting of the Theorem \ref{thm:ToricSubdivision}, for any $j=1,\dots,R$, let $ \epsilon_{S,j}:=\min_{i=1,\dots,n}\{\delta_{j,i}\} $ be the minimum among all the \emph{reparametrized} length $\lvert e_{j,i}\rvert$ of the edges $e_{j,i}$ for $i=1,\dots,n$, i.e. $\delta_{j,i}:=\lvert e_{j,i}\rvert$ if $e_{j,i}$ connect $x_{\sigma_{j}}$ to another point $x_{\sigma_{i}}$ corresponding to a point $p\notin\{p_{1},\dots,p_{R}\}$, while $\delta_{j,i}:=\frac{1}{2}\lvert e_{j,i}\rvert$ if $e_{j,i}$ connect to a point $x_{\sigma_{i}}$ corresponding to a point $p\in\{p_{1},\dots,p_{R}\}$. Then
$$
\epsilon_{S}(L;p_{1},\dots,p_{R})=\min\{\epsilon_{S,j} : \, j=1,\dots,R\}
$$
In particular $\epsilon_{S}(L;p_{1},\dots,p_{R})\in \frac{1}{2}\mathbbm{N}$.
\end{cor}
\subsection{Surfaces}
\label{subsection:Surfaces}
When $X$ has dimension $2$, the following famous decomposition holds.
\begin{thm}[Zariski decomposition]
Let $L$ be a pseudoeffective $\mathbbm{Q}-$line bundle on a surface $X$. Then there exist $\mathbbm{Q}-$line bundles $P, N$ such that
\begin{itemize}
\item[i)] $L=P+N$;
\item[ii)] $P$ is nef;
\item[iii)] $N$ is effective;
\item[iv)] $H^{0}(X,kP)\simeq H^{0}(X,kL)$ for any $k\geq 1$;
\item[v)] $P\cdot E=0$ for any $E$ irreducible curve contained in $\mathrm{Supp}(N)$.
\end{itemize}
\end{thm}
Moreover we recall that by the main theorem of \cite{BKS04} there exists a locally finite decomposition of the big cone into rational polyhedral subcones (\emph{Zariski chambers}) such that in each interior of these subcones the negative part of the Zariski decomposition has constant support and the restricted and augmented base loci are equal (i.e. the divisors with cohomology classes in a interior of some Zariski chambers are \emph{stable}, see \cite{ELMNP06}).\\
Similarly to Theorem 6.4. in \cite{LM09} and the first part of Theorem B in \cite{KLM12} we describe the multipoint Okounkov bodies as follows:
\begin{thm}
\label{thm:Surfaces}
Let $L$ be a big line bundle over a surface $X$, let $p_{1},\dots,p_{N}\in X$, and let $\nu^{p_{j}}$ be a family of valuations associated to admissible flags centered at $p_{1},\dots,p_{N}$ with $Y_{1,j}=C_{i\vert U_{p_{j}}}$ for irreducible curves $C_{j}$, $j=1,\dots,N$. Then for any $j=1,\dots,N$ such that $\Delta_{j}(L)^{\circ}\neq \emptyset$ there exist piecewise linear functions ${\alpha_{j},\beta_{j}: [t_{j,-},t_{j,+}]\to \mathbbm{R}_{\geq 0}}$ for
\begin{multline*}
0\leq t_{j,-}:=\inf\{t\geq 0 \, : \, C_{j}\not\subset\mathbbm{B}_{+}(L-t\mathbbm{G})\}<t_{j,+}:=\sup\{t\geq 0 \, : \, C_{j}\not\subset\mathbbm{B}_{+}(L-t\mathbbm{G})\}\leq\\
\leq \mu(L;\mathbbm{G}):=\sup\{t\geq 0 \, : \, L-t\mathbbm{G}\, \mathrm{is}\, \mathrm{big}\}
\end{multline*}
where $\mathbbm{G}=\sum_{j=1}^{N}C_{j}$, with $\alpha_{j}$ convex and $\beta_{j}$ concave, $\alpha_{j}\leq \beta_{j}$, such that
$$
\Delta_{j}(L)=\{(t,y)\in\mathbbm{R}^{2}\, : \, t_{j,-}\leq t\leq t_{j,+} \, \mathrm{and}\, \alpha_{j}(t)\leq y\leq \beta_{j}(t)\}
$$
In particular $\Delta_{j}(L)$ is polyhedral for any $j=1,\dots,N$.
\end{thm}
\begin{proof}
By Lemma \ref{lem:IntOk} and Theorem \ref{ThmA} we may assume $\Delta_{j}(L)^{\circ}\neq\emptyset$ for any $j=1,\dots,N$ unless removing some of the points. % and setting $t_{j,-}$, $t_{j,+}$ as before, observing that we could have defined $t_{j,-},t_{j,+}$ in a similar way considering $\mathbbm{B}_{-}(L-t\mathbbm{G})$ thanks to Theorem \ref{thm:Slices} and Proposition $1.26.$ in \cite{ELMNP06}.
Then by Theorem A and C in \cite{ELMNP09} it follows that ${0\leq t_{j,-}<t_{j,+}\leq \mu(L;\mathbbm{G})}$ and that $[t_{j,-},t_{j,+}]\times \mathbbm{R}_{\geq 0}$ is the smallest vertical strip containing $\Delta_{j}(L)$. Thus by Theorem \ref{thm:Slices} and Lemma 6.3. in \cite{LM09} we easily obtain $\Delta_{j}(L)=\{(t,y)\in\mathbbm{R}^{2}\, : \, t_{j,-}\leq t\leq t_{j,+} \, \mathrm{and}\, \alpha_{j}(t)\leq y\leq \beta_{j}(t)\}$ defining $\alpha_{j}(t):=\mathrm{ord}_{p_{j}}(N_{t\vert C_{j}})$ and $\beta_{j}(t):=\mathrm{ord}_{p_{j}}(N_{t\vert C_{j}})+(P_{t}\cdot C_{j})$ for $P_{t}+N_{t}$ Zariski decomposition of $L-t\mathbbm{G}$ ($N_{t}$ can be restricted to $C_{j}$ since $\mathrm{Supp}(N_{t})=\mathbbm{B}_{-}(L-t\mathbbm{G})$).\\
Next we proceed similarly to \cite{KLM12} to show the polyhedrality of $\Delta_{j}(L)$, i.e. we set $L':=L-t_{j,+}\mathbbm{G}$, $s=t_{j,+}-t$ and consider  ${L'_{s}:=L'+s\mathbbm{G}=L-t\mathbbm{G}}$ for $s\in[0,t_{j,+}-t_{j,-}]$. Thus the function ${s\to N'_{s}} $ is decreasing, i.e. $N'_{s'}-N'_{s}$ is effective for any ${0\leq s'<s\leq t_{j,+}-t_{j,-}}$, where $L'_{s}=P'_{s}+N'_{s}$ is the Zariski decomposition of $L'_{s}$. Moreover, letting $F_{1},\dots,F_{r}$ be the irreducible (negative) curves composing $N_{0}'$, we may assume (unless rearraging the $F_{i}$'s) that the support of $N_{t_{j,+}-t_{j,-}}'$ consists of $F_{k+1},\dots,F_{r}$ and that $0=:s_{0}< s_{1}\leq \cdots\leq s_{k}\leq t_{j,+}-t_{j,-}=: s_{k+1}$ where $s_{i}:=\sup\{s\geq 0 \, : \, F_{i}\subset \mathbbm{B}_{-}(L'_{s})=\mathrm{Supp}(N'_{s})\}$ for any $i=1,\dots,k$.\\
So, by the continuity of the Zariski decomposition in the big cone, it is enough to show that $N'_{s}$ is linear in any not-empty open interval $(s_{i},s_{i+1})$ for $i\in \{0,\dots,k\}$. But the Zariski algorithm implies that $N'_{s}$ is determined by $N'_{s}\cdot F_{l}=(L'+s\mathbbm{G})\cdot F_{l}$ for any $l=i+1,\dots,r$, and, since the intersection matrix of the curves $F_{i+1},\dots,F_{r}$ is non-degenerate, we know that there exist unique divisors $A_{i}$ and $B_{i}$ supported on $\cup_{l=i+1}^{r}F_{l}$ such that $A_{i}\cdot F_{l}=L'\cdot F_{l}$ and $B_{i}\cdot F_{l}=\mathbbm{G}\cdot F_{l}$ for any $l=i+1,\dots,r$. Hence $N'_{s}=A_{i}+sB_{i}$ for any $s\in(s_{i},s_{i+1})$, which concludes the proof.
\end{proof}
\begin{rem}
\emph{We observe that $\Delta_{j}(L)\cap [0,\mu(L;\mathbbm{G})-\epsilon]\times \mathbbm{R}$ is} rational \emph{polyhedral for any $0<\epsilon < \mu(L;\mathbbm{G})$ thanks to the proof and to the main theorem in \cite{BKS04}.} 
\end{rem}
A particular case is when $p_{1},\dots,p_{N}\notin\mathbbm{B}_{+}(L)$ and $\nu^{p_{j}}$ is a family of valuations associated to infinitesimal flags centered respectively at $p_{1},\dots,p_{N}$. Indeed in this case on the blow-up $\tilde{X}=\mathrm{Bl}_{\{p_{1},\dots,p_{N}\}}X$ we can consider the family of valuations $\tilde{\nu}^{\tilde{p}_{j}}$ associated to the admissible flags centered respectively at points $\tilde{p}_{1},\dots,\tilde{p}_{N}\in \tilde{X}$ (see paragraph $\S$\ref{paragraph:ParticularValuations}). Observe that $\tilde{Y}_{1,j}=E_{j}$ are the exceptional divisors over the points.
\begin{lem}
\label{lem:Exceptional}
In the setting just mentioned, we have $t_{j,-}=0$ and $t_{j,+}=\mu(f^{*}L;\mathbbm{E})$ where $\mathbbm{E}=\sum_{i=1}^{N}E_{i}$ and $f:\tilde{X}\to X$ is the blow-up map.
\end{lem}
\begin{proof}
Theorem $\ref{ThmB}$ easily implies $t_{j,-}=0$ for any $j=1,\dots,N$ since $p_{1},\dots,p_{N}\notin\mathbbm{B}_{+}(L)$ and $F(\Delta_{j}(L))=\Delta_{j}(f^{*}L)$ for ${F(x_{1},x_{2})=(x_{1}+x_{2},x_{1})}$.\\
Next assume by contradiction there exists $j\in\{1,\dots,N\}$ such that $t_{j,+}<\mu(f^{*}L;\mathbbm{E})$. Then by Theorem \ref{thm:Slices} and Theorem $A$ and $C$ in \cite{ELMNP09} we get ${\bar{t}:=\sup\{t\geq 0\, : \,} {E_{j}\not\subset \mathbbm{B}_{+}(f^{*}L-t\mathbbm{E})\}}=\sup\{t\geq 0\, : \, E_{j}\not\subset \mathbbm{B}_{-}(f^{*}L-t\mathbbm{E})\}<\mu(f^{*}L;\mathbbm{E})$. Therefore setting $L_{t}:=f^{*}L-t\mathbbm{E}$ and letting $L_{t}=P_{t}+N_{t}$ be its Zariski decomposition, we get that $E_{j}\in \mathrm{Supp}(N_{t})$ iff $t>\bar{t}$ (see Proposition $1.2.$ in \cite{KL15a}). But for any $\bar{t}<t<\mu(f^{*}L;\mathbbm{E})$ we find out
$$
0=(L_{t}+t\mathbbm{E})\cdot E_{j}=L_{t}\cdot E_{j}+tE_{j}^{2}< -t
$$
where the first equality is justified by $P_{t}+N_{t}+t\mathbbm{E}=f^{*}L$ while the inequality is a consequence of $L_{t}\cdot E_{j}<0$ (since $E_{j}\in\mathrm{Supp}(N_{t})$) and of $E_{i}\cdot E_{j}=-\delta_{i,j}$. Hence we obtain a contradiction.
\end{proof}
\paragraph{About the Nagata's Conjecture:}
One modern version of the Nagata's conjecture says that for a choice of very general points ${p_{1},\dots,p_{N}\in\mathbbm{P}^{2}}$ and $N\geq 9$, the ample line bundle $\mathcal{O}_{\mathbbm{P}^{2}}(1)$ has maximal multipoint Seshadri constant at $p_{1},\dots,p_{N}$, i.e. $\epsilon_{S}(\mathcal{O}_{\mathbbm{P}^{2}}(1); N)=1/\sqrt{N}$ where to simplify the notation we did not report the points since they are very general. Thanks to Theorems \ref{ThmA}, \ref{ThmB}, we can then read the Nagata's conjecture in the following way.
\begin{conget}[\cite{Nag58}, Nagata's Conjecture]
For $N\geq 9$ very general points in $\mathbbm{P}^{2}$, let $\{\Delta_{j}(\mathcal{O}_{\mathbbm{P}^{2}}(1))\}_{j=1}^{N}$ be the multipoint Okounkov bodies calculated from a family of valuations $\nu^{p_{j}}$ associated to a family of infinitesimal flags centered respectively at $p_{1},\dots,p_{N}$. Then the following equivalent statements hold:
\begin{itemize}
\item[i)] $\epsilon_{S}(\mathcal{O}_{\mathbbm{P}^{2}}(1); N)=1/\sqrt{N}$;
\item[ii)] $\Delta_{j}(\mathcal{O}_{\mathbbm{P}^{2}}(1))=\frac{1}{\sqrt{N}}\Sigma_{2}$, where $\Sigma_{2}$ is the standard $2-$symplex;
\item[iii)] $\Omega_{j}(\mathcal{O}_{\mathbbm{P}^{2}}(1))=B_{\frac{1}{\sqrt{N}}}(0)$;
\end{itemize}
\end{conget}
\begin{rem}
\emph{It is well know that the conjecture holds if $N\geq 9$ is a perfect square. And a similar conjecture (called Biran-Nagata-Szemberg's conjecture) claims that for any ample line bundle $L$ on a projective manifold of dimension $n$ there exist $N_{0}=N_{0}(X,L)$ big enough such that $\epsilon_{S}(L;N)=\sqrt[n]{\frac{L^{n}}{N}}$ for any $N\geq N_{0}$ very general points, i.e. it is maximal. This conjecture can be similarly read through the multipoint Okounkov bodies as $\Delta_{j}(L)=\sqrt[n]{\frac{L^{n}}{N}}\Sigma_{n}$ for any $N\geq N_{0}$ very general points at $X$.}
\end{rem}
\begin{thm}
\label{thm:ProjPlane}
For $N\geq 9$ very general points in $\mathbbm{P}^{2}$, there exists a family of valuations $\nu^{p_{j}}$ associated to a family of infinitesimal flags centered respectively at $p_{1},\dots,p_{N}$ such that
\begin{equation*}
\Delta_{j}\Big(\mathcal{O}_{\mathbbm{P}^{2}}(1)\Big)=\bigg\{(x,y)\in\mathbbm{R}^{2}\, : \, 0\leq x\leq \epsilon \,\, \mathrm{and} \, \, 0\leq y\leq \frac{1}{N\epsilon}\Big(1-\frac{x}{\epsilon}\Big)\bigg\}=Conv\big(\vec{0},\epsilon\vec{e}_{1},\frac{1}{N\epsilon}\vec{e}_{2}\big)
\end{equation*}
where $\epsilon:=\epsilon_{S}(\mathcal{O}_{\mathbbm{P}^{2}}(1); N)$. In particular $\mu(L,\mathbbm{E})=\frac{1}{N\epsilon}$ and
%Then either the Nagata's conjecture holds for $N$ points and $\Delta_{j}(\mathcal{O}_{\mathbbm{P}^{2}(1)})=\frac{1}{\sqrt{N}}\Sigma_{2}$ for any $j=1,\dots,N$ or setting $\epsilon:=\epsilon_{S}(\mathcal{O}_{\mathbbm{P}^{2}}(1); N)$ we have
%\begin{multline*}
%F\Big(\Delta_{j}(\mathcal{O}_{\mathbbm{P}^{2}}(1))\Big)=\{(t,y)\in \mathbbm{R}^{2} \, : \, 0\leq t\leq \epsilon \, \mathrm{and}\, 0\leq y \leq t\}\,\cup\\
%\cup\{(t,y)\in\mathbbm{R}^{2} \, : \, \epsilon \leq t \leq \frac{1}{N\epsilon} \, \mathrm{and} \, \frac{m_{j}}{\frac{1}{N\epsilon}-\epsilon}(t-\epsilon)\leq y\leq \frac{m_{j}}{\frac{1}{N\epsilon}-\epsilon}(t-\epsilon)+\frac{\epsilon}{\frac{1}{N\epsilon}-\epsilon}\Big(\frac{1}{N\epsilon}-t\Big)\}
%\end{multline*}
%for some $0\leq m_{j} \leq \frac{1}{1-N\epsilon^{2}}$ where $F(x_{1},x_{2}):=(x_{1}+x_{2},x_{1})$. In particular for any $\epsilon \leq t \leq \frac{1}{N\epsilon}$ we have
$$
\mathrm{Vol}_{X\vert E_{j}}(f^{*}\mathcal{O}_{\mathbbm{P}^{2}}(1)-t\mathbbm{E}))= \, \begin{cases} t \quad\quad \mathrm{if} \quad 0\leq t\leq \epsilon \\  \frac{\epsilon}{\frac{1}{N\epsilon}-\epsilon}\Big(\frac{1}{N\epsilon}-t\Big) \quad \mathrm{if} \quad \epsilon\leq t\leq \frac{1}{N\epsilon} \end{cases}
$$
where $f:X=\mathrm{Bl}_{\{p_{1},\dots,p_{N}\}}\mathbbm{P}^{2}\to X$ is the blow-up at $Z=\{p_{1},\dots,p_{N}\}$, $E_{1},\dots,E_{N}$ the exceptional divisors and $\mathbbm{E}=\sum_{j=1}^{N}E_{j}$.
\end{thm}
\begin{proof}
If $\epsilon_{S}(\mathcal{O}_{\mathbbm{P}^{2}}(1);N)=1/\sqrt{N}$, i.e. maximal, then ${\Delta_{j}(\mathcal{O}_{\mathbbm{P}^{2}}(1))=\frac{1}{\sqrt{N}}\Sigma_{2}}$ as a consequence of Theorem \ref{ThmA} and Theorem \ref{ThmB}. Thus we may assume $\epsilon_{S}(\mathcal{O}_{\mathbbm{P}^{2}}(1);N)<1/{\sqrt{N}}$, and we know that there exists ${C=\gamma H-\sum_{j=1}^{N}m_{j}E_{j}}$ sub-maximal curve, i.e. an irreducible curve such that ${\epsilon_{S}(\mathcal{O}_{\mathbbm{P}^{2}}(1);N)=\frac{\gamma}{M}}$ where $M:=\sum_{j=1}^{N}m_{j}$. Moreover, since the points are very general, for any cycle $\sigma$ of lenght $N$ there exists a curve $C_{\sigma}=\gamma H-\sum_{j=1}^{N}m_{\sigma(j)}E_{j}$. This yields $\mu(f^{*}\mathcal{O}_{\mathbbm{P}^{2}}(1);\mathbbm{E})\geq \frac{M}{N\gamma}= \frac{1}{N\epsilon}$ since we can easily construct a section $s\in H^{0}\big(\mathbbm{P}^{2},\mathcal{O}_{\mathbbm{P}^{2}}(N\gamma)\big)$ such that $\mathrm{ord}_{p_{j}}(s)=M$ for any $j$. Recall that $\mu(f^{*}\mathcal{O}_{\mathbbm{P}^{2}}(1);\mathbbm{E})=\sup\{t\geq 0 \, : \, f^{*}\mathcal{O}_{\mathbbm{P}^{2}}(1)-t\mathbbm{E} \, \mathrm{is} \, \mathrm{big}\}$. Moreover for any $j=1,\dots,N$ we can fix holomorphic coordinates $(z_{1,j},z_{2,j})$ such that $\nu^{p_{j}}(s)=(0,M)$ with respect to the deglex order. So, considering an ample line bundle $A$ such that there exist sections $s_{1},\dots,s_{N}\in H^{0}(X,A)$ with $\nu^{p_{j}}(s_{j})=(0,0)$ and $\nu^{p_{i}}(s_{j})>0$ for any $i\neq j$, we get $s^{l}\otimes s_{j}^{N\gamma}\in V_{N\gamma,j}(lL+A)$, i.e. $(0,\frac{M}{N\gamma})=(0,\frac{1}{N\epsilon})\in \Delta_{j}(L+\frac{1}{l}A)$ by homogeneity (Proposition \ref{prop:Homo}), for any $l\in \mathbbm{N}$ and any $j=1,\dots,N$. Hence by Theorem \ref{thm:GlobalMOBJ} we deduce $(0,\frac{M}{N\gamma})\in \Delta_{j}(L)$ for any $j=1,\dots,N$.\\
Finally since by Theorem \ref{ThmB} we know that $\epsilon_{S}(\mathcal{O}_{\mathbbm{P}^{2}}(1);N)\Sigma_{2}\subset \Delta_{j}(L)$ for any $j=1\dots,N$, Theorem \ref{ThmA} and the convexity imply that the multipoint Okounkov bodies have necessarily the shape requested.
% Furthermore, since by Lemma \ref{lem:Exceptional} we know that $t_{j,+}=\mu(f^{*}L;\mathbbm{E})$ since the maximum \emph{inverted} symplex that fits into $F(\Delta_{j}(L))$ for any $j=1,\dots,N$ is $F(\epsilon\Sigma_{n})$ by Theorem \ref{ThmB}, an easy calculation using Theorem \ref{ThmA} and the convexity of $\Delta_{j}(L)$ implies that $\mu=\frac{1}{N\epsilon}$ and that $F\Big(\Delta_{j}(\mathcal{O}_{\mathbbm{P}^{2}}(1)\Big)$ has necessarily the shape requested.
\end{proof}
\begin{cor}
\label{cor:Chambers}
The ray $f^{*}\mathcal{O}_{\mathbbm{P}^{2}}(1)-t\mathbbm{E}$ meet at most two Zariski chambers.
\end{cor}
Corollary \ref{cor:Chambers} was already proved in Proposition 2.5. of \cite{DKMS15}.
\begin{rem}
\emph{We recall that Biran in \cite{Bir97} gave an homological criterion to check if a $4-$dimensional symplectic manifold admits a full symplectic packings by $N$ equal balls for large $N$, showing that $(\mathbbm{P}^{2},\omega_{FS})$ admits a full symplectic packings for $N\geq 9$. Moreover it is well-known that for any $N\leq 9$ the supremum over all $r$ such that $\{(B_{r}(0),\omega_{st})\}_{j=1}^{N}$ packs into $(\mathbbm{P}^{2},\mathcal{O}_{\mathbbm{P}^{2}}(1))$ coincides with the supremum over all $r$ such that $(\mathbbm{P}^{2},\omega_{FS})$ admits a symplectic packings of $N$ balls of radii $r$ (called \emph{Gromov width}). Therefore by Theorem \ref{ThmC} and Corollary \ref{cor:EcklGen} the Nagata's conjecture is true iff the Gromov width of $N$ balls on $(\mathbbm{P}^{2},\omega_{FS})$ coincides with the multipoint Seshadri constant of $\mathcal{O}_{\mathbbm{P}^{2}}(1)$ at $N$ very general points.}
\end{rem}

\end{document}